\documentclass{amsart}

\usepackage{latexsym,amsfonts,amssymb,exscale,enumerate,comment}
\usepackage{amsmath,amsthm,amscd}

\newcommand{\wt}{\tilde{w}}

\usepackage{url}
\usepackage[bookmarks=true,%
    colorlinks=true,%
    linkcolor=blue,%
    citecolor=blue,%
    filecolor=blue,%
    menucolor=blue,%
    urlcolor=blue,%
    breaklinks=true]{hyperref}

\input xy
\usepackage[all]{xy}
\xyoption{line}
\xyoption{arrow}
\xyoption{color}
\SelectTips{cm}{}

\usepackage{tikz}
\usetikzlibrary{shapes,snakes}
\usetikzlibrary{decorations.markings}
\usetikzlibrary{decorations.pathreplacing}
\tikzstyle directed=[postaction={decorate,decoration={markings,
    mark=at position #1 with {\arrow{>}}}}]

\newcommand{\hackcenter}[1]{
 \xy (0,0)*{#1}; \endxy}

\tikzset{->-/.style={decoration={
  markings,
  mark=at position #1 with {\arrow{>}}},postaction={decorate}}}

\tikzset{middlearrow/.style={
        decoration={markings,
            mark= at position 0.5 with {\arrow{#1}} ,
        },
        postaction={decorate}
    }
}

\usepackage{graphicx}
\usepackage{color}

\def\P{\mathsf{P}}
\def\Q{\mathsf{Q}}

\newcommand{\SL}{\mathrm{SL}}

\def\id{\mathrm{id}}
\def\Id{\mathrm{Id}}

\theoremstyle{plain}
\newtheorem{theorem}{Theorem}

\newtheorem{corollary}[theorem]{Corollary}
\newtheorem{proposition}[theorem]{Proposition}
\newtheorem{lemma}[theorem]{Lemma}

\theoremstyle{definition}
\newtheorem{example}[theorem]{Example}
\newtheorem{definition}[theorem]{Definition}

\theoremstyle{definition}
\newtheorem{remark}[theorem]{Remark}

\numberwithin{equation}{section}
\numberwithin{theorem}{section}

\newcommand{\sym}{{\rm Sym}}
\newcommand{\maps}{\colon}

\newcommand{\refequal}[1]{\xy {\ar@{=}^{#1}
(-1,0)*{};(1,0)*{}};
\endxy}

\newcommand{\Hom}{{\rm Hom}}

\renewcommand{\to}{\rightarrow}

\def\Ind{{\mathrm{Ind}}}

\def\fmod{{\mathrm{-fmod}}}   

\def\Id{\mathrm{Id}}

\def\mf{\mathfrak}

\numberwithin{equation}{section}
\def\TL#1{\textcolor[rgb]{1.00,1.00,1.00}{[TL: #1]}}%
\def\AL#1{\textcolor[rgb]{1.00,0.00,0.00}{[AL: #1]}}%
\def\PS#1{\textcolor[rgb]{0.00,0.49,0.25}{[PS: #1]}}%
\def\JS#1{\textcolor[rgb]{0.40,0.00,0.90}{[JS: #1]}}%
\let\hat=\widehat
\let\tilde=\widetilde

\let\epsilon=\varepsilon

\usepackage{bbm}
\def\C{{\mathbb{C}}}
\def\N{{\mathbbm N}}
\def\R{{\mathbbm R}}
\def\Z{{\mathbbm Z}}
\def\H{{\mathcal{H}}}

\def\cal#1{\mathcal{#1}}%
\def\1{\mathbbm{1}}%
\def\nn{\notag}

\def\la{\langle}
\def\ra{\rangle}

\usepackage{bbm}

\def\cal#1{\mathcal{#1}}

\def \k {\mathbbm{C}}
\def \Z {\mathbbm{Z}}

\def \N {\mathbbm{N}}
\def \E {\mathcal{E}}

\def \Tr{\operatorname{Tr}}

\def \Span{\operatorname{Span}}
\def \Ob{\operatorname{Ob}}

\def \HH{\operatorname{HH}}

\def \Id {{\rm Id}}

\def\k{\C}

\def\bH{\mathbb{H}}
\def\bP{\mathbb{P}}
\def\bQ{\mathbb{Q}}

\newcommand\nc{\newcommand}
\nc\rnc{\renewcommand}
\nc\Kar{\operatorname{Kar}}
\nc\End{\operatorname{End}}
\nc\modQ {{\mathbb Q}}
\nc\modZ {{\mathbb Z}}
\nc\simeqto{\overset{\simeq}{\longrightarrow }}
\nc\K{\mathcal {K}}
\nc\CC{\mathbf{C}}

\nc\tE{\check{E}}
\nc\qh{\mathcal{H}}
\nc\hbm{\mathcal{B}}
\nc\bu{\mathbf{u}}
\nc\bk{\mathbf{\kappa}}
\nc\bZ{\mathbf{Z}}
\nc\theirs{\mathrm{theirs}}
\nc\ours{\mathrm{ours}}
\nc\hE{\mathcal{\hat E}}
\nc\bK{\mathbf{K}}
\nc\bw{\mathbf{w}}
\nc\ba{\mathbf{a}}
\nc\bb{\mathbf{b}}

\nc{\urcap}{\textrm{uRCap}}
\nc{\urcup}{\textrm{uRCup}}
\nc{\ulcap}{\textrm{uLCap}}
\nc{\ulcup}{\textrm{uLCup}}
\nc{\drcap}{\textrm{dRCap}}
\nc{\drcup}{\textrm{dRCup}}
\nc{\dlcap}{\textrm{dRCap}}
\nc{\dlcup}{\textrm{dRCup}}
\newcommand{\dah}{{\rm DH}}
\newcommand{\ah}{{\rm AH}}
\nc{\cl}{\mathrm{cl}}
\nc{\cala}{\mathcal{A}}
\nc{\calb}{\mathcal{B}}
\nc{\calc}{\mathcal{C}}
\nc{\bx}{\mathbf{x}}
\nc{\by}{\mathbf{y}}
\nc{\bE}{\mathbbm{E}}
\nc{\bElong}{{\mathbbm E^{\bullet, \geq}}}

\nc{\Sk}{\mathrm{Sk}}
\nc{\Hilb}{\mathrm{Hilb}}

\newcommand{\scs}{\scriptstyle}

\nc\col{\colon\thinspace}

\allowdisplaybreaks

\title{
The elliptic Hall algebra and the deformed Khovanov Heisenberg category
}

\begin{document}
\setcounter{tocdepth}{2}

\author{Sabin Cautis}
\email{cautis@math.ubc.ca}
\address{Department of Mathematics\\ University of British Columbia \\ Vancouver, Canada}

\author{Aaron D. Lauda}
\email{lauda@usc.edu}
\address{Department of Mathematics\\ University of Southern California \\ Los Angeles, CA}

\author{Anthony M. Licata}
\email{anthony.licata@anu.edu.au}
\address{Mathematical Sciences Institute\\ Australian National University \\ Canberra, Australia}

\author{Peter Samuelson}
\email{peter.samuelson@ed.ac.uk}
\address{Department of Mathematics\\ University of Edinburgh\\ Edinburgh, UK}

\author{Joshua Sussan}
\email{jsussan@mec.cuny.edu}
\address{Department of Mathematics \\ CUNY Medgar Evers \\ Brooklyn, NY}
\date{\today}

\maketitle

\begin{abstract}
We give an explicit description of the trace, or Hochschild homology, of the quantum Heisenberg category defined in \cite{LS13}. We also show that as an algebra, it is isomorphic to ``half'' of a central extension of the elliptic Hall algebra of Burban and Schiffmann \cite{BS12}, specialized at $\sigma = \bar\sigma^{-1} = q$. A key step in the proof may be of independent interest: we show that the sum (over $n$) of the Hochschild homologies of the positive affine Hecke algebras $\ah_n^+$ is again an algebra, and that this algebra injects into both the elliptic Hall algebra and the trace of the $q$-Heisenberg category. Finally, we show that a natural action of the trace algebra on the space of symmetric functions agrees with the specialization of an action constructed by Schiffmann and Vasserot using Hilbert schemes.
\end{abstract}

\tableofcontents

\section{Introduction}
Given a $\mathbbm k$-linear monoidal category $\mathcal C$, there are two ways to decategorify it to obtain an algebra. One may take the Grothendieck group $K_0(\mathcal C)$, or instead take the zeroth Hochschild homology, or trace $\Tr(\mathcal C)$, which is defined as
\[
\Tr(\mathcal C) := \frac{\bigoplus_{M \in Ob(\mathcal C)} \End_{\mathcal C}(M)} {\{f\circ g-g\circ  f \mid f:X \to Y,\, g:Y \to X\}}.
\]
The multiplication on $\Tr(\mathcal C)$ is given by $[f]\cdot [g] := [f\otimes g]$. There is a natural \emph{Chern character} map $K_0(\mathcal C) \to \Tr(\mathcal C)$ which sends an object to the class of its identity morphism. In many cases the Chern character is an isomorphism, but recently several examples have been given where $\Tr(\mathcal C)$ is a much larger, more interesting algebra than $K_0(\mathcal C)$
\cite{BGHL,BHLW,BHLZ,EL14,CLLS15}.

In this paper we study the trace of the quantum Heisenberg category $\H$, which is a monoidal category coming from higher representation theory \cite{Kh-Heis, LS13}.  We give an explicit presentation of the trace $\Tr(\H)$, and we show that it is isomorphic to the $\sigma=q=\bar\sigma^{-1}$ specialization of ``half'' of a central extension of the elliptic Hall algebra $\E$. This algebra was defined by Burban and Schiffmann \cite{BS12}, who showed it is the `universal' Hall algebra of elliptic curves over finite fields. We construct this isomorphism by relating a triangular piece of $\Tr(\H)$ to the HOMFLYPT skein algebra $\Sk_q(T^2)$ of the torus. This algebra comes from low dimensional topology and was related to the same specialization of $\E$ in \cite{MS14}. Below we briefly describe these algebras, give a precise statement of the results, and provide a rough outline of the proof.

\subsection{The Heisenberg category}

In \cite{Kh-Heis}, Khovanov introduced a monoidal category $\H_{q=1}$ whose Grothendieck group contains (and is conjecturally isomorphic to) the Heisenberg algebra. In this category, objects can be thought of as compositions of induction and restriction functors for the symmetric groups $S_n$, and morphisms are generated by certain natural transformations between these functors. These transformations satisfy certain relations which are independent of $n$, and it is natural to define this category diagrammatically. In this definition, the objects are sequences of upward or downward pointing arrows on a line in the plane, and the morphisms are certain oriented diagrams connecting these arrows, modulo relations defined locally on diagrams.

The trace of the category $\H_{q=1}$ was described explicitly in \cite{CLLS15}, where it was shown to be isomorphic to an algebra referred to as $W_{1+\infty}$.
It was shown in \cite{SV13b} that $W_{1+\infty}$ is a specialization of a central extension of a limit of degenerate double affine Hecke algebras.  The centrally extended elliptic Hall algebra was related to a limit of double affine Hecke algebras in \cite{SV13}.
This gives an indication that a deformation of the category $\H_{q=1}$ may be related to the elliptic Hall algebra itself.

A deformation $\H$ of Khovanov's Heisenberg category was given by the third author and Savage in \cite{LS13}. In this construction the symmetric groups $S_n$ are replaced by Hecke algebras $H_n$, and the resulting diagrammatic relations depend on a parameter $q$. It was shown in \cite{LS13} that the Heisenberg algebra also injects into $K_0(\H)$. This injection is conjectured to be an isomorphism, and if it
is, then the Grothendieck group of $\H$ is not sensitive to $q$ (at
least for generic $q$). However, as we will see, the trace $\Tr(\H)$ is a much larger algebra than what the algebra $K_0(\H)$ is conjectured to be, and the relations in $\Tr(\H)$ depend on $q$ in an essential way.

We prove the following theorem
(which is combination of Theorem \ref{thm_mainthm} and Proposition \ref{prop_tridecomp}):
\begin{theorem}\label{thm_intropres}
The algebra $\Tr(\H)$ has a presentation with generators $w_{a,b}$ for $a \in \Z$ and $b \in \N$, subject to the relations
\begin{equation*}
[w_\bx,w_\by] = \left\{\begin{array}{cl}
-k & \textrm{if } \bx = (k,0) = -\by\\
-\left(q^{d/2}-q^{-d/2}\right) w_{\bx+\by} & \textrm{otherwise}\end{array}\right.
\end{equation*}
where $d = \mathrm{det}(\bx\,\by)$ and $\bx,\by \in \Z \times \Z_{\geq 0}$.
\end{theorem}
This theorem gives a generators and relation description of $\Tr(\H)$, but this description is only valid when $q-1$ is invertible. In particular, the generators $w_{a,b}$ in the previous theorem have poles at $q=1$, so this presentation does not reproduce the description of $\Tr(\H_{q=1})$ from \cite{CLLS15}.

\subsection{The elliptic Hall algebra}
An abelian category $\mathcal A$ is called \emph{finitary} if $\Hom(M,N)$ and $\mathrm{Ext}^1(M,N)$ are finite sets, for all objects $M$ and $N$. Examples of such categories are $\mathrm{Coh}(X)$, where $X$ is a smooth variety over a finite field, or representations of quivers over finite fields. If $\mathcal A$ is finitary, then its \emph{Hall algebra} is the space formally spanned (over a field $\Bbbk$) by isomorphism classes $[M]$ of objects in $\mathcal A$, with multiplication given by
\begin{align*}
F_{M,N}^E &:= \# \left\{E' \subset E \, \mid E' \simeq M,\,\, E / E' \simeq N\right\}\\
[M]\cdot [N]&:= \sum_{[E]} F_{M,N}^E [E]
\end{align*}
(The finitary condition implies both the structure constants and the sum are finite.) One of the early appearances of Hall algebras in representation theory is due to Ringel, who showed that the Hall algebra of quiver representations (over $\mathbb F_q$) for a Dynkin quiver is isomorphic to the positive part of the quantum group with that type \cite{Rin90}.

Given a smooth elliptic curve $X$ over a finite field $\mathbb F_{p^k}$, let $\E^+_X$ be the Hall algebra of the category of coherent sheaves over $X$. In \cite{BS12}, Burban and Schiffmann gave an explicit definition of the \emph{elliptic Hall algebra}  $\E_{q,t}$ using generators and relations. They showed that if $\sigma(X)$ and $\bar \sigma(X)$ are the Frobenius eigenvalues of $X$, then the specialization $\E_{q=\sigma(X),t=\bar\sigma(X)^{-1}}$ is isomorphic to the Drinfeld double $\E_X$ of $\E^+_X$. They also constructed $\SL_2(\Z)$ actions on $\E_{q,t}$ and $\E_X$ and showed their isomorphism is $\SL_2(\Z)$-equivariant.

It turns out that the algebra $\E_{q,t}$ (or one of its cousins\footnote{By `cousin' we mean either the `positive half' $\E^+_{q,t}$ or a central extension of $\E_{q,t}$.}) has been realized in several contexts:
\begin{itemize}
\item a generalized quantum affine algebra in \cite{DI97},
\item `a $(q,\gamma)$ analog of the $W_{1+\infty}$ algebra' in \cite{Mik07},
\item the `shuffle algebra' of \cite{FT11} (see also \cite{Neg14}),
\item the `spherical $\mathrm{gl}_\infty$ double affine Hecke algebra' in \cite{SV11} (see also \cite{FFJMM11a}),
\item the `quantum continuous $\mathrm{gl}_\infty$' in \cite{FFJMM11a},
\item an algebra acting on $\oplus_n K^T(\mathrm{Hilb}_n(\C^2))$ in \cite{SV13} (see also \cite{FT11}, \cite{FFJMM11a}, \cite{Neg15}),
\item the `skein algebra of the torus,' (when $q=t$, see \cite{MS14}).
\end{itemize}

We describe the last two items in more detail since they will be most relevant for this paper. First, there are natural operators on the equivariant $K$-theory $\sym := \oplus_n K^{\C^*\times \C^*}(\mathrm{Hilb}_n(\C^2))$ of the Hilbert schemes which come from convolution on certain flag Hilbert schemes. Schiffmann and Vasserot used these operators to construct an action of $\hE_{q,t}$ on $\sym$, where $\hE_{q,t}$ is a central extension of the elliptic Hall algebra $\E_{q,t}$. The extension is by $\k[\bk_1^{\pm 1},\bk_2^{\pm 1}]$, and to recover $\E_{q,t}$ from the extension $\hE_{q,t}$, one specializes $\bk_1=\bk_2=1$. In Section \ref{sec:elliptic} we recall an explicit description of $\hE_{q,t}$ from \cite{SV13}.

Second, the HOMFLYPT skein algebra $\Sk_q(T^2)$ of the torus is the space formally spanned by closed, framed, oriented links in $T^2\times [0,1]$ modulo the local skein relations in Equation \eqref{eq:homfly}. The algebra structure is given by stacking links on top of each other (in the $[0,1]$ direction). This algebra was described explicitly in \cite{MS14}, and it was shown to be isomorphic to the $t=q$ specialization of $\E_{q,t}$ (without the central extension). Under this isomorphism, the natural action of $\SL_2(\Z)$ on $\Sk_q(T^2)$ agrees with the Burban-Schiffmann action on $\E_{q,t}$.

\subsection{The comparison}
We now describe the comparison between the trace of the quantum Heisenberg category and the elliptic Hall algebra. Let $\mathbb E$ be the algebra obtained from $\hE_{q,t}$ by first specializing $\bk_1=1$ and $\bk_2 = (qt^{-1})^{-1/2}$, and then specializing $t=q$.
The description of $\hE_{q,t}$ by Schiffmann and Vasserot provides certain generators $\bu_{a,b}$ for $a,b \in \Z^2$, and we let
$\bElong$ be the subalgebra generated by $\bu_{a,b}$ with $b \geq 0$.

Our main theorem is the following (see Theorem \ref{thm_mainthm}).
\begin{theorem}\label{thm:mainintro}
The algebras $\Tr(\H)$ and $\bElong$ are isomorphic.
\end{theorem}

We briefly outline the proof, which has three main steps. First, we note in Section \ref{sec:towers} that the sum
\[
A := \bigoplus_n \HH_0(\ah^+_n)
\]
of the Hochschild homologies of the (positive) affine Hecke algebras has a natural algebra structure. By a lemma in \cite{LS13}, the algebras $\ah^+_n$ appear in certain endomorphism spaces in $\H$, and this provides an algebra map $A \to \Tr(\H)$. On the other hand, by interpreting the affine Hecke algebra as the affine braid group modulo the HOMFLYPT skein relations, we construct a natural algebra map $A \to \Sk_q(T^2)$. We show in Theorem \ref{thm_ahatoskein} that this map is injective, and by \cite{MS14}, this gives an algebra map $A \hookrightarrow \bElong$. We define $\bE^+$ to be the image of this map, and we write $\varphi^+: \bE^+ \to \Tr(\H)$ for the induced map.

Second, we establish a triangular decomposition of $\bElong = \bE^+ \otimes_\C \bE^0 \otimes_\C \bE^{-}$, where each triangular piece is a subalgebra. In Proposition \ref{prop_tridecomp} we provide sufficient cross-relations between the triangular pieces to ensure this is an isomorphism of algebras.

Finally, we use a triangular decomposition of $\Tr(\H)$ (see Lemma \ref{HHVSLEMMA}) to show that the map $\varphi^+$ is injective, and to extend it to a linear isomorphism $\varphi:\bElong \to \Tr(\H)$. To finish the proof of Theorem \ref{thm:mainintro}, all that is left is to check that the cross-relations describing $\bElong$ hold in $\Tr(\H)$. We do this by explicit diagrammatic computation in Section \ref{sec_crossrelations}. We remark that in these computations there are some substantial cancellations (see, e.g.~Lemma \ref{lemma_crazyidentity} or Corollary \ref{cor:prob6}) -- in the end, these cancellations are due to a very careful choice of generators by Burban and Schiffmann.

\subsection{Representations}
After establishing this isomorphism, it is natural to ask what kinds of representations of $\Tr(\H)$ can be constructed using the category $\H$, and how these representations compare to previously constructed representations of the Hall algebra $\E_{q,t}$. Since the definition of $\H$ is based on a tower of (finite) Hecke algebras $H_n$ it naturally acts via induction and restriction functors on $\oplus_{n \ge 0} H_n\fmod$, the sum of the categories of finitely generated $H_n$-modules. This action categorifies the Bosonic Fock space representation of the Heisenberg algebra (see \cite{LS13}).

This categorical action of $\H$ induces an action of $\Tr(\H)$ on $\oplus_{n \ge 0} \Tr(H_n\fmod)$. We identify the latter with the ring $\sym$ of symmetric functions. In Section~\ref{sec:symaction} we characterize this action by explicitly describing the action of certain generators of $\Tr(\H)$ on $\sym$.

In \cite{SV13} Schiffmann and Vasserot constructed an action of $\hE_{\sigma,\bar \sigma}$ on $\sym$ using certain convolution operators on the equivariant $K$-theory of the Hilbert schemes $K^{\C^\times \times \C^\times}(\Hilb_n(\C^2))$. We also show that the $\sigma = \bar\sigma^{-1}$ specialization of the Schiffmann-Vasserot action on $\sym$ agrees with a twist of the above action of $\Tr(\H)$ on $\sym$.

We now briefly summarize the contents of the paper. In Section \ref{sec:traces} we discuss traces of categories. In Section \ref{sec:heckealgebras} we discuss various versions of Hecke algebras that we use, as well as their relations to the skein algebra of the torus. The elliptic Hall algebra (and the triangular decomposition we need) is discussed in Section \ref{sec:elliptic}. The $q$-Heisenberg category is defined in Section \ref{sec:qheis}, which also contains basic results, such as its triangular decomposition. In Section \ref{sec:mainthm} we prove the main theorem relating the trace of the quantum Heisenberg category with the elliptic Hall algebra. In Section \ref{sec:symaction} we compare the action of $\Tr(\H)$ and $\E_{q,t=q}$ on the ring $\sym$ of symmetric functions.

\subsection*{Acknowledgements} The authors are grateful to B. Elias, F. Goodman, S. Morrison, O. Schiffmann and D. Tubbenhauer for helpful conversations. S.C. was supported by an NSERC discovery/accelerator grant. A.D.L. was partially supported by NSF grant DMS-1255334 and by the Simons Foundation. A.M.L. was supported by an Australian Research Council Discovery Early Career fellowship. J.S. was supported by NSF grant DMS-1407394, PSC-CUNY Award 67144-0045, and an Alfred P. Sloan Foundation CUNY Junior Faculty Award.

%
\section{Traces and the Chern character map}\label{sec:traces}
%

%
\subsection{The trace decategorification}
%

The trace, or zeroth Hochschild homology, $\Tr(\mathcal{C} )$ of a $\Bbbk$-linear category $\mathcal{C}$ is the $\Bbbk$-vector space given by
\begin{gather*}
  \Tr(\cal{C} )=
\left( \bigoplus_{X\in \Ob(\mathcal{C} )}\mathcal{C} (X,X) \right)\bigg/ \mathcal{I},
\end{gather*}
where $\mathcal{C}(X,X)=\End_\cal{C}(X)$ and $\mathcal{I}= \Span_\Bbbk\{fg-gf\}$ where $f$ and $g$ run through all pairs of morphisms $f\col X\to Y$, $g\col Y\to X$ with $X,Y\in \Ob(\mathcal{C} )$.  For a morphism $f$ in $\cal{C}$ we denote its class in $ \Tr(\cal{C}) $ by $ [f]$.

If $\cal{C}$ is an additive category, then by ~\cite[Lemma 3.1]{BHLZ} we have
\begin{equation}
  [f\oplus g] = [f] + [g].
\end{equation}
The trace map is invariant under passage to the additive closure.  Similarly, if $\cal{C}$ is a monoidal category then $\Tr(\cal{C})$ acquires an algebra structure with
\begin{equation}
  [f] \cdot [g] := [f \otimes g] .
\end{equation}
When $\cal{C}$ is graded, so that it is equipped with an autoequivalence $\la 1\ra \maps \cal{C} \to \cal{C}$, then $\Tr(\cal{C})$ has the structure of $\Z[t,t^{-1}]$-module by defining
\begin{equation}
  [f\la s\ra ] := t^s[f],
\end{equation}
where, for $s \in \Z$, $\la s \ra$ denotes $\la 1\ra$ (or its inverse) applied $s$ times.

Taking the trace is a functor from small $\Bbbk$-linear categories to the category of $\Bbbk$-vector spaces. Moreover, if $F\col\mathcal{C}\to\mathcal{D}$ is a $\Bbbk$-linear functor, then $F$ induces a linear map
\begin{gather*}
  \Tr(F)\col\Tr(\mathcal{C})\to\Tr(\mathcal{D})
\end{gather*}
given by
$\Tr(F)([f]) = [F(f)]$
for endomorphisms $f\col x\to x$ in $\mathcal{C}$. For an introduction to the trace decategorification functor see~\cite{BGHL}.


%
\subsection{Split Grothendieck groups and the Chern character}
\label{sec:split-groth-groups}
%

For a $\Bbbk$-linear additive category $\mathcal{C} $, the {\em split
Grothendieck group} $K_0(\mathcal{C} )$ of $\mathcal{C} $ is the abelian group
generated by the isomorphism classes of objects of $\mathcal{C} $ with
relations $[x\oplus y]_{\cong}=[x]_{\cong}+[y]_{\cong}$ for
$x,y\in\Ob(\mathcal{C})$.  Here $[x]_{\cong}$ denotes the isomorphism class
of~$x$.  For $\cal{C}$ monoidal the split Grothendieck group is a ring with
$[x]_{\cong} \cdot [y]_{\cong}:= [x\otimes y]_{\cong}$.

For a $\Bbbk$-linear additive category $\mathcal{C}$, the {\em Chern character}
for $\cal{C}$ is the $\Bbbk$-linear map
\begin{gather*}
  h_\mathcal{C} \col K^\Bbbk_0(\cal{C}):=K_0(\mathcal{C} )\otimes\Bbbk\longrightarrow \Tr(\mathcal{C} )
\end{gather*}
defined by $h_\mathcal{C} ([x]_{\cong})=[1_x]$ for $x\in \Ob(\mathcal{C} )$.
(Although $h_\mathcal{C}$ can be defined on $K_0(\mathcal{C})$, we consider
only the above $\Bbbk$-linear version for simplicity.)  If $\cal{C}$ is monoidal then the Chern character map is an algebra homomorphism.    

Recall also that the trace is well behaved with respect to passing to the Karoubi envelope.

\begin{proposition}
\label{HH0HequalsHH0Hprime}
\cite[Proposition 3.2]{BHLZ}
The natural map $ \Tr(\mathcal{C}) \rightarrow \Tr(Kar(\mathcal{C}))$ induced by inclusion of categories is an isomorphism.
\end{proposition}

\begin{proposition}\cite[Lemma 2.1]{BHLW}
  \label{prop:indecomposables}
  Let $\cal{C}$ be a $\Bbbk$-linear additive category.
  Let $S\subset\Ob(\cal{C})$ be a subset such that every object in $\cal{C}$ is
  isomorphic to the direct sum of finitely many copies of objects in
  $S$.  Let $\cal{C}|_S$ denote the full subcategory of $\cal{C}$ with
  $\Ob(\cal{C}|_S)=S$.  Then, the inclusion functor $\cal{C}|_S\to \cal{C}$ induces
  an isomorphism
  \begin{gather}
    \label{e9}
    \Tr(\mathcal{C}|_S)\cong\Tr(\mathcal{C}).
  \end{gather}
\end{proposition}

%
\section{Affine Hecke algebras and the skein of the torus}\label{sec:heckealgebras}
%

In this section we will relate the trace algebra of the tower of affine Hecke algebras to the skein algebra of the torus.

%
\subsection{Towers of algebras}
\label{sec:towers}
%

\begin{definition}
Let $\{A_n\}_{n\geq 1}$ be a collection of unital algebras over a field $\Bbbk$ with unit element denoted $1_n$.  We set $A_{0}:=\Bbbk$.  This collection defines a \emph{tower of algebras} if we are given unital algebra homomorphisms
\begin{equation}
 \mu_{m,n}: A_m \otimes_\Bbbk A_n \to A_{m+n}
\end{equation}
that satisfy the associativity conditions
\begin{equation}\label{eq_ass}
 \mu_{\ell+m,n}\circ (\mu_{\ell,m}\otimes \id_n) = \mu_{\ell,m+n}\circ(\id_\ell\otimes \mu_{m,n}).
\end{equation}
\end{definition}

A tower of algebras $\{A_n\}$ naturally defines a $\Bbbk$-linear monoidal category $\cal{A}$ whose objects are positive integers $n \in \Z_{\geq 0}$ and whose only nonzero morphisms are the endomorphism spaces $\Hom_{\cal{A}}(n,n) := A_n$.  The monoidal structure is given on objects by $n\otimes m := n+m$ and on morphisms via the maps $\mu_{m,n}$.  The unit object for this monoidal category is the object $0$, which we regard as indexing $A_0=\Bbbk$.

By definition, the trace $\Tr(\cal{A})$ is equal the $\Bbbk$-vector space
$$\bigoplus_{n\in \Z_{\geq 0}} \HH_0(A_n)$$
equipped with a graded algebra structure via $[a] \cdot [b] := [\mu_{m,n}(a\otimes b)]$ for $a \in A_m$ and $b \in A_n$. Here $\HH_0(A_n)$ denotes the zeroth Hochschild homology, which is the quotient of the algebra $A_n$ by the $\Bbbk$-linear subspace spanned by elements of the form $ab-ba$, for $a,b \in A_n$.

\begin{remark}
One should not confuse the category $\cal{A}$ defined above with potentially richer categories where the objects are (certain subcategories of) $(A_n,A_m)$-bimodules and the morphisms are natural transformations between them.
\end{remark}

%
\subsection{Hecke algebras}
%

\begin{definition} \label{def:finite-hecke}
The (finite) Hecke algebra $H_n$ is the $\k[q,q^{-1}]$-algebra generated by $t_1,t_2,\dots,t_{n-1}$ and relations
\begin{equation} \label{eq:fHecke}
\begin{split}
  t_i^2 &= (q-1)t_i + q, \\
  t_i t_j &= t_j t_i \qquad \text{if $|i-j|>1$,} \\
  t_it_{j}t_i &= t_jt_it_j \qquad \text{if $|i-j|=1$. }
\end{split}
\end{equation}
\end{definition}

\begin{definition} \hfill \label{def:affine-hecke}
\begin{enumerate}
\item The \emph{affine Hecke algebra} $\ah_n$ is the $\k[q,q^{-1}]$-algebra  with generators
\[
\{t_1,t_2, \dots, t_{n-1} \} \cup \{ x_1^{\pm1},x_2^{\pm1}, \dots, x_n^{\pm1}\}
\]
satisfying the following relations:
\begin{equation}\label{eq:affineheckerel}
\begin{aligned}[l]
  &t_i^2 = (q-1)t_i + q, \\
  &t_i t_j = t_j t_i \qquad \text{if $|i-j|>1$,} \\
  &t_it_{j}t_i = t_jt_it_j \quad \text{if $|i-j|=1$, }
\end{aligned}
\qquad
\begin{aligned}[l]
   &x_i x_j = x_j x_i,  \\
    &t_i x_{i+1}t_i = qx_{i},
  \end{aligned}
\qquad
\begin{aligned}[l]
  &x_ix_i^{-1}=1; \\
  &t_ix_j = x_jt_i \qquad \text{if $j \neq i,i+1$.}
\end{aligned}
\end{equation}
\item The \emph{positive affine Hecke algebra $\ah_{n}^+$} is the $\k[q,q^{-1}]$-subalgebra of $\ah_n$ generated by the $t_i$ and the $x_j$.
\end{enumerate}
\end{definition}

\begin{remark}
Our definition of \eqref{def:affine-hecke} differs from the usual definition, see for example \cite[Section 3.5]{LS13}, by the isomorphism sending $x_i \mapsto x_{n+1-i}$ and $t_i \mapsto t_{n-i}$.
\end{remark}

The algebra $\ah_n$ is $\Z$-graded, and $\ah_n^+$ is $\Z^+$-graded subalgebra of $\ah_n$, with $\deg(x_j^{\pm1}) = \pm 1$, and $\deg(t_i)=0$.

\begin{definition} \label{def:qdegen}
The \emph{$q$-degenerate affine Hecke algebra} $\dah_n$ is the $\k[q,q^{-1}]$-algebra  with generators
\[
\{t_1,t_2, \dots, t_{n-1} \} \cup \{ y_1,y_2, \dots, y_n\}
\]
and defining relations:
\begin{equation}
\begin{aligned}[c]
  t_i^2 &= (q-1)t_i + q, \\
  t_i t_j &= t_j t_i \qquad \text{if $|i-j|>1$,} \\
  t_it_{j}t_i &= t_jt_it_j \qquad \text{if $|i-j|=1$, }
\end{aligned}
\qquad \qquad
\begin{aligned}[c]
  y_it_k &= t_k y_i, \qquad i\neq k,k+1, \\
y_it_i &=t_iy_{i+1} +(q-1)y_{i+1} +q, \\
  t_i y_i&= y_{i+1}t_i +(q-1)y_{i+1}+q.
\end{aligned}
\end{equation}
\end{definition}

The algebra $\dah_n$ is a $\Z^+$-filtered with with $\deg(y_j ) =  1$, and $\deg(t_i)=0$.
If we set $q=1$ in Definition~\ref{def:qdegen} we get the definition of the \emph{degenerate affine Hecke algebra} $\dah_n(q=1)$.

\begin{lemma}[Lemma 3.8 \cite{LS13}]  \label{lemma_filtgrad}
There is an isomorphism of filtered algebras
\[
 \dah_n \otimes_{\k[q,q^{-1}]}\k(q) \longrightarrow \ah_n^+ \otimes_{\k[q,q^{-1}]}\k(q).
\]
In particular,  the algebra $\dah_n$ has a graded presentation over $\k(q)$.
\end{lemma}

\begin{proof}
It is a straight-forward verification after setting
\begin{equation}\label{eq_ydef}
y_i := (q-1)x_i - \frac{q}{q-1}.
\end{equation}
\end{proof}

\begin{remark} \hfill
\begin{enumerate}[(i)]
  \item The isomorphism from Lemma~\ref{lemma_filtgrad} remains true if we restrict $\k(q)$  to the ring $\k[q,q^{-1},(q-1)^{-1}]$.

  \item Even though both $\dah_n$ and $\ah_n^+$  are defined (and flat) over $\k[q,q^{-1}]$, their $q\to 1$ limits are \emph{not} isomorphic as algebras.
\end{enumerate}

\end{remark}

\begin{lemma}\label{lemma_flat}
The algebras $\ah_n$, $\ah^+_n$, and $\dah_n$ are each flat over $\k[q,q^{-1}]$.
\end{lemma}
\begin{proof}
In each case, we use Bergman's Diamond Lemma to show that the algebra in question is free (and flat, in particular) over the base ring $\k[q,q^{-1}]$. For example, in the case of the affine Hecke algebra $\ah_n$, the Diamond Lemma shows that the set $\{x_1^{j_1}\cdots x_n^{j_n}T_w\,\mid w\in S_n\}$ is a basis of $\ah_n$ over $\k[q,q^{-1}]$ (where the $T_w$ are the natural basis for the Hecke algebra).
\end{proof}

%
\subsection{Traces of Hecke algebras}
%

\begin{definition}\label{def_ahainc}
There are inclusions
\[
\mu_{m,n}: \ah_m \otimes \ah_n \hookrightarrow \ah_{m+n}
\]
given by
\begin{equation}\label{eq_ahainc}
t_i\otimes 1 \mapsto t_i,\quad x_j\otimes 1 \mapsto x_j,\quad 1\otimes t_i \mapsto t_{i+m},\quad 1\otimes x_j \mapsto x_{j+m}
\end{equation}
\end{definition}
These inclusions obviously satisfy the associativity condition \eqref{eq_ass} and also restrict to the subalgebras $\ah^+_n$ and $\dah_n$. We denote the corresponding monoidal categories (see section \ref{sec:towers}) by $\ah,\ah^+$ and $\dah$. Subsequently we have three trace algebras
\[
\Tr(\ah) := \bigoplus_{n\geq 1} \HH_0(\ah_n),\quad \quad
\Tr(\ah^+) := \bigoplus_{n\geq 1} \HH_0(\ah_n^+),\quad \quad
\Tr(\dah) := \bigoplus_{n\geq 1} \HH_0(\dah_n).
\]

\begin{remark}
The $\k(q)$-isomorphisms $\dah_n \cong \ah_n^+$ induce a $\k(q)$-isomorphism $\Tr(\dah) \cong \Tr(\ah^+)$, but this isomorphism does not extend to $q=1$.
\end{remark}

\begin{example}\label{example:trrel}
Here is an example of a relation in $\Tr(\ah)$ that occurs because of the trace condition.
We first note that if we view $x_1$ as an element of $\ah_n$, then $[x_1]\cdot [1_m]\in \Tr(\ah_{n+m})$ is the same as the class $[x_1 ]$, but $[1_m]\cdot [x_1] \in \Tr(\ah_{n+m})$ is equal to the class $[x_{m+1}]$ because of the inclusions defined by \eqref{eq_ahainc}.

We then compute in $\Tr(\ah_{n+1})$
\[
[x_1]\cdot [1_1]= [x_1] = q^{-1}[t_1 x_2 t_1] \equiv q^{-1}[t^2_1 x_2] = q^{-1}[((q-1) t_1 + q) x_2] = \frac{(q-1)}{q}[t_1 x_2] +[1_1]\cdot[x_1].
\]
We can therefore rewrite this relation as a commutator relation
\[
 [x_1]\cdot [1_1] - [1_1]\cdot[x_1] = \frac{(q-1)}{q}[t_1x_2].
\]
This is a basic example of the commutator relations that occur in $\Tr(\ah^+)$. A skein-theoretic interpretation of this computation is in Example \ref{example:diagrelation}.
\end{example}

\begin{lemma}
The algebra $\Tr(\ah)$ is $\Z_{\geq 1} \times \Z$-graded, $\Tr(\ah^+)$ is $\Z_{\geq 1} \times \Z_{\geq 0}$-graded, and $\Tr(\dah)$ is $\Z_{\geq 1}$ graded and $\Z_{\geq 0}$ filtered.
\end{lemma}
\begin{proof}
Each algebra $\ah_n$ is graded, with $\deg(t_i) = 0$ and $\deg(x_i) = 1$ (so that $\deg(x_i^{-1}) = -1$). Since these gradings are preserved by the embeddings $\mu_{m,n}$, the algebra $\Tr(\ah)$ inherits this $\Z$-grading. The trace of a tower of algebras always has a rank grading, where $x \in \ah_n$ has degree $n$. The filtration on $\Tr(\dah)$ comes from the filtrations on the $\dah_n$, which are compatible with the embeddings $\mu_{m,n}$.
\end{proof}
\begin{definition}\label{def_grading}
We refer to the the grading coming from the subscript $n$ in $\ah_n$ as the \emph{rank} grading, and the internal $\Z$ (or $\Z_{\geq 0}$) grading induced by the gradings on the $\ah_n$ as the \emph{degree} grading. (We remark that this is consistent with the rank and degree gradings of the elliptic Hall algebra, which come from the rank and degree of coherent sheaves over elliptic curves.)
\end{definition}

\begin{proposition}\label{prop_trdim}
The $\Z_{\geq 1} \times \Z$-graded pieces of $\Tr(\ah^+)$ are finite dimensional, and their dimensions are the same as the dimensions of the graded pieces of the polynomial algebra freely generated by $\{ u_{a,b} \mid a \geq 1, b \geq 0\}$, where $u_{a,b}$ is assigned degree $(a,b) \in \Z_{\geq 1}\times \Z$.
\end{proposition}

\begin{proof}
 First, by Lemma \ref{lemma_filtgrad}, the graded pieces of $\Tr(\ah^+)$ are isomorphic (as vector spaces) to the associated graded pieces of $\Tr(\dah)$. Since $\Tr(\dah)$ is flat over $\k[q]$ (see Lemma \ref{lemma_flat}), the associated graded pieces of $\Tr(\dah)$ and $\Tr(\dah(q=1))$ have the same dimensions.
 By \cite[Prop. 6, Cor. 38]{CLLS15}, the graded pieces of $D_{q=1}$ are finite dimensional, and their Poincare series is given by
 \[
  P(u,v) = \prod_{a>0,\,b\geq 0} \frac 1 {1-u^a v^b}.
 \]
 Then \cite[Prop. 7]{CLLS15} shows that this Poincare series is the same as the Poincare series of the polynomial algebra generated by $\{ u_{a,b} \mid a \geq 1, b \geq 0\}$, where $u_{a,b}$ is assigned degree $(a,b) \in \Z_{\geq 1} \times \Z$.
\end{proof}

%
\subsection{The skein algebra of the torus}
%

The HOMFLYPT skein module $\Sk_{q}(M)$ of a 3-manifold $M$ is the space of formal $\k(q^{\pm 1/2},v)$-linear\footnote{It suffices to work over the ground ring $\k[q^{\pm 1/2},(q-1)^{-1},v^{\pm 1}]$.} combinations of framed links in $M$ modulo the local `skein relations'
\begin{equation}\label{eq:homfly}
\begin{split}
\hackcenter{
\begin{tikzpicture}[scale =0.8]
\draw[very thick, directed=1] (-.5,0) -- (.5,1.5);
\draw[very thick ] (.5,0) -- (.1,.6);
\draw[very thick, directed=1 ] (-.1,.9) -- (-.5,1.5);
\end{tikzpicture}}
\quad - \quad
\hackcenter{
\begin{tikzpicture} [scale =0.8]
\draw[very thick, directed=1] (.5,0) -- (-.5,1.5);
\draw[very thick ] (-.5,0) -- (-.1,.6);
\draw[very thick, directed=1 ] (.1,.9) -- (.5,1.5);
\end{tikzpicture} }
\;\; = \;\;
(q^{1/2}-q^{-1/2}) \;\;
\hackcenter{
\begin{tikzpicture} [scale =0.8]
\draw[very thick, directed=1] (.5,0) .. controls ++(-.2,.4) and ++(-.2,-.4) .. (.5,1.5);
\draw[very thick,directed=1 ] (-.5,0)  .. controls ++(.2,.4) and ++(.2,-.4) .. (-.5,1.5);
\end{tikzpicture}}
\\
\hackcenter{
\begin{tikzpicture} [scale =0.8]
\draw[very thick ] (0,-.25) -- (0,.5)
    .. controls ++(0,.4) and ++(0,.5) .. (.75,.5)
    .. controls ++(0,-.4) and ++(0,-.15) .. (.25,.5);
\draw[very thick, directed=1 ] (.1,.95) ..controls  ++(-.1,.3) and ++(0,-.3).. (0,1.75);
\end{tikzpicture}}
\;\; = \;\; v^{-1}\;\;
\hackcenter{
\begin{tikzpicture} [scale =0.8]
\draw[very thick,directed=1 ] (0,-.25) -- (0,1.75);
\end{tikzpicture}}
\end{split}
\end{equation}
and the relation that a trivially framed unknot is equal to $(-v+v^{-1})/(q^{1/2}-q^{-1/2})$.
The pictures represent three links in $M$ which are identical outside of an embedded oriented ball in $M$ and which inside the ball appear as pictured.

If $M$ is a thickened surface $\Sigma \times [0,1]$, then we denote its skein module by $\Sk_q(\Sigma)$. This space is an algebra, where the product $L_1\cdot L_2$ of links is the link obtained by stacking $L_1$ on top of $L_2$. (Precisely, the product is induced by the map $\left(\Sigma\times [0,1] \right)\sqcup \left(\Sigma\times [0,1]\right) \to \Sigma \times [0,1]$ given in the first component by $(p,t) \mapsto (p,t/2)$ and in the second component by $(p',t') \mapsto (p',1/2 + t'/2)$.)

Let $\Sk_q(T^2)$ be the HOMFLYPT skein algebra of the torus, which is $\mathrm{H}_1(T^2) = \Z^2$-graded. We recall the following theorem from \cite{MS14} (where their $q^2$ is our $q$).

\begin{theorem}[\cite{MS14}]\label{thm_skeinpres}
There exists elements $P_\bx \in \Sk_q(T^2)$ for $\bx \in \Z\times \Z$ which generate $\Sk_q(T^2)$ and which satisfy the relations
\begin{equation}\label{eq_homflyrels}
[P_\bx, P_\by] = -(q^{d/2} - q^{-d/2}) P_{\bx+\by},\quad \quad d = \mathrm{det}[\bx\,\by].
\end{equation}
The algebra $\Sk_q(T^2)$ is isomorphic to the abstract algebra generated by the $P_\bx$ with the relation above.
\end{theorem}
We write
\begin{equation}\label{eq:def-qbrace}
\begin{split}
\{d\}_q &:= q^{d/2}-q^{-d/2}, \\
\end{split}
\end{equation}
so that $[P_\bx, P_\by] = -\{d\}_q P_{\bx+\by}$.
When we omit the subscript in ~\eqref{eq:def-qbrace} we will assume it is $q$.

\begin{remark}
 The $P_\bx$ must be chosen carefully for the above commutation relations to be so simple.
 Generically there are a large number of crossings in the skein product of $P_\bx$ and $P_\by$, so it surprising that the right-hand side has so few terms.  A similar phenomena occurs in the trace of the $q$-Heisenberg category. (See, for example,
Lemma \ref{lemma_crazyidentity} or the last step in the proof of Corollary \ref{cor:prob6}.)
%
\end{remark}

We will use the following subalgebras of $\Sk_q(T^2)$.

\begin{definition} Fix an identification $T^2 \cong S^1 \times S^1$, and call the first and second copies of $S^1$ the \emph{meridian} and \emph{longitude}, respectively. We will visualize the meridian as `horizontal' and the longitude as `vertical.'
\begin{enumerate}
\item Let $\Sk_q^{>,\bullet}(T^2)$ be the subalgebra of $\Sk_q(T^2)$ generated by loops which have a representative which only crosses the meridian positively.
\item Let $\Sk_q^{>,\geq}(T^2)$ be the subalgebra of $\Sk_q(T^2)$ generated by loops which have a representative which only crosses the meridian positively, and which only crosses the longitude non-negatively.
\end{enumerate}
\end{definition}

\begin{remark}\label{rmk_presentation}
The subalgebra $\Sk_q^{>,\bullet}(T^2)$ is generated by $\{P_{a,b} \mid a > 0\}$ and the subalgebra $\Sk^{>,\geq}_q(T^2)$ is generated by $\{P_{a,b}\mid a > 0, b \geq 0\}$. Since all relations between generators of these subalgebras are words in these generators, Theorem \ref{thm_skeinpres} also gives a presentation of $\Sk_q^{>,\bullet}(T^2)$ and $\Sk_q^{>,\geq}(T^2)$.
\end{remark}

\begin{corollary}\label{cor_homflydim}
 The graded pieces of $\Sk_q^{>,\geq}(T^2)$ are finite dimensional, and their dimensions are the same as the dimensions of the graded pieces of the polynomial algebra generated by $\{ u_{a,b} \mid a \geq 1, b \geq 0\}$, where $u_{a,b}$ is assigned degree $(a,b) \in \Z_{\geq1} \times \Z_{\geq 0}$.
\end{corollary}
\begin{proof}
 By Remark \ref{rmk_presentation}, the algebra $\Sk^{>,\geq}_q(T^2)$ has a presentation where generators are $\{P_{a,b}\mid a > 0,\, b \geq 0\}$ and the relations are given by \eqref{eq_homflyrels}, which shows that the graded pieces of $\Sk_q^{>,\geq}(T^2)$ are at most the size claimed. Using a PBW-type theorem of Przytycki from \cite{Prz92}, it was shown in \cite[Cor. 3.3]{MS14} that $\Sk_q(T^2)$ has a basis indexed by unordered words in the generators $P_\bx$. This implies that a basis of $\Sk^{>, \geq}_q(T^2)$ is given by unordered words in the set $\{P_{a,b} \mid a > 0,\, b \geq 0\}$.
\end{proof}

%
\subsection{The comparison}\label{sec_comparison}
%

The proof of the relations in $\Sk_q(T^2)$ involved the $\mathrm{SL}_2(\Z)$ symmetry of $\Sk_q(T^2)$ in an essential way. One of the goals of this paper is to translate these relations to the trace of the quantum Heisenberg category, which does not have such a large symmetry group. This begins by relating the tower of affine Hecke algebras to the skein algebra of the torus.

\begin{definition}
Let $BH_n$ be the space of annular braids modulo the HOMFLYPT skein relations. This is an algebra where the product is given by composition of braids (which is vertical composition in the diagrams in \eqref{def:braidgens} below).
\end{definition}
We will represent annular braids diagrammatically by drawing them in a square whose vertical edges are to be identified. For example, we use the following braids:
\begin{equation}\label{def:braidgens}
\sigma_i \;\; = \;\;
\hackcenter{
\begin{tikzpicture} [scale=0.8]
\fill[gray!20!white] (-1.8,0) rectangle (1.8,2);
\draw[ thick, red, directed=0.6,directed=0.68] (-1.8,0) -- (-1.8,2);
\draw[ thick, red, directed=0.6,directed=0.68] (1.8,0) -- (1.8,2);
\draw[ thick, blue ] (1.8,0) -- (-1.8,0);
\draw[ thick, blue] (1.8,2) -- (-1.8,2);
\draw[ultra thick,->] (-1.2,0) to (-1.2,2);
\draw[ultra thick,->] (1.2,0) to (1.2,2);
\draw[ultra thick,->] (.3,0) .. controls ++(0,1.2) and ++(0,-1.2) .. (-.3,2);
    \fill[gray!20!white] (-0.1,.8) rectangle (0.1,1.2);
\draw[ultra thick,->] (-.3,0) .. controls ++(0,1.2) and ++(0,-1.2) .. (.3,2);
\node at (-.75,1) {$\cdots$};
\node at (.75,1) {$\cdots$};
\node at (-1.2,-.25) {$\scs 1$};
\node at (-.3,-.25) {$\scs i$};
\node at (.35,-.25) {$\scs i+1$};
\node at (1.2,-.25) {$\scs n$};
\end{tikzpicture} }
\qquad \quad
X_i \;\; = \;\;
\hackcenter{
\begin{tikzpicture} [scale=0.8]
\fill[gray!20!white] (-3,0) rectangle (1.8,3);
\draw[ thick, red, directed=0.7,directed=0.78] (-3,0) -- (-3,3);
\draw[ thick, red, directed=0.7,directed=0.78] (1.8,0) -- (1.8,3);
\draw[ thick, blue ] (1.8,0) -- (-3,0);
\draw[ thick, blue] (1.8,3) -- (-3,3);
\draw[ultra thick,<-] (-.6,3) .. controls ++(0,-1.2) and ++(1.2,0) .. (-3,1.5);
    \fill[gray!20!white] (-2.55,1.2) rectangle (-2.25,2.2);
    \fill[gray!20!white] (-1.35,1.2) rectangle (-1.05,2.2);
\draw[ultra thick,->] (-2.4,0) to (-2.4,3);
\draw[ultra thick,->] (-1.2,0) to (-1.2,3);
\draw[ultra thick,->] (0,0) to (0,3);
\draw[ultra thick,->] (1.2,0) to (1.2,3);
    \fill[gray!20!white] (-.1,.95) rectangle (.2,1.3);
    \fill[gray!20!white] (.55,1.3) rectangle (1.45,1.65);
\draw[ultra thick,->] (-.6,0) .. controls ++(0,1.2) and ++(-1.2,0) .. (1.8,1.5);
\node at (-1.8,.7) {$\cdots$};
\node at (.6,.7) {$\cdots$};
\node at (-2.4,-.25) {$\scs 1$};
\node at (-0.6,-.25) {$\scs i$};
\node at (1.2,-.25) {$\scs n$};
\end{tikzpicture}}
\end{equation}

\begin{lemma}\label{lemma:ahtobraid}
There is an algebra map $\ah_n \to BH_n$ which sends $t_i \mapsto q^{1/2}\sigma_i$ and $x_i \mapsto X_i$.
\end{lemma}
\begin{proof}
It is elementary to check that all the relations in the affine Hecke algebra $\ah_n$ are satisfied in $BH_n$. For example, the quadratic relation involving $t_i$ is sent to the skein relation.
\end{proof}
\begin{remark}\label{rmk:product}
The map in the previous lemma intertwines the inclusions $\mu_{m,n}$ between affine Hecke algebras with the diagrammatic construction which takes two braids $a$ and $b$ and juxtaposes them horizontally with $a$ to the left of $b$, and all strands of $a$ going over all strands of $b$.
\end{remark}
\begin{example}\label{example:diagrelation}

We now give a skein theoretic interpretation of Example \ref{example:trrel}.  Using the definitions above one has
\begin{align}
  [x_1] \cdot[1_1] &\;\; = \;\;
\hackcenter{
\begin{tikzpicture} [scale=0.8]
\fill[gray!20!white] (-1.5,0) rectangle (1.5,3);
\draw[ thick, red, directed=0.7,directed=0.78] (-1.5,0) -- (-1.5,3);
\draw[ thick, red, directed=0.7,directed=0.78] (1.5,0) -- (1.5,3);
\draw[ thick, blue ] (1.5,0) -- (-1.5,0);
\draw[ thick, blue] (1.5,3) -- (-1.5,3);
\draw[ultra thick,<-] (-.4,3) .. controls ++(0,-1.0) and ++(.8,0) .. (-1.5,1.5);
\draw[ultra thick,->] (.4,0) to (.4,3);
    \fill[gray!20!white] (.2,1.05) rectangle (.6,1.45);
\draw[ultra thick,->] (-.4,0) .. controls ++(0,1.0) and ++(-1,0) .. (1.5,1.5);
\end{tikzpicture} }   \nn
\\
[x_2]\; = \;[1_1]\cdot [x_1] &\;\; = \;\;
\hackcenter{
\begin{tikzpicture} [scale=0.8]
\fill[gray!20!white] (-1.5,0) rectangle (1.5,3);
\draw[ thick, red, directed=0.7,directed=0.78] (-1.5,0) -- (-1.5,3);
\draw[ thick, red, directed=0.7,directed=0.78] (1.5,0) -- (1.5,3);
\draw[ thick, blue ] (1.5,0) -- (-1.5,0);
\draw[ thick, blue] (1.5,3) -- (-1.5,3);
\draw[ultra thick,<-] (.4,3) .. controls ++(0,-1.2) and ++(1,0) .. (-1.5,1.5);
    \fill[gray!20!white] (-.6,1.2) rectangle (-.2,1.8);
\draw[ultra thick,->] (-.4,0) to (-.4,3);
\draw[ultra thick,->] (.4,0) .. controls ++(0,1.0) and ++(-.6,0) .. (1.5,1.5);
\end{tikzpicture} }
\;\; \textcolor[rgb]{1.00,0.00,0.00}{=}\;\;
\hackcenter{
\begin{tikzpicture} [scale=0.8]
\fill[gray!20!white] (-1.5,0) rectangle (1.5,3);
\draw[ thick, red, directed=0.7,directed=0.78] (-1.5,0) -- (-1.5,3);
\draw[ thick, red, directed=0.7,directed=0.78] (1.5,0) -- (1.5,3);
\draw[ thick, blue ] (1.5,0) -- (-1.5,0);
\draw[ thick, blue] (1.5,3) -- (-1.5,3);
\draw[ultra thick,<-] (-.4,3) .. controls ++(0,-1.0) and ++(.8,0) .. (-1.5,1.5);
\draw[ultra thick,->] (-.4,0) .. controls ++(0,1.0) and ++(-1,0) .. (1.5,1.5);
    \fill[gray!20!white] (.2,1.0) rectangle (.6,1.4);
\draw[ultra thick,->] (.4,0) to (.4,3);
\end{tikzpicture} }\nn
\end{align}
in $\Sk_q(T^2)$, so by the skein relation we have
\begin{align} \nn
[x_1]\cdot [1_1]- [1_1]\cdot [x_1]
& \;\;\refequal{\eqref{eq:homfly}}\;\;
(q^{\frac{1}{2}} - q^{-\frac{1}{2}})\;\;
\hackcenter{
\begin{tikzpicture} [scale=0.8]
\fill[gray!20!white] (-1.5,0) rectangle (1.5,3);
\draw[ thick, red, directed=0.7,directed=0.78] (-1.5,0) -- (-1.5,3);
\draw[ thick, red, directed=0.7,directed=0.78] (1.5,0) -- (1.5,3);
\draw[ thick, blue ] (1.5,0) -- (-1.5,0);
\draw[ thick, blue] (1.5,3) -- (-1.5,3);
\draw[ultra thick,<-] (-.4,3) .. controls ++(0,-1.0) and ++(.8,0) .. (-1.5,1.5);
\draw[ultra thick,->] (.4,0) .. controls ++(0,1.0) and ++(-.8,0) .. (1.5,1.5);
\draw[ultra thick,->] (-.4,0) .. controls ++(0,1.3) and ++(0,-1.3) .. (.4,3);
\end{tikzpicture} }
\;\; =\;\;
q^{-\frac{1}{2}}(q - 1)\;\;
\hackcenter{
\begin{tikzpicture} [scale=0.8]
\fill[gray!20!white] (-1.5,0) rectangle (1.5,3);
\draw[ thick, red, directed=0.7,directed=0.78] (-1.5,0) -- (-1.5,3);
\draw[ thick, red, directed=0.7,directed=0.78] (1.5,0) -- (1.5,3);
\draw[ thick, blue ] (1.5,0) -- (-1.5,0);
\draw[ thick, blue] (1.5,3) -- (-1.5,3);
\draw[ thick, blue, dotted] (1.5,2) -- (-1.5,2);
\draw[ultra thick,<-] (.4,2) .. controls ++(0,-.8) and ++(.8,0) .. (-1.5,1);
    \fill[gray!20!white] (-.6,.8) rectangle (-.2,1.8);
\draw[ultra thick] (-.4,0) to (-.4,1.8);
\draw[ultra thick] (.4,0) .. controls ++(0,.8) and ++(-.6,0) .. (1.5,1);
\draw[ultra thick] (.4,2) .. controls ++(0,.4) and ++(0,-.4) .. (-.4,3);
    \fill[gray!20!white] (-.15,2.3) rectangle (.15,2.6);
\draw[ultra thick,->] (-.4,1.8) .. controls ++(0,.6) and ++(0,-.4) .. (.4,3);
\end{tikzpicture} }
\\
& \;\; = \;\;
\frac{(q - 1)}{q}\;\; [x_2 t_1].
\end{align}
In the notation used below (see Remark \ref{rmk:allgens}), we have $[x_1] = w_{1,1}$ and $[1_1] = w_{1,0}$, so this computation can be written $[w_{1,1}, w_{1,0}] = (q^{1/2}-q^{-1/2})w_{2,1}$.
\end{example}

\begin{theorem}\label{thm_ahatoskein}
There is a $\Z^2$-graded algebra isomorphism
\[
\cl: \Tr(\ah) \stackrel \sim \to  \Sk_q^{>,\bullet}(T^2).
\]
This map restricts to a $\Z^2$-graded algebra isomorphism
\[
\cl^+: \Tr(\ah^+) \stackrel \sim \to \Sk_q^{>,\geq}(T^2).
\]
\end{theorem}

\begin{proof}
The strategy of proof is as follows. We first define the closure map linearly, and then show it is an algebra map. We then show it is graded and that $\cl$ and $\cl^+$ are both surjective. We then show $\cl^+$ is injective using a dimension count. Finally, we `bootstrap' this to show injectivity of $\cl$ itself by using a Dehn twist to twist an element $a \in \Tr(\ah)$ in the supposed kernel of $\cl$ into $\Tr(\ah^+)$.

There is a natural linear map $BH_n \to \Sk_q(T^2)$ given by gluing the top and bottom edges of the braid diagram to obtain a diagram of a closed link in the thickened torus $T^2 \times [0,1]$. Composing this with the map $\ah_n \to BH_n$ of Lemma \ref{lemma:ahtobraid}, we obtain a linear map $\ah_n \to \Sk_q(T^2)$. This factors through the trace to produce a linear map
\[
\cl_n: \mathrm{HH}_0(\ah_n) \to \Sk_q(T^2)
\]
(This is well-defined because in the closure $\cl_n(ab)$, the $a$ part of the braid can be slid around the torus to obtain the diagram $\cl_n(ba)$. In other words, the links $\cl_n(ab)$ and $\cl_n(ba)$ are isotopic inside the torus.)

Next we check that taking the closure is an algebra map. The algebra structure on $\Sk_q(T^2)$ is defined so that in the product $ab$, the diagram of $a$ is stacked on top of the diagram for $b$. The product in $\Tr(\ah)$ is defined using the inclusions $\mu_{m,n}:\ah_m\otimes \ah_n \to \ah_{m+n}$. By Remark \ref{rmk:product}, these two products agree, so the closure map is an algebra map.

Next we check that $\cl$ and $\cl^+$ are graded and surjective. Under the closure map, the number $n$ of strands in $\ah_n$ becomes the (algebraic) crossing number of the image with the (horizontal) meridian. The internal grading of $\ah_n$ (where $x_i$ has degree 1) turns into the (algebraic) crossing number of the image with the longitude. This shows $\cl$ is a graded map, that $\cl(\Tr(\ah)) \subset \Sk_q^{>,\bullet}(T^2)$, and that $\cl^+(\Tr(\ah^+)) \subset \Sk_q^{>,\geq}(T^2)$. The proof of surjectivity of both maps is essentially identical to \cite[Lemma 3.1]{MS14}.

The map $\cl^+$ is a graded, surjective map, and the graded pieces of the source and target are finite dimensional. By Corollary \ref{cor_homflydim} and Proposition \ref{prop_trdim}, these dimensions are the same, so $\cl^+$ is an isomorphism.

We now use the fact that $\cl^+$ is injective to show that $\cl$ is injective. Let $z_n:\ah_n \to \ah_n$ be the linear isomorphism given by multiplication by $x_1\cdots x_n$. Since this element is central and invertible, this map descends to a linear isomorphism $z_n: \HH_0(\ah_n) \to \HH_0(\ah_n)$. The maps $z_n$ are also compatible with the inclusions $\mu_{m,n}$, in the sense that $\mu_{m,n} \circ(z_m \otimes z_n) = z_{m+n}\circ \mu_{m,n}$ (again since $x_1\cdots x_n$ is central). Therefore, the $z_n$ induce an \emph{algebra isomorphism} $Z: \Tr(\ah) \to \Tr(\ah)$. (The closure map intertwines the isomorphism $Z$ with a Dehn twist of the torus $T^2$.) It is clear that for any $a \in \Tr(\ah)$, the element $Z^k(a) \in \Tr(\ah^+)$ for $k \gg 0$. Since the restriction of the closure map to $\Tr(\ah^+)$ was shown to be injective and since the closure map intertwines $Z$ with a Dehn twist, this implies that the closure map itself is injective.
\end{proof}

\begin{remark}\label{rmk:allgens}
 This theorem implies combined with \cite{MS14} shows that there are elements $w_{a,b}$ for $a > 0$ and $b \in \Z$ that generate $\Tr(\ah)$ as an algebra and that satisfy the relations\footnote{The reason additional minus sign (compared to \cite{MS14}) is that in \cite{MS14}, a generator $P_{a,b}$ was homologically equal to $a\bar x + b \bar y$, where $\bar x$ is the (horizontal) meridian, and $\bar y$ is the (vertical) longitude. However, in our conventions, $a$ and $b$ are the (algebraic) crossing numbers with the meridian and longitude, and are therefore switched from the $a$ and $b$ of \cite{MS14}.}
 \begin{equation}\label{eq_aharmk}
 [w_\bx,w_\by] = -\left(q^{d/2}-q^{-d/2}\right) w_{\bx+\by}
 \end{equation}
 for $d = \det[\bx\,\by]$ and $\bx,\by \in \N \times \Z$. For the sake of completeness, we give an algebraic definition of these generators (even though we won't formally use all of them).

 Let $B_n$ be the affine braid group, which has generators $\sigma_i$ and $x_j^{\pm 1}$ for $1\leq i < n$ and $1 \leq j \leq n$, which satisfy the relations of the affine Hecke algebra $\ah_n$ \eqref{eq:affineheckerel}, without the quadratic relation for $t_i$, and where $\sigma_i := q^{-1/2} t_i$. We will use a \emph{threading map}
 \[\Phi_{n,d}: B_n \to B_{dn}\]
which is a group homomorphism given by replacing each strand with $d$ parallel strands. Algebraically, it is defined as follows:
 \begin{align*}
  \Phi_{n,d}(x_j) &= x_{d(j-1)+1}x_{d(j-1) + 2}\cdots x_{d(j-1)+d}\\
  \Phi_{n,d}(\sigma_i) &= \sigma_{di, d(i-1) + 1}\sigma_{di+1,d(i-1)+2}\cdots \sigma_{di+d-1,d(i-1)+d}
 \end{align*}
where for $a > b$ we have written $\sigma_{a,b}$ for the product
$\sigma_{a}\sigma_{a-1}\cdots \sigma_b$. Note that $\Phi_{n,1}$ is the identity map. The group homomorphism $\Phi_{n,d}$ does \emph{not} descend to an algebra map between affine Hecke algebras, but it will be useful in describing the elements $w_{a,b}$. For clarity, let $\pi: B_n \to \ah_n$ be the projection. Define elements $p_k \in \ah_n$ as
 \[
p_k := \frac{\{1\}}{\{k\}} \sum_{i=0}^{k-1} \sigma_1\cdots\sigma_i \sigma_{i+1}^{-1}\cdots \sigma_{k-1}^{-1}
\]
where by convention, $p_1 := 1$. Given $n,m$ relatively prime and $d \geq 1$, define  $w_{dn,dm} \in \ah_{dn}$ via
\[
 w_{dn,dm} := p_d \cdot \pi\left( \Phi_{n,d}\left(x_n \sigma_{n-1}\cdots \sigma_1\right)^m\right).
\]
Then Theorem \ref{thm_ahatoskein} implies that these elements satisfy equation \eqref{eq_aharmk}.
\end{remark}

%
\section{The elliptic Hall Algebra} \label{sec:elliptic}
%

The elliptic Hall algebra $\E$ of Burban and Schiffmann \cite{BS12} is the Drinfeld double of the universal Hall algebra of the category of coherent sheaves over a smooth elliptic curve over a finite field. An explicit presentation of this algebra was given in \cite{BS12}, and the structure constants were shown to be Laurent polynomials in $\sigma$, $\bar \sigma$, which are the Frobenius eigenvalues of the curve. In \cite{MS14} it was shown that the $\sigma = \bar \sigma^{-1}$ specialization of $\E$ is isomorphic to the HOMFLYPT skein algebra $\Sk_q(T^2)$ of the torus.

In this paper we relate a central extension $\hE$ of $\E$ to the trace of the quantum Heisenberg category. In this section we recall a presentation of a central extension of $\E$ from \cite{SV13}. We then describe the specialization of this algebra that we use, and we give a presentation of the positive `half' of this algebra that we will need.

\subsection{Notation}\label{sec_notation}

We will deal with several algebras which are graded by (some subset of) $\Z \times \Z$, and most of these algebras will have a ``standard'' generator in each degree. We will also need to define various subalgebras generated by subsets of these generators, and we will use notation for these subalgebras which suggests what their generators are (with the convention that an inequality refers to $0$ and that the symbol ``$\bullet$'' stands for ``no restriction''). For example, if $\E$ has generators $\bu_\bx$ for each $\bx \in \Z^2$, then we have the following subalgebras:

\begin{equation}\label{eq_subalgs}
 \E^{>,\bullet} := \langle \bu_{a,b} \mid a > 0\rangle,\quad\quad \E^{>, \geq} := \langle \bu_{a,b}\mid a > 0, b \geq 0\rangle.
\end{equation}

%
\subsection{The definition}
%

The central extension $\hE \twoheadrightarrow \E$ is an algebra over $\bK := \k(\sigma^{1/2}, \bar \sigma^{1/2})$ generated by elements $\bu_\bx$ and $\bk_\bx$
for $\bx \in \Z^2$.   We recall a description of this algebra from \cite{SV13}. Before describing the relations we establish some notation. Let
\[
\bZ^+ = \{(i,j) \in \Z^2 \mid i > 0,\textrm{ or } i=0,\, j > 0\},\quad \quad \epsilon_\bx := \left\{\begin{array}{cl}
1&\textrm{if }\bx \in \bZ^+,\\
-1&\textrm{if }\bx \in -\bZ^+.
\end{array}
\right.
\]
Let $\epsilon_{\bx,\by} := sign(\det(\bx\,\by))$, and if $\bx = (a,b)$ we write $d(\bx) := gcd(a,b) \in \Z_{\geq 1}$. Define constants
\begin{eqnarray*}
 \alpha_n &:=& (1-(\sigma \bar \sigma)^{-n})(1-\sigma^{n})(1-\bar\sigma^n)/n\\
 \alpha(\bx,\by) &:=& \left\{
 \begin{array}{cl} \epsilon_\bx(\epsilon_\bx \bx + \epsilon_\by \by - \epsilon_{\bx+\by}(\bx+\by))/2 & \textrm{if } \epsilon_{\bx,\by} = 1\\
 \epsilon_\by(\epsilon_\bx \bx + \epsilon_\by \by - \epsilon_{\bx+\by}(\bx+\by))/2 & \textrm{if } \epsilon_{\bx,\by} = -1
 \end{array} \right.\\
 \delta(\bx,\by) &:=&
 \left\{\begin{array}{cl}
1&\textrm{if }\bx = \by \\
0&\textrm{otherwise}
\end{array}
\right.
\end{eqnarray*}
 We define elements $\theta_{j\bx}$ for $d(\bx)=1$ using the following equality of formal power series:
\begin{equation*}
 \sum_{j \geq 0} \theta_{j\bx} z^j := \mathrm{exp}\left( \sum_{r \geq 1}\alpha_r \bu_{r\bx} z^r\right).
\end{equation*}

The relations in $\hE$ satisfied by the generators $\bu_\bx$ and $\bk_\bx$ are the following:
\begin{itemize}
 \item The $\bk_\bx$ are central, $\bk_\bx\bk_\by = \bk_{\bx+\by}$, and $\bk_{0,0} = 1$.
 \item If $\bx,\by \in \Z^{2}-\{(0,0)\}$ are on the same line (through the origin), then
 \begin{equation}\label{eq_linerel}
  [\bu_\by, \bu_\bx] = \delta(\bx,-\by) \frac{\bk_\bx - \bk_{\bx}^{-1}}{\alpha_{d(\bx)}}.
 \end{equation}
 \item If $\bx,\by \in \Z^{2}-\{(0,0)\}$, $d(\bx) = 1$, and the triangle with vertices $\mathbf{0}$, $\bx$, and $\bx+\by$ has no interior lattice points, then
 \begin{equation}\label{eq_reltriangle}
  [\bu_\by,\bu_\bx] = \epsilon_{\bx,\by} \bk_{\alpha(\bx,\by)}\frac{\theta_{\bx+\by}}{\alpha_1}.
 \end{equation}
\end{itemize}

For convenience we set $\bu_{(0,0)}=1$. If we write $\bk_1 := \bk_{0,1}$ and $\bk_2 := \bk_{1,0}$, then the first relation implies the $\bk_\bx$ generate an algebra isomorphic to $\k[\bk_1^{\pm 1},\bk_2^{\pm 1}]$. We may therefore view $\hE$ as an algebra over $\bK[\bk_1^{\pm 1},\bk_2^{\pm 1}]$.

\begin{definition}
We also define some algebras related to $\hE$.
\begin{enumerate}
\item The algebra $\hE^+$ is the $\bK[\bk_1^{\pm 1},\bk_2^{\pm 1}]$-subalgebra of $\hE$ generated by $\bu_\bx$ for all $\bx \in \bZ^+$.
\item The algebra $\E$ is defined as the quotient of $\hE$ by the relations $\bk_1 = \bk_2 = 1$.
\item The algebra $\E^+$ is the quotient of $\hE^+$ by $\bk_1=\bk_2=1$.
\end{enumerate}
\end{definition}

\begin{remark}\label{rmk_plusnoextension}
 We note that if $\bx,\by \in \bZ^+$, then $\alpha(\bx,\by) = (0,0) \in \Z^2$.
Then the commutator in \eqref{eq_linerel} is always zero
 and
$\bk_{\alpha(\bx,\by)} = 1$ in \eqref{eq_reltriangle}.
This implies that $\hE^+$ is isomorphic to $\E^+ \otimes_\bK \bK[\bk_1^{\pm 1},\bk_2^{\pm 1}]$. Similarly, $\E^{>,\bullet}\otimes_\bK \bK[\bk_1^{\pm 1},\bk_2^{\pm 1}]$ and $\hE^{>,\bullet}$ are isomorphic. (See \eqref{eq_subalgs} for notational conventions.)
\end{remark}

 We now recall a very useful triangular decomposition of $\hE$.

\begin{proposition}[{\cite[Proposition 1.1]{SV13}}]\label{prop_svtriangle}
 The multiplication map induces a $\bK[\bk_1^{\pm 1},\bk_2^{\pm 1}]$-linear isomorphism
 \[
 \hE^{<,\bullet} \otimes \hE^{0,\bullet} \otimes \hE^{>,\bullet} \stackrel \sim \to \hE.
 \]
 The algebra $\hE$ is isomorphic to the algebra generated by $\hE^{<,\bullet}$, $\hE^{0,\bullet}$, and $\hE^{>,\bullet}$, subject to the following relations:
 \begin{equation}\label{eq_nogoodnameforthis}
  \, [\bu_{0,k},\bu_{1,l}]
 \;=\;
 \left\{
   \begin{array}{ll}
    \bu_{1,l+k}, & \hbox{$k >0$,} \\
     -\bk_1^k \bu_{1,l+k}, & \hbox{$k < 0$,}
   \end{array}
 \right.
 \qquad \quad
 \,[\bu_{-1,l},\bu_{0,k}] \; =\;
 \left\{
   \begin{array}{ll}
     \bk_1^{-k}\bu_{-1,k+l}, & \hbox{$k > 0$,} \\
     -\bu_{-1,k+l}, & \hbox{$k<0$,}
   \end{array}
 \right.
\end{equation}
\begin{equation} \label{eq_oneone}
 \,[\bu_{-1,k},\bu_{1,l}]
 \; = \;
 \left\{
   \begin{array}{ll}
    \bk_2\bk_1^{-k}\frac{\theta_{0,k+l}}{\alpha_1}, & \hbox{$k+l>0$,} \\
     \frac{\bk_1^{-k}\bk_2-\bk_1^k\bk_2^{-1}}{\alpha_1}, & \hbox{$k=-l$,} \\
       -\bk_2^{-1}\bk_1^{-l}\frac{\theta_{0,k+l}}{\alpha_1}, & \hbox{$k+l<0$.}
   \end{array}
 \right.
\end{equation}
\end{proposition}

%
\subsection{The specialization}\label{sec_specialization}
%

We would like to specialize $\sigma = \bar \sigma^{-1} = q$ and $\bk_1=\bk_2=1$, but since the structure constants of $\hE$ have poles at these specializations, we must specify an integral version of $\hE$ and the order in which we perform these specializations. To shorten notation, let $\tau := (\sigma \bar \sigma)^{-1/2}$
and let $R := \k[\sigma^{\pm 1/2},\bar \sigma^{\pm 1/2}]$.  Recall that we also write $\{d\}_q = q^{d/2}-q^{-d/2}$.

Let $\hE_{1,\tau}$ be the $R$-subalgebra of $\hE$ generated by the $\bu_\bx$ and specialized at $\bk_1 = 1$ and $\bk_2 = \tau$. This has the effect of changing the middle equation of \eqref{eq_oneone} to the following:
\[
[\bu_{-1,k},\bu_{1,l}]= \frac{-\tau^{-1}}{(1-\sigma)(1-\bar\sigma)},\quad k=-l.
\]
\begin{definition}
Let $\mathbb{E}$ be the $\k[q^{\pm 1/2}]$-algebra
obtained from $\hE_{1,\tau}$ by specializing $\sigma = q$ and $\bar \sigma = q^{-1}$.
\end{definition}
In this specialization we have the following straightforward identity (see, e.g. \cite[Lemma 5.4]{MS14})
\[
 \left.\frac{\theta_\bx}{\alpha_1} \right\vert_{\bar\sigma=\sigma^{-1}} =
 \frac{\{d(\bx)\}_\sigma^2}{\{1\}_\sigma^2} \bu_\bx
\]

We define renormalized generators
\begin{equation} \label{eq:defbw}
\bw_\bx := \left(q^{d(\bx)/2}-q^{-d(\bx)/2}\right)\bu_\bx,
\end{equation}
 and for clarity we rewrite the presentation of the specialization $\mathbb{E}$ in terms of these generators.
\begin{corollary}\label{cor_trianglepres}
 The multiplication map induces a $\k(q)$-linear isomorphism
 \[
 \mathbb{E}^{>,\bullet} \otimes \mathbb{E}^{0,\bullet} \otimes \mathbb{E}^{<,\bullet} \stackrel \sim \to \mathbb{E}.
 \]
We also have the following:
\begin{enumerate}
\item \label{boldE-trianglepres1}The algebra $\mathbb{E}$ is isomorphic to the algebra generated by $\mathbb{E}^{<,\bullet}$, $\mathbb{E}^{0,\bullet}$, and $\mathbb{E}^{>,\bullet}$, subject to the following cross relations:
 \begin{align*}
  \, [\bw_{0,k},\bw_{1,l}] &= \{k\}\bw_{1,l+k}, \\
 \,[\bw_{-1,l},\bw_{0,k}] &= \{k\}\bw_{-1,k+l}, \\
 \,[\bw_{-1,k},\bw_{1,l}] &=
 \left\{
   \begin{array}{ll}
     \{k+l\} \bw_{0,k+l}, & \hbox{$k+l\not= 0$,} \\
     1, & \hbox{$k=-l$.}
   \end{array}
 \right.
 \end{align*}
 \item \label{boldE-trianglepres2} The assignment $\bw_\bx \mapsto P_\bx$ extends to an algebra isomorphism $\mathbb{E}^{>,\bullet}\stackrel{\sim}{\to} \Sk^{>,\bullet}_q(T^2)$ between the positive parts of the (central extension of the) elliptic Hall algebra and the skein algebra.
 \item \label{boldE-trianglepres3} The assignment $\bw_{a,b} \mapsto \bw_{-a,b}$ extends to an anti-involution on $\mathbb{E}$ which restricts to the identity on the commutative subalgebra $\mathbb{E}^{0,\bullet}$ and switches $\mathbb{E}^{>,\bullet}$ and $\mathbb{E}^{<,\bullet}$.
\end{enumerate}
\end{corollary}
\begin{proof}
The triangular decomposition  and the first statement \eqref{boldE-trianglepres1} follow immediately from the definition of $\bw_{\bx}$ in \eqref{eq:defbw} and Proposition~\ref{prop_svtriangle}.  (We note that the signs in the relations of equation \eqref{eq_nogoodnameforthis} are absorbed into the definition of $\{k\} = q^{k/2}-q^{-k/2}$, since $d(\bx) = d(-\bx)$.)
In \cite[Thm.\ 5.6]{MS14} an isomorphism is given between a specialization of the elliptic Hall algebra and $\Sk^{>,\bullet}_q(T^2)$.  Remark~\ref{rmk_plusnoextension} identifies the positive parts of $\hE$ and $\E$, so that the second assertion~\eqref{boldE-trianglepres2} follows.  The last statement \eqref{boldE-trianglepres3} follows from \eqref{boldE-trianglepres1}.

\end{proof}

\begin{remark}\label{rmk_ellheis}
Let $\bx = (k,0)= -\by$ with $k > 0$. We note that if we specialize relation \eqref{eq_linerel} at $\bk_1=1$ and $\bk_2 = \tau$, then it becomes
\[
[\bu_{-k,0}, \bu_{k,0}] = \frac{\bk_2^k - \bk_2^{-k}}{\alpha_{k}} = \frac{(1-\tau^{-2k})\tau^k k}{(1-\tau^{2k})(1-\sigma^k)(1-\bar \sigma^k)}.
\]
If we then specialize $\sigma = q$ and $\bar \sigma = q^{-1}$, and rewrite this relation in terms of the renormalized generators $\bw_\bx = \{d(\bx)\} \bu_\bx$ we obtain
\[
[\bw_{-k,0},\bw_{k,0}] = \{k\}^2 [\bu_{-k,0},\bu_{k,0}] = \{k\}^2 \frac{-k}{\{k\}\{-k\}} = k.
\]
\end{remark}

%
\subsection{A triangular decomposition of the ``upper half''}\label{sec_upperhalf}
%

The algebra which we eventually relate to the trace of the $q$-Heisenberg is the subalgebra $\mathbb{E}^{\bullet,\geq} = \langle \bw_{j,k} \mid j \in \Z,\, k \geq 0\rangle$ of the algebra $\mathbb{E}$ described in the previous subsection. The goal of this section is to give a presentation of  $\mathbb{E}^{\bullet, \geq}$ using a triangular decomposition modulo cross relations. (See Section \ref{sec_notation} for notational conventions regarding subalgebras of $\mathbb{E}$, which we use throughout this section.)

We will use Corollary \ref{cor_trianglepres} to give a triangular decomposition of $\mathbb{E}^{\bullet,\geq}$. However, one key fact used in the triangular decomposition for $\mathbb{E}$ is that the pieces are generated by the elements
$\{\bw_{1,b}\}$, $\{\bw_{0,b}\}$, and $\{\bw_{-1,b}\}$, respectively, for $b \in \Z$. Unfortunately, this is no longer true for the pieces of the subalgebra $\mathbb{E}^{\bullet,\geq}$, so the analogue of the presentation in Corollary \ref{cor_trianglepres} is necessarily more complicated. For example, to generate $\mathbb{E}^{>, \geq}$ we need the elements $\{\bw_{1,k},\bw_{j,0}\}$ for $k,j > 0$.

In what follows, we write $d(\bx, \by) = \det\left(\bx\,\, \by\right)$ for $\bx, \by \in \Z^2$ and $d(\bx) = gcd(m,n)$ when $\bx = (m,n)$.
\begin{definition}
We will use the following terminology:
\begin{equation}\label{eq_defofgood}
(\bx,\by) \in \Z^2 \times \Z^2 \textrm{ is \emph{good} if    } [\bw_\bx,\bw_\by] = -\{d(\bx,\by)\}\bw_{\bx+\by}.
\end{equation}	
\end{definition}

\begin{remark}\label{rmk_positivegood}
By Corollary \ref{cor_trianglepres}(\ref{boldE-trianglepres2}), if $x_1,x_2 > 0$, then the pair $((x_1,y_1), (x_2,y_2))$ is good.
\end{remark}

We first recall a useful technical lemma from \cite{MS14} (for the convenience of the reader we include the proof, since we will need to refer to it later). The idea is that to compute a commutator $[\bw_\bx,\bw_\by]$, we will induct on the determinant of the matrix $\left(\bx\,\,\by\right)$, and we use the lemma below to perform the inductive step.

\begin{lemma}[{\cite[Lemma 3.4]{MS14}}]\label{lemma_trueforab}
Assume $\ba + \bb = \bx$ and that $(\ba,\bb)$ is good. Further assume that the four pairs of vectors $(\by, \ba)$, $(\by, \bb)$, $(\by+\ba,\bb)$, and $(\by+\bb,\ba)$ are good. Then the pair $(\bx,\by)$ is good.
\end{lemma}
\begin{proof}
 By the first assumption, we have $[\bw_\ba, \bw_\bb] = - \{d(\ba, \bb)\}\bw_\bx$. We then use the Jacobi identity and the remaining assumptions to compute
 \begin{eqnarray*}
  \{d(\ba, \bb)\}[\bw_\bx,\bw_\by] &=& -[[\bw_\ba,\bw_\bb],\bw_\by]\\
  &=& [[\bw_\by,\bw_\ba],\bw_\bb] + [[\bw_\bb,\bw_\by],\bw_\ba]\\
  &=& - \{d(\by,\ba)\}[\bw_{\by+\ba},\bw_\bb] - \{d(\bb,\by)\}[\bw_{\bb+\by},\bw_\ba]\\
  &=& \Big(\{d(\by,\ba)\}\{d(\by+\ba,\bb)\} + \{d(\bb,\by)\}\{d(\bb+\by,\ba)\}\Big) \bw_{\bx+\by}\\
  &=:& c\bw_{\bx+\by}.
 \end{eqnarray*}
A short computation then shows $c = -\{d(\ba,\bb)\}\{d(\bx,\by)\}$ which completes the proof of the lemma.
\end{proof}

%
We can now prove the following triangular decomposition of $\mathbb{E}^{\bullet, \geq}$.\\

\begin{proposition}\label{prop_tridecomp}
 The algebra $\mathbb{E}^{\bullet,\geq}$ has a triangular decomposition $$\mathbb{E}^{\bullet,\geq} = \mathbb{E}^{>,\geq} \otimes \mathbb{E}^{0,\geq} \otimes \mathbb{E}^{<,\geq}.$$
Furthermore, the algebra $\mathbb{E}^{\bullet,\geq}$ is isomorphic to the algebra $\tE$ generated  by the subalgebras $\mathbb{E}^{>,\geq}$, $\mathbb{E}^{0,\geq}$, and $\mathbb{E}^{<,\geq}$ subject to the following cross-relations:
 \begin{eqnarray}
 \,[\bw_{\pm 1,r},\bw_{0,k}] &=& \{\mp k\} \bw_{\pm 1, r+k}\label{eq_relhard}\\
 \,[\bw_{\pm 1,0},\bw_{\mp 1,k}] &=& \{\mp k\}\bw_{0,k}\label{eq_rel2b}\\
 \,[\bw_{\pm 1,k},\bw_{\mp 1,1}]&=& \mp \{k+1\}\bw_{0,k+1}\label{eq_rel3}\\
\,[\bw_{i,0},\bw_{j,0}]&=& \left\{ \begin{array}{cl} -i & \mathrm{if }\,\, i=-j>0\\ 0 & \mathrm{otherwise}\end{array}\right.\label{eq_rel4}\\
 \,[\bw_{j,0},\bw_{\pm 1,1}] &=& \{-j\} \bw_{j \pm 1,1}.\label{eq_rel6}
\end{eqnarray}
(In these relations, $i,j \in \Z$ and $r,k \in \Z_{\geq 0}$. We have also abused notation by referring to the generators of $\tE$ as $\bw_\bx$, instead of the more cumbersome $\check{\bw}_\bx$.)

Finally, the algebra $\mathbb{E}^{\bullet, \geq}$ is generated by $\bw_{a,b}$ with $b \geq 0$ subject to the relations
\begin{equation}\label{eq_mainrel}
[\bw_\bx,\bw_\by] = \left\{\begin{array}{cl}
-k & \textrm{if } \bx = (k,0) = -\by\\
-\{d(\bx,\by)\} \bw_{\bx+\by} & \textrm{otherwise}\\\end{array}\right.
\end{equation}
In particular, the algebra $\mathbb{E}^{>,\geq}$ is isomorphic to the skein algebra $\Sk^{>,\geq}_q(T^2)$.
\end{proposition}

\begin{proof}
It will be useful in the following to note that the map
\begin{equation}\label{eq_anti}
\bw_{a,b} \mapsto \bw_{-a,b}
\end{equation}
extends to an algebra anti-involution of the algebra $\tE$ since the cross relations we impose are preserved by this map.

We begin by noting that there is an obvious surjection $\tE \twoheadrightarrow \mathbb{E}$, since the cross-relations \eqref{eq_relhard}--\eqref{eq_rel6} hold in $\mathbb{E}$. (For relation \eqref{eq_rel4}, see Remark \ref{rmk_ellheis}.) In what follows we will show that relation \eqref{eq_mainrel} follows from the cross-relations and the relations inside each of the triangular pieces. This will show that elements of $\tE$ can be written as sums of products of the form $\bw_{x_1,y_1}\bw_{0,y_2}\bw_{x_3,y_3}$ with $x_1<0$ and $x_3>0$. This shows that ${\tE}^{\bullet, \geq}$ and $\mathbb{E}^{\bullet,\geq}$ are isomorphic. Finally, the triangular decomposition of $\mathbb{E}^{\bullet,\geq}$ follows from the triangular decomposition of $\mathbb{E}$ in Corollary \ref{cor_trianglepres}.

We will prove relation \eqref{eq_mainrel} by induction on $d := \lvert \det\left(\bx\,\by\right)\rvert$. The statement when $d=0$ (i.e. when $\bx$ and $\by$ are collinear) follows from the fact that the triangular pieces of $\mathbb{E}$ are subalgebras and from relation \eqref{eq_rel4}.

We then need to show that if $\lvert d(\bx, \by) \rvert \geq 1$, then $(\bx,\by)$ is \emph{good} (see equation \eqref{eq_defofgood} for the definition).
For the rest of the proof we use the following notation:
\begin{equation}\label{eq_notation}
\by = (-j,k),\quad \bx = (m,n),\quad d := \lvert\, d(\bx,\by)\,\rvert  = mk+jn, \quad j,k,m,n \geq 0.
\end{equation}
 To justify the inequality, we argue as follows. If the assumption $j,m \geq 0$ doesn't hold, then there are three possibilities. If one of $j,m$ is equal to 0, then the other is strictly negative, and we can switch $\bx$ and $\by$. If both $j$ and $m$ are strictly negative, we can again switch $\bx$ and $\by$. Finally, if one of $j,m$ is strictly positive and the other is strictly negative, then both $\bw_\bx$ and $\bw_\by$ are in the same triangular piece since $-j$ and $m$ have the same sign. Then $(\bx,\by)$ is good because the triangular pieces are subalgebras with the desired presentation by Theorem ~\ref{thm_skeinpres} and Corollary ~\ref{cor_trianglepres}, item ~\eqref{boldE-trianglepres2}.
This justifies the inequality at the end of \eqref{eq_notation}.

For the base case $d =1$, one of the indices $j,k,m,n$ must be equal to 0, and by the symmetry \eqref{eq_anti}, we may assume either $j=0$ or $k=0$. If $j=0$, then $m=k=1$, and the pair $(\bx,\by)$ is good by relation \eqref{eq_relhard}. If $k=0$, then $j = n = 1$, and we need to prove $[\bw_{-1,0},\bw_{m,1}] = \{1\}\bw_{m-1,1}$. To save time later, we instead prove the following more general equality:
\begin{equation}\label{eq_case0}
[\bw_{-1,0},\bw_{m,n}] = \{n\}\bw_{m-1,n},\quad \quad n \geq 1.
\end{equation}
We prove this by induction in $m$, and the base case $m=1$ follows from \eqref{eq_rel2b}.
For the inductive step, $m \geq 2$, and we write $\bx = \ba + \bb = (1,0) + (m-1,n)$ and use Lemma \ref{lemma_trueforab}. The pair $(\ba,\bb)$ is good because both $x$-coordinates are strictly positive. The pair $(\by, \bb)$ is good by induction. The pair $(\by+\bb, \ba)$ is good because both $x$-coordinates are strictly positive if $m > 2$, and by relation \eqref{eq_relhard} if $m=2$. The pair $(\by,\ba)$ is not good since the commutator of $\bw_\by$ and $\bw_\ba$ is $1$ instead of $0$. However, this commutator only appears in the proof of Lemma \ref{lemma_trueforab} in the second line of the computation in the term $[[\bw_\by,\bw_\ba],\bw_\bb]$, and this term is still $0$ as the proof requires, since $1$ is central. In other words, the conclusion of Lemma \ref{lemma_trueforab} is still true. This completes the induction, which completes the base case $d=1$.
\begin{equation*}
\begin{array}{cc|cc|c|l}
\textrm{pair}& &\textrm{vectors}& &\lvert \det \rvert&
\textrm{good because}\\
\hline
\bx&\by&(m,n)&(-1,0)&m&\textrm{Lemma \ref{lemma_trueforab} and the table below}\\
\hline
\ba&\bb& (1,0)&(m-1,n)&n&
\textrm{$x$-coordinates strictly positive}\\
\by&\bb&(-1,0)&(m-1,n)&n&\textrm{induction on $m$}\\
\by+\bb& \ba&(m-2,n)&(1,0)&n&\textrm{if $m=2$, eq. \eqref{eq_relhard},}\\
\,& & & &\, & \textrm{otherwise, $x$-coordinates strictly positive}
\end{array}
\end{equation*}



We now make the inductive assumption that equation \eqref{eq_mainrel} holds for all $\bx,\, \by$ with $1 \leq \lvert d(\bx,\by)\rvert < d$, and suppose $d(\bx, \by) = d$. We want to show relation \eqref{eq_mainrel}. We split the inductive step into cases.

\noindent \textbf{Case 1:} $j=0$. We have the following assumptions, and we split the proof into subcases
\[
\quad \by = (0,k),\quad \bx = (m,n),\quad d=mk.
\]
\begin{enumerate}
\item $m=1$. This is relation \eqref{eq_relhard}.

\item $m \geq 2$ and $n = 0$, so $\by = (0,k)$ and $\bx = (m,0)$.\\
In this subcase we will split $\by = \ba + \bb = (-1,0) + (1, k)$ and use Lemma \ref{lemma_trueforab} (with $\bx$ and $\by$ switched).

The pair $(\ba,\bb)$ is good by
the inductive assumption because $d(\ba,\bb) = k < mk$ (since $m \geq 2$). The pair $(\bx,\ba)$ is good by relation \eqref{eq_rel4}. The pairs $(\bx, \bb)$ and $(\bx+\ba,\bb)$ are good since they are both in one of the triangular pieces (since $m \geq 2$, both their $x$-coordinates are strictly positive).
The pair $(\bx+\bb, \ba)$ is good by induction, since $\lvert d(\bx+\bb,\ba) \rvert = k < mk = d$.

\begin{equation*}
\begin{array}{cc|cc|c|l}
\textrm{pair}& &\textrm{vectors}& &\lvert \det\rvert &
\textrm{good because}\\
\hline
\by&\bx&(0,k)&(m,0)&mk&\textrm{argument below}\\
\hline
\ba&\bb&(-1,0)&(1,k)&k&0<k < mk \textrm{ since } m \geq 2\\
\bx&\ba&(m,0)&(-1,0)&-&\textrm{relation \eqref{eq_rel4}}\\
\bx&\bb&(m,0)&(1,k)&-&\textrm{$x$-coordinates strictly positive}\\
\bx+\ba&\bb&(m-1,0)&
(1,k)&-&
\textrm{$x$-coordinates strictly positive $(m \geq 2)$}\\
\bx+\bb&\ba&(m+1,k)&(-1,0)&k&0<k < mk \textrm{ since } m \geq 2
\end{array}
\end{equation*}

\item $m \geq 2$ and $n \geq 1$, so $\by = (0,k)$ and $\bx = (m,n)$.\\
In this subcase we split $\bx = \ba + \bb = (1,0) + (m-1,n)$ and use Lemma \ref{lemma_trueforab}.

The pairs $(\ba,\bb)$, $(\by+\ba,\bb)$, and $(\by+\bb,\ba)$ are all in the same triangular piece and therefore are good. The pair
$(\by,\ba)$ is good by induction (since $d(\by,\ba) = k < mk$). The pair $(\by, \bb)$ is good by the induction hypothesis because $1 \leq d(\by,\bb) = k(m-1) < km = d$ (where we used $m \geq 2$). This finishes the proof of this subcase and the proof of Case 1.

\begin{equation*}
\begin{array}{cc|cc|c|l}
\textrm{pair}& &\textrm{vectors}& &\lvert \det\rvert &
\textrm{good because}\\
\hline
\bx&\by&(m,n)&(0,k)&mk&\textrm{argument below}\\
\hline
\ba&\bb&(1,0)&(m-1,n)&n&\textrm{$x$-coordinates strictly positive}\\
\by&\ba&(0,k)&(1,0)&k&0<k<mk \textrm{ since } m \geq 2\\
\by&\bb&(0,k)&(m-1,n)&k(m-1)&0 < k(m-1) < km \textrm{ since } m \geq 2\\
\by+\ba&\bb&(1,k)&(m-1,n)&-&
\textrm{$x$-coordinates strictly positive}\\
\by+\bb&\ba&(m-1,k+n)&(1,0)&-&\textrm{$x$-coordinates strictly positive}
\end{array}
\end{equation*}

\end{enumerate}

\noindent \textbf{Case 2:} $k=0$. If $m=0$ we can switch $\bx$ and $\by$ and use Case 1.  If $n=0$ then $d=0$, which we have already ruled out. We therefore have the following assumptions, and we split the proof into subcases
\[
 \by = (-j,0),\quad \bx = (m,n),\quad d = jn,\quad m,\, n \geq 1.
 \]

\begin{enumerate}
\item If $j=1$, this case was proved in \eqref{eq_case0}.
\item Suppose $j \geq 2$. We prove that the pair $(\by,\bx)$ is good by induction in $m$.

For the base case $m=1$, if $n=1$, then this is relation \eqref{eq_rel6}. Otherwise, $n \geq 2$ and we have
\[
 \by = (-j,0),\quad \bx = (1,n),\quad j,n \geq 2,\quad d=jn.
\]
We split $\bx = \ba + \bb = (0,1) + (1,n-1)$.

The pair $(\ba,\bb)$ is good by relation \eqref{eq_relhard}.
The pair $(\by,\ba)$ is good by the inductive hypothesis on $d$, since $1 \leq \lvert d(\by,\ba) \rvert = j < jn$ (we used $n \geq 2$).
The pair $(\by,\bb)$ is good since $1 \leq \lvert d(\by,\bb) \rvert = j(n-1) < jn = d$.
The pair $(\by+\ba,\bb)$ is good since $1 \leq \lvert d(\by+\ba,\bb)\rvert = j(n-1)+1 < jn$ (we used $j,n \geq 2$).
The pair $(\by+\bb,\ba)$ is good since $1 \leq \lvert d(\by+\bb,\ba)\rvert = j-1 < jn$ (we used $j \geq 2$). This completes the proof of the base case $m=1$.

\begin{equation*}
\begin{array}{cc|cc|c|l}
\textrm{pair}& &\textrm{vectors}& &\lvert \det\rvert &
\textrm{good because}\\
\hline
\bx&\by&(1,n)&(-j,0)&jn&\textrm{argument below}\\
\hline
\ba&\bb&(0,1)&(1,n-1)&1&\textrm{relation \eqref{eq_relhard}}\\
\by&\ba&(-j,0)&(0,1)&j&0<j<jn \textrm{ since } n \geq 2\\
\by&\bb&(-j,0)&(1,n-1)&j(n-1)&0<j(n-1)<jn \textrm{ since } n \geq 2\\
\by+\ba&\bb&(-j,1)&(1,n-1)&j(n-1)+1&0<j(n-1)+1<jn \textrm{ since } n,j \geq 2\\
\by+\bb&\ba&(-j+1,n-1)&(0,1)&j-1&0<j-1<jn \textrm{ since } n,j \geq 2
\end{array}
\end{equation*}

We may now assume both inductive assumptions, i.e. that $(\bx',\by')$ is good if either
\begin{equation}\label{eq_doubleind}
 1 \leq \lvert d(\bx',\by') \rvert < jn,\quad\textrm{ or } \quad \big[ \bx' = (m',n) \textrm{ and } \by = (-j,0) \textrm{ and } 0 < m' < m \big].
\end{equation}
We now split $\bx = \ba + \bb = (1,0) + (m-1,n)$, and recall $j,m,n \geq 2$ and $d=jn$.

The pair $(\ba,\bb)$ is good since both $x$-coordinates are positive, so $\bw_\ba$ and $\bw_\bb$ are in the same triangular piece.
The pair $(\by,\ba)$ is good by relation \eqref{eq_rel4}.
The pair $(\by,\bb)$ is good by the inductive assumption on $m$.
The pair $(\by+\ba,\bb)$ is good since $1 \leq \lvert d(\by+\ba,\bb)\rvert  = (j-1)n < d = jn$.
The pair $(\by+\bb,\ba)$ is good since $1 \leq \lvert d(\by+\bb,\ba) \rvert = n < d = jn$.
This finishes the inductive step for $m$, which finishes subcase (2) and Case 2.

\begin{equation*}
\begin{array}{cc|cc|c|l}
\textrm{pair}& &\textrm{vectors}& &\lvert \det\rvert &
\textrm{good because}\\
\hline
\bx&\by&(m,n)&(-j,0)&jn&\textrm{argument below}\\
\hline
\ba&\bb&(1,0)&(m-1,n)&-&\textrm{both $x$-coordinates strictly positive}\\
\by&\ba&(-j,0)&(1,0)&-&\textrm{relation \eqref{eq_rel4}}\\
\by&\bb&(-j,0)&(m-1,n)&-&\textrm{assumption \eqref{eq_doubleind}}\\
\by+\ba&\bb&(-j+1,0)&(m-1,n)&(j-1)n&0 < (j-1)n < jn \textrm{ since } j,n \geq 2\\
\by+\bb&\ba&(m-j-1,n)&(1,0)&n&0<n<jn \textrm{ since } j \geq 2
\end{array}
\end{equation*}

\end{enumerate}

\noindent \textbf{Case 3:} $j,k \geq 1$. By the symmetry \eqref{eq_anti} and Cases 1 and 2 we may also assume that $m,n \geq 1$
\[
\by = (-j,k),\quad \bx = (m,n),\quad j,k,m,n \geq 1,\quad d = jn+km.
\]
We now split into subcases.
\begin{enumerate}
\item $j = m = 1$, so $\by = (-1,k)$ and $\bx = (1,n)$.\\
We proceed by induction on $n$, and the base case $n=1$ is relation \eqref{eq_rel3}.

We now prove the inductive step, and may assume
\[
\by = (-1,k),\quad \bx = (1,n),\quad k \geq 1,\quad n \geq 2,\quad d = k+n.
\]
We split $\bx = \ba + \bb = (1,1) + (0,n-1)$.

The pair $(\ba,\bb)$ is good because $1 \leq d(\ba,\bb) = n-1 < k+n = d$ (we used $n \geq 2$). The pair $(\by,\ba)$ is good because $1 \leq d(\by,\ba) = k+1 < k+n$ (we used $n \geq 2$). The pair $(\by,\bb)$ is good because $1 \leq d(\by,\bb) = n-1 < n+k$. The pair $(\by+\ba,\bb)$ is good because both $x$-coordinates are 0. The pair $(\by+\bb,\ba)$ is good by relation \eqref{eq_rel3}. This finishes the induction and the proof of Subcase (1).
\begin{equation*}
\begin{array}{cc|cc|c|l}
\textrm{pair}& &\textrm{vectors}& &\lvert \det\rvert &
\textrm{good because}\\
\hline
\bx&\by&(1,n)&(-1,k)&k+n&\textrm{argument below}\\
\hline
\ba&\bb&(1,1)&(0,n-1)&n-1&0<n-1<k+n \textrm{ since } n \geq 2\\
\by&\ba&(-1,k)&(1,1)&k+1&0<k+1<k+n\textrm{ since }n \geq 2\\
\by&\bb&(-1,k)&(0,n-1)&n-1&0<n-1<k+n\textrm{ since }n \geq 2\\
\by+\ba&\bb&(0,k+1)&(0,n-1)&-&\textrm{both $x$-coordinates are 0}\\
\by+\bb&\ba&(-1,k+n-1)&(1,1)&-&\textrm{relation \eqref{eq_rel3}}
\end{array}
\end{equation*}

\item $m \geq 2$. (If $m=1$ and $j \geq 2$ we can apply the symmetry \eqref{eq_anti} and switch $\bx$ and $\by$.)\\
We split $\bx = \ba + \bb = (1,0) + (m-1,n)$.

The pair $(\ba,\bb)$ is good because both $x$-coordinates are strictly positive.
The pair $(\by,\ba)$ is good by induction because $1 \leq \lvert d(\by,\ba) \rvert = k < jn+mk = d$ (we used $j,k,n \geq 1$).
The pair $(\by,\bb)$ is good by induction because $1 \leq \lvert d(\by,\bb) \rvert = k(m-1)+jn < d$ (we used $j,k,n \geq 1$).
The pair $(\by+\ba,\bb)$ is good by induction because $1 \leq \lvert d(\by+\ba,\bb) \rvert = (j-1)n + k(m-1) < d$ (we used $m \geq 2$ and $k \geq 1$).
The pair $(\by+\bb,\ba)$ is good by induction because $1 \leq \lvert d(\by+\bb,\ba)\rvert = k + n < jn+km$. This finishes the proof of subcase (2) and Case 3, and finally finishes the proof of the proposition.
\begin{equation*}
\begin{array}{cc|cc|c|l}
\textrm{pair}& &\textrm{vectors}& &\lvert \det\rvert &
\textrm{good because}\\
\hline
\bx&\by&(m,n)&(-j,k)&d := jn+km&\textrm{argument below}\\
\hline
\ba&\bb&(1,0)&(m-1,n)&-&\textrm{$x$-coordinates strictly positive}\\
\by&\ba&(-j,k)&(1,0)&k&k<d\\
\by&\bb&(-j,k)&(m-1,n)&jn+k(m-1)&jn+k(m-1) < d \textrm{ since } m \geq 2\\
\by+\ba&\bb&(-j+1,k)&(m-1,n)&(j-1)n + k(m-1)&(j-1)n+k(m-1)<d\textrm{ since } m \geq 2\\
\by+\bb&\ba&(m-j-1,k+n)&(1,0)&k+n&k+n<d \textrm{ since } j,k,m,n \geq 1
\end{array}
\end{equation*}

\end{enumerate}

\end{proof}

\begin{corollary} \label{cor:generators}
The algebra $\mathbb{E}^{\bullet,\geq}$ is generated by the elements $\bw_{0,1}$ and
 $\bw_{k,0}$ for  $k \in \Z$.
\end{corollary}

\begin{proof}
This follows from \eqref{eq_mainrel} since the $(k,0)$ and $(0,1)$ generate the lattice $\N \times \Z$ as a monoid.
\end{proof}

%
\section{The $q$-deformed Heisenberg category}\label{sec:qheis}
%

%
\subsection{The integral Heisenberg algebra}
%
We begin by defining an integral version of the infinite Heisenberg algebra. Let $\mathfrak{h}_{\mathbb{Z}}$ be the unital ring with generators $a_n, b_n$, $n \in \mathbb{Z}_{\geq 1}$ and relations
\begin{equation*}
a_n b_m = b_m a_n + b_{m-1}a_{n-1}, \hspace{.5in} a_n a_m = a_m a_n,
\hspace{.5in}
b_n b_m = b_m b_n.
\end{equation*}
Here we adopt the convention that $a_0 = b_0 = 1$.

A diagrammatic approach to the categorification of $\mathfrak{h}$ was first developed by Khovanov ~\cite{Kh-Heis}.  He constructed a category whose Grothendieck group contains $\mathfrak{h}$.  Furthermore he conjectured that the inclusion of the Heisenberg algebra into the Grothendieck group is an isomorphism.

A $q$-deformation of the Khovanov Heisenberg category was introduced in \cite{LS13}.  We recall its definition now.

%
\subsection{Definition of $\cal{H}$}
%
We define an additive $\k[q,q^{-1}]$-linear strict monoidal category $\H$ as follows. The set of objects is generated by two objects $\P$ and $\Q$. Thus an arbitrary object of $\H$ is a finite direct sum of tensor products $\P_{\epsilon} := \P_{\epsilon_1} \otimes \cdots \otimes \P_{\epsilon_n}$, where $\epsilon = (\epsilon_1, \ldots, \epsilon_n)$ is a finite sequence of $+$ and $-$ signs and we set $\P_+:=\P$ and $\P_-:=\Q$. The unit object is $1 = \P_{\emptyset}$, which we represent diagramatically as the empty diagram.
The space of morphisms $\Hom_{\H}(\P_{\epsilon}, \P_{\epsilon'} )$
is the $\k[q, q^{-1}]$-module generated by planar diagrams modulo local relations. The diagrams are oriented compact one-manifolds immersed in the strip $\R \times  [0, 1]$, modulo rel boundary isotopies. The endpoints of the one-manifold are located at $\{1,\ldots,m\} \times \{0\}$ and $\{1,\ldots,k\}\times \{1\}$, where $m$ and $k$ are the lengths of the sequences $\epsilon$ and $\epsilon'$ respectively. The orientation of the one-manifold at the endpoints must agree with the signs in the sequences $\epsilon$ and $\epsilon'$ and triple intersections are not allowed.  The local relations are given below.

\begin{minipage}{0.45\textwidth}
\begin{align}
\label{heis:up-double}
\hackcenter{\begin{tikzpicture}[scale=0.8]
    \draw[thick] (0,0) .. controls ++(0,.5) and ++(0,-.5) .. (.75,1);
    \draw[thick] (.75,0) .. controls ++(0,.5) and ++(0,-.5) .. (0,1);
    \draw[thick, ->] (0,1 ) .. controls ++(0,.5) and ++(0,-.5) .. (.75,2);
    \draw[thick, ->] (.75,1) .. controls ++(0,.5) and ++(0,-.5) .. (0,2);
    \node at (0,-.25) {$\;$};
    \node at (0,2.25) {$\;$};
\end{tikzpicture}}
&\;\; = \;\;
q\;
\hackcenter{\begin{tikzpicture}[scale=0.8]
    \draw[thick, ->] (0,0) -- (0,2);
    \draw[thick, ->] (.75,0) -- (.75,2);
\end{tikzpicture}}
\;\; +(q-1) \;
\hackcenter{\begin{tikzpicture}[scale=0.8]
    \draw[thick, ->] (0,0) .. controls ++(0,.75) and ++(0,-.75) .. (.75,2);
    \draw[thick, ->] (.75,0) .. controls ++(0,.75) and ++(0,-.75) .. (0,2);
\end{tikzpicture}}
\\
  \label{heis:downup}
  \hackcenter{\begin{tikzpicture}[scale=0.8]
    \draw[thick,<-] (0,0) .. controls ++(0,.5) and ++(0,-.5) .. (.75,1);
    \draw[thick] (.75,0) .. controls ++(0,.5) and ++(0,-.5) .. (0,1);
    \draw[thick, ->] (0,1 ) .. controls ++(0,.5) and ++(0,-.5) .. (.75,2);
    \draw[thick] (.75,1) .. controls ++(0,.5) and ++(0,-.5) .. (0,2);
\end{tikzpicture}}
&\;\; = \;\;
q\;\;
\hackcenter{\begin{tikzpicture}[scale=0.8]
    \draw[thick, <-] (0,0) -- (0,2);
    \draw[thick, ->] (.75,0) -- (.75,2);
\end{tikzpicture}}
\;\; - \;\;
q\;
\hackcenter{\begin{tikzpicture}[scale=0.8]
    \draw[thick, <-] (0,0) .. controls ++(0,.75) and ++(0,.75) ..(.75,0);
    \draw[thick, <-] (.75,2) .. controls ++(0,-.75) and ++(0,-.75) .. (0,2);
\end{tikzpicture}}
  \\
  \label{eq:heis-bub}
\hackcenter{\begin{tikzpicture}[scale=0.8]
\draw [shift={+(0,0)}](0,0) arc (180:360:0.5cm) [thick];
\draw [shift={+(0,0)}][->](1,0) arc (0:180:0.5cm) [thick];
\node at (0,-1) {$\;$};
\node at (0,1) {$\;$};
\end{tikzpicture}}
& \;\; = \;\; 1,
\end{align}
\end{minipage}%
\hfill
\begin{minipage}{0.45\textwidth}
\begin{align} \label{heis:up-triple}
\hackcenter{\begin{tikzpicture}[scale=0.8]
    \draw[thick, ->] (0,0) .. controls ++(0,1) and ++(0,-1) .. (1.2,2.5);
    \draw[thick, ] (.6,0) .. controls ++(0,.5) and ++(0,-.5) .. (0,1.25);
    \draw[thick, ->] (0,1.25) .. controls ++(0,.5) and ++(0,-.5) .. (0.6,2.5);
    \draw[thick, ->] (1.2,0) .. controls ++(0,1) and ++(0,-1) .. (0,2.5);
\end{tikzpicture}}
&\;\; = \;\;
\hackcenter{\begin{tikzpicture}[scale=0.8]
    \draw[thick, ->] (0,0) .. controls ++(0,1) and ++(0,-1) .. (1.2,2.5);
    \draw[thick, ] (.6,0) .. controls ++(0,.5) and ++(0,-.5) .. (1.2,1.25);
    \draw[thick, ->] (1.2,1.25) .. controls ++(0,.5) and ++(0,-.5) .. (0.6,2.5);
    \draw[thick, ->] (1.2,0) .. controls ++(0,1) and ++(0,-1) .. (0,2.5);
\end{tikzpicture}}
\\
\label{heis:up down}
  \hackcenter{\begin{tikzpicture}[scale=0.8]
    \draw[thick] (0,0) .. controls ++(0,.5) and ++(0,-.5) .. (.75,1);
    \draw[thick,<-] (.75,0) .. controls ++(0,.5) and ++(0,-.5) .. (0,1);
    \draw[thick] (0,1 ) .. controls ++(0,.5) and ++(0,-.5) .. (.75,2);
    \draw[thick, ->] (.75,1) .. controls ++(0,.5) and ++(0,-.5) .. (0,2);
\end{tikzpicture}}
&\;\; = \;\;
q\;\;
\hackcenter{\begin{tikzpicture}[scale=0.8]
    \draw[thick,->] (0,0) -- (0,2);
    \draw[thick, <-] (.75,0) -- (.75,2);
\end{tikzpicture}}
\\
\hackcenter{
\begin{tikzpicture}[scale=0.8]
\draw  [shift={+(5,0)}](0,0) .. controls (0,.5) and (.7,.5) .. (.9,0) [thick];
\draw  [shift={+(5,0)}](0,0) .. controls (0,-.5) and (.7,-.5) .. (.9,0) [thick];
\draw  [shift={+(5,0)}](1,-1) .. controls (1,-.5) .. (.9,0) [thick];
\draw  [shift={+(5,0)}](.9,0) .. controls (1,.5) .. (1,1) [->] [thick];
\end{tikzpicture}}
&\;\;=\;\; 0
\end{align}
\end{minipage}%

Let $Kar(\H)$ be the Karoubi envelope of $\H$.

\begin{remark}
In \cite{LS13} the authors write $\H'(q)$ for the category we call $\H$ and $\H(q)$ for our category $Kar(\H)$.
\end{remark}

\begin{theorem}
\cite[Theorem 4.12]{LS13}
There is an injective map $\phi \colon \mathfrak{h} \rightarrow K_0(Kar(\H))$.
\end{theorem}

%
\subsection{Deformed degenerate affine Hecke algebra in $\H$}
%

Inside $\H$ consider the element $X_i \in \End(\P^n)$, acting on the $i$th factor $\P$, as illustrated in the left hand side of (\ref{eq:curl}). In \cite{Kh-Heis} this element was studied in the $q=1$ setting and it was encoded diagrammatically by a solid dot, as shown 
\begin{equation}\label{eq:curl}
\hackcenter{\begin{tikzpicture}[scale=0.8]
\draw  (1.9,0) .. controls (1.9,.5) and (1.3,.5) .. (1.1,0) [thick];
\draw  (1.9,0) .. controls (1.9,-.5) and (1.3,-.5) .. (1.1,0) [thick];
\draw  (1,-1) .. controls (1,-.5) .. (1.1,0) [thick];
\draw  (1.1,0) .. controls (1,.5) .. (1,1) [->] [thick];
\draw  (2.5,0) node {$=:$};
\draw (3,-1) -- (3,1) [thick][->];
\filldraw [black] (3,0) circle (2pt);
\end{tikzpicture}}
\end{equation}

In \cite{LS13} it was shown that these $X_i$'s together with the finite Hecke algebra generated by the crossings $H_n\subset \End(\P^n)$ generate a copy of the $q$-deformed degenerate affine Hecke algebra $\dah_n$.  In particular, using the relations in $\H$ the equations
\begin{align} \label{eq:nil-dot}
\hackcenter{\begin{tikzpicture}[scale=0.8]
    \draw[thick, ->] (0,0) .. controls (0,.75) and (.75,.75) .. (.75,1.5)
        node[pos=.25, shape=coordinate](DOT){};
    \draw[thick, ->] (.75,0) .. controls (.75,.75) and (0,.75) .. (0,1.5);
    \filldraw  (DOT) circle (2.5pt);
\end{tikzpicture}}
\quad-\quad
\hackcenter{\begin{tikzpicture}[scale=0.8]
    \draw[thick, ->] (0,0) .. controls (0,.75) and (.75,.75) .. (.75,1.5)
        node[pos=.75, shape=coordinate](DOT){};
    \draw[thick, ->] (.75,0) .. controls (.75,.75) and (0,.75) .. (0,1.5);
    \filldraw  (DOT) circle (2.5pt);
\end{tikzpicture}}
\quad=\quad
\hackcenter{\begin{tikzpicture}[scale=0.8]
    \draw[thick, ->] (0,0) .. controls (0,.75) and (.75,.75) .. (.75,1.5);
    \draw[thick, ->] (.75,0) .. controls (.75,.75) and (0,.75) .. (0,1.5)
        node[pos=.75, shape=coordinate](DOT){};
    \filldraw  (DOT) circle (2.5pt);
\end{tikzpicture}}
\quad-\quad
\hackcenter{\begin{tikzpicture}[scale=0.8]
    \draw[thick, ->] (0,0) .. controls (0,.75) and (.75,.75) .. (.75,1.5);
    \draw[thick, ->] (.75,0) .. controls (.75,.75) and (0,.75) .. (0,1.5)
        node[pos=.25, shape=coordinate](DOT){};
      \filldraw  (DOT) circle (2.5pt);
\end{tikzpicture}}
&\quad=\quad
(q-1)\;
\hackcenter{\begin{tikzpicture}[scale=0.8]
    \draw[thick, ->] (0,0) -- (0,1.5) node[pos=.55, shape=coordinate](DOT){};;
    \draw[thick, ->] (.75,0) -- (.75,1.5);
    \filldraw  (DOT) circle (2.5pt);
\end{tikzpicture}}
\;\; + \;\;
q\;
\hackcenter{\begin{tikzpicture}[scale=0.8]
    \draw[thick, ->] (0,0) -- (0,1.5);
    \draw[thick, ->] (.75,0) -- (.75,1.5);
\end{tikzpicture}}
\end{align}
follow. More precisely, denote a crossing of the $i$th and $(i+1)$st strands by $T_i$.
The algebra generated by $ T_i $ for $i=1, \ldots, n-1$ and $ X_i $ for $ i=1,\ldots,n$ satisfying the relations in Definition~\ref{def:qdegen} is the $q$-deformed degenerate affine Hecke algebra $\dah_n$.

We now define bubbles which are endomorphisms of $\mathbf{1}=P_{+}^0$ which can be tensored with endomorphisms of $\P^n$ to give new endomorphisms. 

\begin{equation}\label{eq:bubbles}
\begin{tikzpicture}
\draw  [shift={+(0,0)}](-1,0) node {$c_n:=$};
\draw  [shift={+(0,0)}](.2,.125) node {$n$};
\filldraw [shift={+(0,0)}][black] (0,.125) circle (2pt);
\draw [shift={+(0,0)}](0,0) arc (180:360:0.5cm) [thick];
\draw [shift={+(0,0)}][<-](1,0) arc (0:180:0.5cm) [thick];
\draw  [shift={+(5,0)}](-1,0) node {$\tilde{c}_n :=$};
\draw  [shift={+(5,0)}](.2,.125) node {$n$};
\filldraw [shift={+(5,0)}][black] (0,.125) circle (2pt);
\draw [shift={+(5,0)}][->](0,0) arc (180:360:0.5cm) [thick];
\draw [shift={+(5,0)}][](1,0) arc (0:180:0.5cm) [thick];
\end{tikzpicture}
~.
\end{equation}

%
\subsection{Helpful relations in $\H$}
%

We collect below some additional relations that follow from those in the previous section.
\begin{equation} \label{heis:mix-triple}
\hackcenter{\begin{tikzpicture}[scale=0.8]
    \draw[thick, ->] (0,0) .. controls ++(0,1) and ++(0,-1) .. (1.2,2.5);
    \draw[thick,<- ] (.6,0) .. controls ++(0,.5) and ++(0,-.5) .. (0,1.25);
    \draw[thick ] (0,1.25) .. controls ++(0,.5) and ++(0,-.5) .. (0.6,2.5);
    \draw[thick, ->] (1.2,0) .. controls ++(0,1) and ++(0,-1) .. (0,2.5);
\end{tikzpicture}}
\;\; = \;\;
\hackcenter{\begin{tikzpicture}[scale=0.8]
    \draw[thick, ->] (0,0) .. controls ++(0,1) and ++(0,-1) .. (1.2,2.5);
    \draw[thick, <- ] (.6,0) .. controls ++(0,.5) and ++(0,-.5) .. (1.2,1.25);
    \draw[thick ] (1.2,1.25) .. controls ++(0,.5) and ++(0,-.5) .. (0.6,2.5);
    \draw[thick, ->] (1.2,0) .. controls ++(0,1) and ++(0,-1) .. (0,2.5);
\end{tikzpicture}}
  \;\; + q(q-1)\;\;
\hackcenter{\begin{tikzpicture}[scale=0.8]
    \draw[thick, ->] (0,0) -- (0,2.5);
    \draw[thick, <- ] (.6,0) .. controls ++(0,.75) and ++(0,.75) .. (1.2,0);
    \draw[thick, <-] (1.2,2.5) .. controls ++(0,-.75) and ++(0,-.75) .. (0.6,2.5);
\end{tikzpicture}}
\end{equation}

\subsection{Circle crossing relations}

The crossing is invertible with inverse given by
\begin{align} \label{eq:inv-crossing}
\hackcenter{\begin{tikzpicture}[scale=0.8]
    \draw[thick, ->] (0,0) .. controls (0,.75) and (.75,.75) .. (.75,1.5)
        node[pos=.25, shape=coordinate](DOT){};
    \draw[thick, ->] (.75,0) .. controls (.75,.75) and (0,.75) .. (0,1.5);
    \draw  (.375,.75) circle (4pt);
\end{tikzpicture}}
\quad := \quad
q^{-1} \;\;
\hackcenter{\begin{tikzpicture}[scale=0.8]
    \draw[thick, ->] (0,0) .. controls (0,.75) and (.75,.75) .. (.75,1.5)
        node[pos=.75, shape=coordinate](DOT){};
    \draw[thick, ->] (.75,0) .. controls (.75,.75) and (0,.75) .. (0,1.5);
\end{tikzpicture}}
\;\; -\;\; \frac{(q-1)}{q}\;\;
\hackcenter{\begin{tikzpicture}[scale=0.8]
    \draw[thick, ->] (0,0) -- (0,1.5);
    \draw[thick, ->] (.75,0)-- (.75,1.5);
\end{tikzpicture}}
\end{align}
in $\H$. One can readily check that the identities
\begin{equation} \label{heis:circle_up down}
  \hackcenter{\begin{tikzpicture}[scale=0.8]
    \draw[thick] (0,0) .. controls ++(0,.5) and ++(0,-.5) .. (.8,1);
    \draw[thick,<-] (.8,0) .. controls ++(0,.5) and ++(0,-.5) .. (0,1);
    \draw[thick] (0,1 ) .. controls ++(0,.5) and ++(0,-.5) .. (.8,2);
    \draw[thick, ->] (.8,1) .. controls ++(0,.5) and ++(0,-.5) .. (0,2);
    \draw  (.4,.5) circle (4pt);
\end{tikzpicture}}
\;\; = \;\;
  \hackcenter{\begin{tikzpicture}[scale=0.8]
    \draw[thick] (0,0) .. controls ++(0,.5) and ++(0,-.5) .. (.8,1);
    \draw[thick,<-] (.8,0) .. controls ++(0,.5) and ++(0,-.5) .. (0,1);
    \draw[thick] (0,1 ) .. controls ++(0,.5) and ++(0,-.5) .. (.8,2);
    \draw[thick, ->] (.8,1) .. controls ++(0,.5) and ++(0,-.5) .. (0,2);
    \draw  (.4,1.5) circle (4pt);
\end{tikzpicture}}
\;\; = \;\;
\;\;
\hackcenter{\begin{tikzpicture}[scale=0.8]
    \draw[thick,->] (0,0) -- (0,2);
    \draw[thick, <-] (.75,0) -- (.75,2);
\end{tikzpicture}}
\qquad \qquad
  \hackcenter{\begin{tikzpicture}[scale=0.8]
    \draw[thick] (0,0) .. controls ++(0,.5) and ++(0,-.5) .. (.8,1);
    \draw[thick,<-] (.8,0) .. controls ++(0,.5) and ++(0,-.5) .. (0,1);
    \draw[thick] (0,1 ) .. controls ++(0,.5) and ++(0,-.5) .. (.8,2);
    \draw[thick, ->] (.8,1) .. controls ++(0,.5) and ++(0,-.5) .. (0,2);
    \draw  (.4,1.5) circle (4pt);
    \draw  (.4,.5) circle (4pt);
\end{tikzpicture}}
\;\; = \;\;
q^{-1}\;\;
\hackcenter{\begin{tikzpicture}[scale=0.8]
    \draw[thick,->] (0,0) -- (0,2);
    \draw[thick, <-] (.75,0) -- (.75,2);
\end{tikzpicture}}
\;\; + \;\;
\frac{(q  -1)^2}{q^2}\;\;
\hackcenter{\begin{tikzpicture}[scale=0.8]
    \draw[thick, ->] (0,0) .. controls ++(0,.75) and ++(0,.75) ..(.75,0);
    \draw[thick, ->] (.75,2) .. controls ++(0,-.75) and ++(0,-.75) .. (0,2);
\end{tikzpicture}}
\end{equation}
\begin{equation} \label{eq:circle-left-twist}
\hackcenter{
\begin{tikzpicture}[scale=0.8]
    \draw  [thick](0,0) .. controls (0,.5) and (.7,.5) .. (.9,0);
    \draw  [thick](0,0) .. controls (0,-.5) and (.7,-.5) .. (.9,0);
    \draw  [thick](1,-1) .. controls (1,-.5) .. (.9,0);
    \draw  [thick,->](.9,0) .. controls (1,.5) .. (1,1) ;
      \draw  (.88,0) circle (4pt);
\end{tikzpicture}}
\;\; = \;\;
\;\;  -\;\; \frac{(q-1)}{q}\;\;
\hackcenter{
\begin{tikzpicture}[scale=0.8]
    \draw[thick,->]  [thick](0,-1) -- (0,1);
\end{tikzpicture}}
\end{equation}
follow from the defining relations in $\H$.

Almost all of the triple point moves one can imagine still hold,
\begin{equation} \label{eq:circ-triple}
\hackcenter{\begin{tikzpicture}[scale=0.8]
    \draw[thick, ->] (0,0) .. controls ++(0,1) and ++(0,-1) .. (1.2,2.5);
    \draw[thick, ] (.6,0) .. controls ++(0,.7) and ++(0,-.5) .. (0,1.25);
    \draw[thick, ->] (0,1.25) .. controls ++(0,.5) and ++(0,-.7) .. (0.6,2.5);
    \draw[thick, ->] (1.2,0) .. controls ++(0,1) and ++(0,-1) .. (0,2.5);
    \draw  (.25,1.75) circle (4pt);
\end{tikzpicture}}
\;\; = \;\;
\hackcenter{\begin{tikzpicture}[scale=0.8]
    \draw[thick, ->] (0,0) .. controls ++(0,1) and ++(0,-1) .. (1.2,2.5);
    \draw[thick, ] (.6,0) .. controls ++(0,.7) and ++(0,-.7) .. (1.2,1.25);
    \draw[thick, ->] (1.2,1.25) .. controls ++(0,.5) and ++(0,-.5) .. (0.6,2.5);
    \draw[thick, ->] (1.2,0) .. controls ++(0,1) and ++(0,-1) .. (0,2.5);
    \draw  (1.0,.75) circle (4pt);
\end{tikzpicture}}
\qquad \qquad
\hackcenter{\begin{tikzpicture}[scale=0.8]
    \draw[thick, ->] (0,0) .. controls ++(0,1) and ++(0,-1) .. (1.2,2.5);
    \draw[thick, ] (.6,0) .. controls ++(0,.7) and ++(0,-.5) .. (0,1.25);
    \draw[thick, ->] (0,1.25) .. controls ++(0,.5) and ++(0,-.7) .. (0.6,2.5);
    \draw[thick, ->] (1.2,0) .. controls ++(0,1) and ++(0,-1) .. (0,2.5);
    \draw  (.25,.75) circle (4pt);\draw  (.6,1.22) circle (4pt);
\end{tikzpicture}}
\;\; = \;\;
\hackcenter{\begin{tikzpicture}[scale=0.8]
    \draw[thick, ->] (0,0) .. controls ++(0,1) and ++(0,-1) .. (1.2,2.5);
    \draw[thick, ] (.6,0) .. controls ++(0,.7) and ++(0,-.7) .. (1.2,1.25);
    \draw[thick, ->] (1.2,1.25) .. controls ++(0,.5) and ++(0,-.5) .. (0.6,2.5);
    \draw[thick, ->] (1.2,0) .. controls ++(0,1) and ++(0,-1) .. (0,2.5);
    \draw  (.6,1.3) circle (4pt);\draw  (1.0,1.75) circle (4pt);
\end{tikzpicture}}
\qquad \qquad
\hackcenter{\begin{tikzpicture}[scale=0.8]
    \draw[thick, ->] (0,0) .. controls ++(0,1) and ++(0,-1) .. (1.2,2.5);
    \draw[thick, ] (.6,0) .. controls ++(0,.7) and ++(0,-.5) .. (0,1.25);
    \draw[thick, ->] (0,1.25) .. controls ++(0,.5) and ++(0,-.7) .. (0.6,2.5);
    \draw[thick, ->] (1.2,0) .. controls ++(0,1) and ++(0,-1) .. (0,2.5);
    \draw  (.25,.75) circle (4pt);
    \draw  (.6,1.22) circle (4pt);
    \draw  (.25,1.75) circle (4pt);
\end{tikzpicture}}
\;\; = \;\;
\hackcenter{\begin{tikzpicture}[scale=0.8]
    \draw[thick, ->] (0,0) .. controls ++(0,1) and ++(0,-1) .. (1.2,2.5);
    \draw[thick, ] (.6,0) .. controls ++(0,.7) and ++(0,-.7) .. (1.2,1.25);
    \draw[thick, ->] (1.2,1.25) .. controls ++(0,.5) and ++(0,-.5) .. (0.6,2.5);
    \draw[thick, ->] (1.2,0) .. controls ++(0,1) and ++(0,-1) .. (0,2.5);
    \draw  (1.0,.75) circle (4pt);
    \draw (.6,1.3) circle (4pt);
    \draw  (1.0,1.75) circle (4pt);
\end{tikzpicture}}
\end{equation}

\begin{equation}
\hackcenter{\begin{tikzpicture}[scale=0.8]
    \draw[thick, ->] (0,0) .. controls ++(0,1) and ++(0,-1) .. (1.2,2.5);
    \draw[thick,<- ] (.6,0) .. controls ++(0,.5) and ++(0,-.5) .. (0,1.25);
    \draw[thick ] (0,1.25) .. controls ++(0,.5) and ++(0,-.5) .. (0.6,2.5);
    \draw[thick, ->] (1.2,0) .. controls ++(0,1) and ++(0,-1) .. (0,2.5);
     \draw  (.6,1.25) circle (4pt);
\end{tikzpicture}}
\;\; = \;\;
\hackcenter{\begin{tikzpicture}[scale=0.8]
    \draw[thick, ->] (0,0) .. controls ++(0,1) and ++(0,-1) .. (1.2,2.5);
    \draw[thick, <- ] (.6,0) .. controls ++(0,.5) and ++(0,-.5) .. (1.2,1.25);
    \draw[thick ] (1.2,1.25) .. controls ++(0,.5) and ++(0,-.5) .. (0.6,2.5);
    \draw[thick, ->] (1.2,0) .. controls ++(0,1) and ++(0,-1) .. (0,2.5);
     \draw  (.6,1.25) circle (4pt);
\end{tikzpicture}}
\qquad \qquad
\hackcenter{\begin{tikzpicture}[scale=0.8]
    \draw[thick, ->] (0,0) .. controls ++(0,1) and ++(0,-1) .. (1.2,2.5);
    \draw[thick, -> ] (.6,0) .. controls ++(0,.5) and ++(0,-.5) .. (0,1.25);
    \draw[thick ] (0,1.25) .. controls ++(0,.5) and ++(0,-.5) .. (0.6,2.5);
    \draw[thick, <-] (1.2,0) .. controls ++(0,1) and ++(0,-1) .. (0,2.5);
     \draw  (.25,.75) circle (4pt);
\end{tikzpicture}}
\;\; = \;\;
\hackcenter{\begin{tikzpicture}[scale=0.8]
    \draw[thick, ->] (0,0) .. controls ++(0,1) and ++(0,-1) .. (1.2,2.5);
    \draw[thick, -> ] (.6,0) .. controls ++(0,.5) and ++(0,-.5) .. (1.2,1.25);
    \draw[thick ] (1.2,1.25) .. controls ++(0,.5) and ++(0,-.5) .. (0.6,2.5);
    \draw[thick, <-] (1.2,0) .. controls ++(0,1) and ++(0,-1) .. (0,2.5);
     \draw  (1,1.75) circle (4pt);
\end{tikzpicture}}
\end{equation}
however, the triple point move below
\begin{equation}
\hackcenter{\begin{tikzpicture}[scale=0.8]
    \draw[thick, ->] (0,0) .. controls ++(0,1) and ++(0,-1) .. (1.2,2.5);
    \draw[thick, ] (.6,0) .. controls ++(0,.5) and ++(0,-.5) .. (0,1.25);
    \draw[thick, ->] (0,1.25) .. controls ++(0,.5) and ++(0,-.5) .. (0.6,2.5);
    \draw[thick, ->] (1.2,0) .. controls ++(0,1) and ++(0,-1) .. (0,2.5);
    \draw  (.6,1.25) circle (4pt);
\end{tikzpicture}}
\;\; \mathbf{\neq} \;\;
\hackcenter{\begin{tikzpicture}[scale=0.8]
    \draw[thick, ->] (0,0) .. controls ++(0,1) and ++(0,-1) .. (1.2,2.5);
    \draw[thick, ] (.6,0) .. controls ++(0,.5) and ++(0,-.5) .. (1.2,1.25);
    \draw[thick, ->] (1.2,1.25) .. controls ++(0,.5) and ++(0,-.5) .. (0.6,2.5);
    \draw[thick, ->] (1.2,0) .. controls ++(0,1) and ++(0,-1) .. (0,2.5);
    \draw  (.6,1.25) circle (4pt);
\end{tikzpicture}}
\end{equation}
does {\bf not} hold.

Dots slide through circle crossings via a similar formula to \eqref{eq:nil-dot}.
\begin{align} \label{eq:nil-circle-dotslide}
\hackcenter{\begin{tikzpicture}[scale=0.8]
    \draw[thick, ->] (0,0) .. controls (0,.75) and (.75,.75) .. (.75,1.5)
        node[pos=.25, shape=coordinate](DOT){};
    \draw[thick, ->] (.75,0) .. controls (.75,.75) and (0,.75) .. (0,1.5);
    \filldraw  (DOT) circle (2.5pt);
    \draw  (.375,.75) circle (4pt);
\end{tikzpicture}}
\quad-\quad
\hackcenter{\begin{tikzpicture}[scale=0.8]
    \draw[thick, ->] (0,0) .. controls (0,.75) and (.75,.75) .. (.75,1.5)
        node[pos=.75, shape=coordinate](DOT){};
    \draw[thick, ->] (.75,0) .. controls (.75,.75) and (0,.75) .. (0,1.5);
    \filldraw  (DOT) circle (2.5pt);
    \draw  (.375,.75) circle (4pt);
\end{tikzpicture}}
\quad=\quad
\hackcenter{\begin{tikzpicture}[scale=0.8]
    \draw[thick, ->] (0,0) .. controls (0,.75) and (.75,.75) .. (.75,1.5);
    \draw[thick, ->] (.75,0) .. controls (.75,.75) and (0,.75) .. (0,1.5)
        node[pos=.75, shape=coordinate](DOT){};
    \filldraw  (DOT) circle (2.5pt);
    \draw  (.375,.75) circle (4pt);
\end{tikzpicture}}
\quad-\quad
\hackcenter{\begin{tikzpicture}[scale=0.8]
    \draw[thick, ->] (0,0) .. controls (0,.75) and (.75,.75) .. (.75,1.5);
    \draw[thick, ->] (.75,0) .. controls (.75,.75) and (0,.75) .. (0,1.5)
        node[pos=.25, shape=coordinate](DOT){};
      \filldraw  (DOT) circle (2.5pt);
    \draw  (.375,.75) circle (4pt);
\end{tikzpicture}}
&\quad=\quad
 \;\; \frac{(q-1)}{q}\;\;
\hackcenter{\begin{tikzpicture}[scale=0.8]
    \draw[thick, ->] (0,0) -- (0,1.5);;
    \draw[thick, ->] (.75,0) -- (.75,1.5)node[pos=.55, shape=coordinate](DOT){};
    \filldraw  (DOT) circle (2.5pt);
\end{tikzpicture}}
\;\; + \;\;
\hackcenter{\begin{tikzpicture}[scale=0.8]
    \draw[thick, ->] (0,0) -- (0,1.5);
    \draw[thick, ->] (.75,0) -- (.75,1.5);
\end{tikzpicture}}
\end{align}

%
\subsection{Star calculus}
%

Motivated by the relationship between the algebra $\dah_n$ and $\ah_n^+$, it will be useful to make the following definition
\begin{equation}\label{eq:def-star}
\hackcenter{\begin{tikzpicture}[scale=0.8]
    \draw[thick, ->] (0,0) -- (0,1.5) node[pos=.55, shape=coordinate](DOT){};;
    \node[star,star points=6,star point ratio=0.5, fill] at  (DOT) {} ;
\end{tikzpicture}}
\;\; := \;\;
 \frac{1}{q-1} \;\;
\hackcenter{\begin{tikzpicture}[scale=0.8]
    \draw[thick, ->] (0,0) -- (0,1.5) node[pos=.55, shape=coordinate](DOT){};;
    \filldraw  (DOT) circle (2.5pt);
\end{tikzpicture}} \;\; + \;\; \frac{q}{(q-1)^2} \;\;
\hackcenter{\begin{tikzpicture}[scale=0.8]
    \draw[thick, ->] (0,0) -- (0,1.5) ;
\end{tikzpicture}}
\end{equation}
which we will refer to as a star dot.
These new generators interact with our previous generators via the following equations which are straightforward to establish.
\begin{align} \label{eq:nil-star}
\hackcenter{\begin{tikzpicture}[scale=0.8]
    \draw[thick, ->] (0,0) .. controls (0,.75) and (.75,.75) .. (.75,1.5)
        node[pos=.25, shape=coordinate](DOT){};
    \draw[thick, ->] (.75,0) .. controls (.75,.75) and (0,.75) .. (0,1.5);
   \node[star,star points=6,star point ratio=0.5, fill] at  (DOT) {} ;
\end{tikzpicture}}
\quad-\quad
\hackcenter{\begin{tikzpicture}[scale=0.8]
    \draw[thick, ->] (0,0) .. controls (0,.75) and (.75,.75) .. (.75,1.5)
        node[pos=.75, shape=coordinate](DOT){};
    \draw[thick, ->] (.75,0) .. controls (.75,.75) and (0,.75) .. (0,1.5);
    \node[star,star points=6,star point ratio=0.5, fill] at  (DOT) {} ;
\end{tikzpicture}}
\quad=\quad
\hackcenter{\begin{tikzpicture}[scale=0.8]
    \draw[thick, ->] (0,0) .. controls (0,.75) and (.75,.75) .. (.75,1.5);
    \draw[thick, ->] (.75,0) .. controls (.75,.75) and (0,.75) .. (0,1.5)
        node[pos=.75, shape=coordinate](DOT){};
    \node[star,star points=6,star point ratio=0.5, fill] at  (DOT) {} ;
\end{tikzpicture}}
\quad-\quad
\hackcenter{\begin{tikzpicture}[scale=0.8]
    \draw[thick, ->] (0,0) .. controls (0,.75) and (.75,.75) .. (.75,1.5);
    \draw[thick, ->] (.75,0) .. controls (.75,.75) and (0,.75) .. (0,1.5)
        node[pos=.25, shape=coordinate](DOT){};
      \node[star,star points=6,star point ratio=0.5, fill] at  (DOT) {} ;
\end{tikzpicture}}
&\quad=\quad
(q-1)\;
\hackcenter{\begin{tikzpicture}[scale=0.8]
    \draw[thick, ->] (0,0) -- (0,1.5) node[pos=.55, shape=coordinate](DOT){};;
    \draw[thick, ->] (.75,0) -- (.75,1.5);
    \node[star,star points=6,star point ratio=0.5, fill] at  (DOT) {} ;
\end{tikzpicture}}
\\ \label{eq:nil-circle-star}
\hackcenter{\begin{tikzpicture}[scale=0.8]
    \draw[thick, ->] (0,0) .. controls (0,.75) and (.75,.75) .. (.75,1.5)
        node[pos=.25, shape=coordinate](DOT){};
    \draw[thick, ->] (.75,0) .. controls (.75,.75) and (0,.75) .. (0,1.5);
   \node[star,star points=6,star point ratio=0.5, fill] at  (DOT) {} ;
     \draw  (.375,.75) circle (4pt);
\end{tikzpicture}}
\quad-\quad
\hackcenter{\begin{tikzpicture}[scale=0.8]
    \draw[thick, ->] (0,0) .. controls (0,.75) and (.75,.75) .. (.75,1.5)
        node[pos=.75, shape=coordinate](DOT){};
    \draw[thick, ->] (.75,0) .. controls (.75,.75) and (0,.75) .. (0,1.5);
    \node[star,star points=6,star point ratio=0.5, fill] at  (DOT) {} ;
    \draw  (.375,.75) circle (4pt);
\end{tikzpicture}}
\quad=\quad
\hackcenter{\begin{tikzpicture}[scale=0.8]
    \draw[thick, ->] (0,0) .. controls (0,.75) and (.75,.75) .. (.75,1.5);
    \draw[thick, ->] (.75,0) .. controls (.75,.75) and (0,.75) .. (0,1.5)
        node[pos=.75, shape=coordinate](DOT){};
    \node[star,star points=6,star point ratio=0.5, fill] at  (DOT) {} ;
    \draw  (.375,.75) circle (4pt);
\end{tikzpicture}}
\quad-\quad
\hackcenter{\begin{tikzpicture}[scale=0.8]
    \draw[thick, ->] (0,0) .. controls (0,.75) and (.75,.75) .. (.75,1.5);
    \draw[thick, ->] (.75,0) .. controls (.75,.75) and (0,.75) .. (0,1.5)
        node[pos=.25, shape=coordinate](DOT){};
      \node[star,star points=6,star point ratio=0.5, fill] at  (DOT) {} ;
    \draw  (.375,.75) circle (4pt);
\end{tikzpicture}}
&\quad=\quad
\;\;   \frac{(q-1)}{q}\;\;
\hackcenter{\begin{tikzpicture}[scale=0.8]
    \draw[thick, ->] (0,0) -- (0,1.5);;
    \draw[thick, ->] (.75,0) -- (.75,1.5)node[pos=.55, shape=coordinate](DOT){};
    \node[star,star points=6,star point ratio=0.5, fill] at  (DOT) {} ;
\end{tikzpicture}}
\end{align}

\begin{align} \label{eq:nil-cool}
  \hackcenter{\begin{tikzpicture}[scale=0.8]
    \draw[thick, ->] (0,0) .. controls (0,.75) and (.75,.75) .. (.75,1.5)
        node[pos=.25, shape=coordinate](DOT){};
    \draw[thick, ->] (.75,0) .. controls (.75,.75) and (0,.75) .. (0,1.5);
    \node[star,star points=6,star point ratio=0.5, fill] at  (DOT) {} ;
    \draw  (.375,.75) circle (4pt);
\end{tikzpicture}}
\;\; &= \;\;
q^{-1} \;\;\hackcenter{\begin{tikzpicture}[scale=0.8]
    \draw[thick, ->] (0,0) .. controls (0,.75) and (.75,.75) .. (.75,1.5)
        node[pos=.75, shape=coordinate](DOT){};
    \draw[thick, ->] (.75,0) .. controls (.75,.75) and (0,.75) .. (0,1.5);
    \node[star,star points=6,star point ratio=0.5, fill] at  (DOT) {} ;
\end{tikzpicture}} \qquad \qquad
\hackcenter{\begin{tikzpicture}[scale=0.8]
    \draw[thick, ->] (0,0) .. controls (0,.75) and (.75,.75) .. (.75,1.5);
    \draw[thick, ->] (.75,0) .. controls (.75,.75) and (0,.75) .. (0,1.5)
        node[pos=.75, shape=coordinate](DOT){};
    \node[star,star points=6,star point ratio=0.5, fill] at  (DOT) {} ;
    \draw  (.375,.75) circle (4pt);
\end{tikzpicture}}
\;\; = \;\; q^{-1}
\hackcenter{\begin{tikzpicture}[scale=0.8]
    \draw[thick, ->] (0,0) .. controls (0,.75) and (.75,.75) .. (.75,1.5);
    \draw[thick, ->] (.75,0) .. controls (.75,.75) and (0,.75) .. (0,1.5)
        node[pos=.25, shape=coordinate](DOT){};
   \node[star,star points=6,star point ratio=0.5, fill] at  (DOT) {} ;
\end{tikzpicture}}
\end{align}
\begin{align} \label{eq:ind-star}
\hackcenter{\begin{tikzpicture}[scale=0.8]
    \draw[thick, ->] (0,0) .. controls (0,.75) and (.75,.75) .. (.75,1.5)
        node[pos=.25, shape=coordinate](DOT){};
    \draw[thick, ->] (.75,0) .. controls (.75,.75) and (0,.75) .. (0,1.5);
   \node[star,star points=6,star point ratio=0.5, fill] at  (DOT) {} ;
   \node at (-.25,.35) {$\scs k$};
\end{tikzpicture}}
\quad-\quad
\hackcenter{\begin{tikzpicture}[scale=0.8]
    \draw[thick, ->] (0,0) .. controls (0,.75) and (.75,.75) .. (.75,1.5)
        node[pos=.75, shape=coordinate](DOT){};
    \draw[thick, ->] (.75,0) .. controls (.75,.75) and (0,.75) .. (0,1.5);
    \node[star,star points=6,star point ratio=0.5, fill] at  (DOT) {} ;
    \node at (1,1.15) {$\scs k$};
\end{tikzpicture}}
\quad=\quad
\hackcenter{\begin{tikzpicture}[scale=0.8]
    \draw[thick, ->] (0,0) .. controls (0,.75) and (.75,.75) .. (.75,1.5);
    \draw[thick, ->] (.75,0) .. controls (.75,.75) and (0,.75) .. (0,1.5)
        node[pos=.75, shape=coordinate](DOT){};
    \node[star,star points=6,star point ratio=0.5, fill] at  (DOT) {} ;
 \node at (-.25,1.15) {$\scs k$};
\end{tikzpicture}}
\quad-\quad
\hackcenter{\begin{tikzpicture}[scale=0.8]
    \draw[thick, ->] (0,0) .. controls (0,.75) and (.75,.75) .. (.75,1.5);
    \draw[thick, ->] (.75,0) .. controls (.75,.75) and (0,.75) .. (0,1.5)
        node[pos=.25, shape=coordinate](DOT){};
      \node[star,star points=6,star point ratio=0.5, fill] at  (DOT) {} ;
      \node at (1,.35) {$\scs k$};
\end{tikzpicture}}
&\quad=\quad
(q-1)\;
\sum_{i=1}^{k}
\hackcenter{\begin{tikzpicture}[scale=0.8]
    \draw[thick, ->] (0,0) -- (0,1.5) node[pos=.55, shape=coordinate](DOT){};;
    \draw[thick, ->] (.75,0) -- (.75,1.5)node[pos=.55, shape=coordinate](DOT2){};
    \node[star,star points=6,star point ratio=0.5, fill] at  (DOT) {} ;
    \node[star,star points=6,star point ratio=0.5, fill] at  (DOT2) {} ;
          \node at (1.2,1.15) {$\scs k-i$};
      \node at (-.3,1.15) {$\scs i$};
\end{tikzpicture}}
\end{align}

\begin{equation}
 \hackcenter{ \begin{tikzpicture}
 \draw  [shift={+(0,0)}](.3,.125) node {};
 \node[star,star points=6,star point ratio=0.5, fill] at  (1,.125) {} ;
\draw  (0,0) arc (180:360:0.5cm) [thick];
\draw[->](1,0) arc (0:180:0.5cm) [thick];
\end{tikzpicture} }
\;\; = \;\;
\frac{q}{(q-1)^2}
\end{equation}
\begin{align}\label{eq:star-up-r2}
  \hackcenter{\begin{tikzpicture}[scale=0.8]
    \draw[thick] (0,0) .. controls ++(0,.5) and ++(0,-.5) .. (.8,1);
    \draw[thick] (.8,0) .. controls ++(0,.5) and ++(0,-.5) .. (0,1);
    \draw[thick,->] (0,1 ) .. controls ++(0,.5) and ++(0,-.5) .. (.8,2);
    \draw[thick, ->] (.8,1) .. controls ++(0,.5) and ++(0,-.5) .. (0,2);
    \node[star,star points=6,star point ratio=0.5, fill] at  (.8,1) {} ;;
\end{tikzpicture}}
\;\; = \;\; q\;
\hackcenter{\begin{tikzpicture}[scale=0.8]
    \draw[thick,->] (0,0) -- (0,2);
    \draw[thick, ->] (.75,0) -- (.75,2);
    \node[star,star points=6,star point ratio=0.5, fill] at  (0,1) {} ;;
\end{tikzpicture}}
\end{align}
\begin{align} \label{eq:star-curl}
\hackcenter{
\begin{tikzpicture}[scale=0.8]
    \draw  [thick](0,0) .. controls (0,.5) and (.7,.5) .. (.9,0);
    \draw  [thick](0,0) .. controls (0,-.5) and (.7,-.5) .. (.9,0);
    \draw  [thick](1,-1) .. controls (1,-.5) .. (.9,0);
    \draw  [thick,->](.9,0) .. controls (1,.5) .. (1,1) ;
      \draw  (.88,0) circle (4pt);
      \node[star,star points=6,star point ratio=0.5, fill] at  (0,0) {} ;;
\end{tikzpicture}}
\;\; = \;\;
 0 \qquad \quad
\hackcenter{
\begin{tikzpicture}[scale=0.8]
    \draw  [thick](0,0) .. controls (0,.5) and (.7,.5) .. (.9,0);
    \draw  [thick](0,0) .. controls (0,-.5) and (.7,-.5) .. (.9,0);
    \draw  [thick](1,-1) .. controls (1,-.5) .. (.9,0);
    \draw  [thick,->](.9,0) .. controls (1,.5) .. (1,1) ;
      \node[star,star points=6,star point ratio=0.5, fill] at  (0,0) {} ;;
\end{tikzpicture}}
\;\; = \;\;
\frac{q}{(q-1)}\;\;\;
\hackcenter{
\begin{tikzpicture}[scale=0.8]
    \draw[thick,->]  [thick](0,-1) -- (0,1);
\end{tikzpicture}}
\end{align}

\begin{align}\label{eq:star-up-circler2}
  \hackcenter{\begin{tikzpicture}[scale=0.8]
    \draw[thick] (0,0) .. controls ++(0,.5) and ++(0,-.5) .. (.8,1);
    \draw[thick] (.8,0) .. controls ++(0,.5) and ++(0,-.5) .. (0,1);
    \draw[thick,->] (0,1 ) .. controls ++(0,.5) and ++(0,-.5) .. (.8,2);
    \draw[thick, ->] (.8,1) .. controls ++(0,.5) and ++(0,-.5) .. (0,2);
    \draw  (.4,1.5) circle (4pt);
    \node[star,star points=6,star point ratio=0.5, fill] at  (.8,1) {} ;;
\end{tikzpicture}}
\;\; = \;\;
\hackcenter{\begin{tikzpicture}[scale=0.8]
    \draw[thick,->] (0,0) -- (0,2);
    \draw[thick, ->] (.75,0) -- (.75,2);
    \node[star,star points=6,star point ratio=0.5, fill] at  (0,1) {} ;;
\end{tikzpicture}}
\;\; -\;\;
(q -1)\;\;
\hackcenter{\begin{tikzpicture}[scale=0.8]
    \draw[thick, ->] (0,0) .. controls (0,.75) and (.75,.75) .. (.75,1.5)
        node[pos=.25, shape=coordinate](DOT){};
    \draw[thick, ->] (.75,0) .. controls (.75,.75) and (0,.75) .. (0,1.5);
   \node[star,star points=6,star point ratio=0.5, fill] at  (DOT) {} ;
\draw  (.375,.75) circle (4pt);
\end{tikzpicture}}
\;\; = \;\;
\hackcenter{\begin{tikzpicture}[scale=0.8]
    \draw[thick,->] (0,0) -- (0,2);
    \draw[thick, ->] (.75,0) -- (.75,2);
    \node[star,star points=6,star point ratio=0.5, fill] at  (0,1) {} ;;
\end{tikzpicture}}
\;\; -\;\;
\frac{(q -1)}{q}\;\;
\hackcenter{\begin{tikzpicture}[scale=0.8]
    \draw[thick, ->] (0,0) .. controls (0,.75) and (.75,.75) .. (.75,1.5)
        node[pos=.75, shape=coordinate](DOT){};
    \draw[thick, ->] (.75,0) .. controls (.75,.75) and (0,.75) .. (0,1.5);
   \node[star,star points=6,star point ratio=0.5, fill] at  (DOT) {} ;
\end{tikzpicture}}
\end{align}

\begin{equation}
  \hackcenter{\begin{tikzpicture}[scale=0.8]
    \draw[thick,<-] (0,0) .. controls ++(0,.5) and ++(0,-.5) .. (.75,1);
    \draw[thick] (.75,0) .. controls ++(0,.5) and ++(0,-.5) .. (0,1);
    \draw[thick, ->] (0,1 ) .. controls ++(0,.5) and ++(0,-.5) .. (.75,2);
    \draw[thick] (.75,1) .. controls ++(0,.5) and ++(0,-.5) .. (0,2);
    \draw  (.375,.5) circle (4pt);
\end{tikzpicture}}
\;\; = \;\;
\;\;
\hackcenter{\begin{tikzpicture}[scale=0.8]
    \draw[thick, <-] (0,0) -- (0,2);
    \draw[thick, ->] (.75,0) -- (.75,2);
\end{tikzpicture}}
\;\; - \;\; \frac{(q-1)^2}{q}
\;
\hackcenter{\begin{tikzpicture}[scale=0.8]
    \draw[thick, <-] (0,0) .. controls ++(0,.75) and ++(0,.75) ..(.75,0);
    \draw[thick, <-] (.75,2) .. controls ++(0,-.75) and ++(0,-.75) .. (0,2) node[pos=.55, shape=coordinate](DOT){};;
    \node[star,star points=6,star point ratio=0.5, fill] at  (DOT) {} ;
\end{tikzpicture}}
\end{equation}

\begin{equation} \label{eq:doub-star}
  \hackcenter{\begin{tikzpicture}[scale=0.8]
    \draw[thick,<-] (0,0) .. controls ++(0,.5) and ++(0,-.5) .. (.75,1);
    \draw[thick] (.75,0) .. controls ++(0,.5) and ++(0,-.5) .. (0,1);
    \draw[thick, ->] (0,1 ) .. controls ++(0,.5) and ++(0,-.5) .. (.75,2);
    \draw[thick] (.75,1) .. controls ++(0,.5) and ++(0,-.5) .. (0,2);
    \node[star,star points=6,star point ratio=0.5, fill] at  (.8,1) {} ;;
\end{tikzpicture}}
\;\; = \;\;
q\;\;
\hackcenter{\begin{tikzpicture}[scale=0.8]
    \draw[thick, <-] (0,0) -- (0,2);
    \draw[thick, ->] (.75,0) -- (.75,2);
    \node[star,star points=6,star point ratio=0.5, fill] at  (0,1) {} ;;
\end{tikzpicture}}
\;\; - \;\;
(q-1)^2\;
\hackcenter{\begin{tikzpicture}[scale=0.8]
    \draw[thick, <-] (0,0) .. controls ++(0,.75) and ++(0,.75) ..(.8,0);
    \draw[thick, <-] (.8,2) .. controls ++(0,-.75) and ++(0,-.75) .. (0,2);
    \node[star,star points=6,star point ratio=0.5, fill] at  (.4,1.5) {} ;;
    \node[star,star points=6,star point ratio=0.5, fill] at  (.4,.5) {} ;;
\end{tikzpicture}}
\end{equation}

%
\subsubsection{Closed Star calculus}
%
We introduce starred bubbles
\begin{equation}\label{eq:bubbles}
\begin{tikzpicture}
\draw  [shift={+(0,0)}](-1,0) node {$C_n:=$};
\draw  [shift={+(0,0)}](.3,.125) node {$n$};
 \node[star,star points=6,star point ratio=0.5, fill] at  (0,.125) {} ;
\draw [shift={+(0,0)}](0,0) arc (180:360:0.5cm) [thick];
\draw [shift={+(0,0)}][<-](1,0) arc (0:180:0.5cm) [thick];
\draw  [shift={+(5,0)}](-1,0) node {$\tilde{C}_n:=$};
\draw  [shift={+(5,0)}](.3,.125) node {$n$};
\node[star,star points=6,star point ratio=0.5, fill] at  (5,.125) {} ;
\draw [shift={+(5,0)}][->](0,0) arc (180:360:0.5cm) [thick];
\draw [shift={+(5,0)}][](1,0) arc (0:180:0.5cm) [thick];
\end{tikzpicture}
\end{equation}
which are endomorphisms of the monoidal unit $\1$.   The figure eight can be reduced in two possible ways by sliding either the $a$ stars or the $b$ stars
\begin{equation} \label{eq:figeight}
 \hackcenter{\begin{tikzpicture}[scale=0.8]
    \draw[thick] (0,2) .. controls ++(0,.45) and ++(0,.45) .. (.8,2)
        node[pos=.5, shape=coordinate](DOT2){};
    \node[star,star points=6,star point ratio=0.5, fill] at  (DOT2) {} ;
    \draw[thick] (.8,1) .. controls ++(0,-.45) and ++(0,-.45) .. (0,1);
    \draw[thick,->] (0,1) .. controls ++(0,.5) and ++(0,-.5) .. (.8,2)
           node[pos=.0, shape=coordinate](DOT){};
    \node[star,star points=6,star point ratio=0.5, fill] at  (DOT) {};
    \node at (-.5,.7) {$\scs b$};;
    \draw[thick, ->] (0,2).. controls ++(0,-.5) and ++(0,.5) .. (.8,1);
    \node at (.8,2.45) {$\scs a$};;
\end{tikzpicture}}
\;\; = \;\; (q-1)\sum_{i=1}^{a} \;\;
 \hackcenter{ \begin{tikzpicture}
 \draw  (.7,.4) node {$\scs i$};
 \node[star,star points=6,star point ratio=0.5, fill] at  (.7,.125) {} ;
\draw  (0,0) arc (180:360:0.35cm) [thick];
\draw[->](.7,0) arc (0:180:0.35cm) [thick];
\end{tikzpicture} }
\;
 \hackcenter{ \begin{tikzpicture}
 \draw  (-.3,.4) node {$\scs a+b-i$};
 \node[star,star points=6,star point ratio=0.5, fill] at  (0,.125) {} ;
\draw  (0,0) arc (180:360:0.35cm) [thick];
\draw[<-](.7,0) arc (0:180:0.35cm) [thick];
\end{tikzpicture} }
\;\; = \;\;
 \hackcenter{ \begin{tikzpicture}
 \draw  (-.2,.4) node {$\scs a+b$};
 \node[star,star points=6,star point ratio=0.5, fill] at  (0,.125) {} ;
\draw  (0,0) arc (180:360:0.35cm) [thick];
\draw[<-](.7,0) arc (0:180:0.35cm) [thick];
\filldraw  [black] (.1,-.25) circle (2.5pt);
\end{tikzpicture} }
\; -\;\;
(q-1)\sum_{i=1}^{b} \;\;
 \hackcenter{ \begin{tikzpicture}
 \draw  (.8,.4) node {$\scs a+i$};
 \node[star,star points=6,star point ratio=0.5, fill] at  (.7,.125) {} ;
\draw  (0,0) arc (180:360:0.35cm) [thick];
\draw[->](.7,0) arc (0:180:0.35cm) [thick];
\end{tikzpicture} }
\;
 \hackcenter{ \begin{tikzpicture}
 \draw  (-.2,.4) node {$\scs b-i$};
 \node[star,star points=6,star point ratio=0.5, fill] at  (0,.125) {} ;
\draw  (0,0) arc (180:360:0.35cm) [thick];
\draw[<-](.7,0) arc (0:180:0.35cm) [thick];
\end{tikzpicture} }
\end{equation}
Comparing both sides gives the identity that for $n\geq 1$
\begin{equation} \label{eq:bub-stardot}
 \hackcenter{ \begin{tikzpicture}
 \draw  [shift={+(0,0)}](.3,.125) node {$n$};
 \node[star,star points=6,star point ratio=0.5, fill] at  (0,.125) {} ;
\draw  (0,0) arc (180:360:0.5cm) [thick];
\draw[<-](1,0) arc (0:180:0.5cm) [thick];
\filldraw  [black] (.1,-.25) circle (2.5pt);
\end{tikzpicture} }
\;\; = \;\;
(q-1) \sum_{i=1}^{n}  \tilde{C}_{i} C_{n-i}.
\end{equation}

The following relations describe how to slide starred bubbles slide through vertical strands.
\begin{align} \label{eq:c-slide}
 \hackcenter{ \begin{tikzpicture}
 \draw[thick, ->] (1.2,-.75) -- (1.2,.75) node[pos=.55, shape=coordinate](DOT){};;
    \node[star,star points=6,star point ratio=0.5, fill] at  (DOT) {} ;
  \node at (1.45,.35) {$\scs j$};
 \draw  (.3,.125) node {$\scs n$};
 \node[star,star points=6,star point ratio=0.5, fill] at  (0,.125) {} ;
\draw  (0,0) arc (180:360:0.35cm) [thick];
\draw[<-](.7,0) arc (0:180:0.35cm) [thick];
\end{tikzpicture} }
&\;\; = \;\;
 \hackcenter{ \begin{tikzpicture}
 \draw[thick, ->] (-.5,-.75) -- (-.5,.75) node[pos=.55, shape=coordinate](DOT){};;
    \node[star,star points=6,star point ratio=0.5, fill] at  (DOT) {} ;
  \node at (-.75,.35) {$\scs j$};
 \draw  (.3,.125) node {$\scs n$};
 \node[star,star points=6,star point ratio=0.5, fill] at  (0,.125) {} ;
\draw  (0,0) arc (180:360:0.35cm) [thick];
\draw[<-](.7,0) arc (0:180:0.35cm) [thick];
\end{tikzpicture} }
\; + (n+1) \;
\frac{(q-1)^2}{q}\;
 \hackcenter{ \begin{tikzpicture}
 \draw[thick, ->] (-.5,-.75) -- (-.5,.75) node[pos=.55, shape=coordinate](DOT){};;
    \node[star,star points=6,star point ratio=0.5, fill] at  (DOT) {} ;
  \node at (0,.35) {$\scs j+n+1$};
\end{tikzpicture} }
 -n\;
 \hackcenter{ \begin{tikzpicture}
 \draw[thick, ->] (-.5,-.75) -- (-.5,.75) node[pos=.55, shape=coordinate](DOT){};;
    \node[star,star points=6,star point ratio=0.5, fill] at  (DOT) {} ;
  \node at (0,.35) {$\scs j+n$};
\end{tikzpicture} }
 - \;
\frac{(q-1)^2}{q}
\sum_{i=1}^n
i \;\;
 \hackcenter{ \begin{tikzpicture}
 \draw[thick, ->] (-.7,-.75) -- (-.7,.75) node[pos=.55, shape=coordinate](DOT){};;
    \node[star,star points=6,star point ratio=0.5, fill] at  (DOT) {} ;
  \node at (-1,.35) {$\scs j+i$};
 \draw  (-.1,.5) node {$\scs n-i$};
 \node[star,star points=6,star point ratio=0.5, fill] at  (0,.125) {} ;
\draw  (0,0) arc (180:360:0.35cm) [thick];
\draw[<-](.7,0) arc (0:180:0.35cm) [thick];
\end{tikzpicture} }
\\
 \label{eq:cc-slide}
 \hackcenter{ \begin{tikzpicture}
 \draw[thick, ->] (-.5,-.75) -- (-.5,.75) node[pos=.55, shape=coordinate](DOT){};;
    \node[star,star points=6,star point ratio=0.5, fill] at  (DOT) {} ;
  \node at (-.75,.35) {$\scs j$};
 \draw  (.3,.125) node {$\scs n$};
 \node[star,star points=6,star point ratio=0.5, fill] at  (0,.125) {} ;
\draw[->]  (0,0) arc (180:360:0.35cm) [thick];
\draw(.7,0) arc (0:180:0.35cm) [thick];
\end{tikzpicture} }
&\;\; = \;\;
 \hackcenter{ \begin{tikzpicture}
 \draw[thick, ->] (1.2,-.75) -- (1.2,.75) node[pos=.55, shape=coordinate](DOT){};;
    \node[star,star points=6,star point ratio=0.5, fill] at  (DOT) {} ;
  \node at (1.45,.35) {$\scs j$};
 \draw  (.3,.125) node {$\scs n$};
 \node[star,star points=6,star point ratio=0.5, fill] at  (0,.125) {} ;
\draw[->]  (0,0) arc (180:360:0.35cm) [thick];
\draw(.7,0) arc (0:180:0.35cm) [thick];
\end{tikzpicture} }
\; - \;
\frac{(q-1)^2}{q} \sum_{\ell=1}^{n-1}(n-\ell)
 \hackcenter{ \begin{tikzpicture}
 \draw[thick, ->] (1.2,-.75) -- (1.2,.75) node[pos=.55, shape=coordinate](DOT){};;
    \node[star,star points=6,star point ratio=0.5, fill] at  (DOT) {} ;
  \node at (1.7,.35) {$\scs j+n-\ell$};
 \draw  (.3,.125) node {$\scs \ell$};
 \node[star,star points=6,star point ratio=0.5, fill] at  (0,.125) {} ;
\draw[->]  (0,0) arc (180:360:0.35cm) [thick];
\draw(.7,0) arc (0:180:0.35cm) [thick];
\end{tikzpicture} }
\end{align}

%
\subsection{Triangular decomposition for $\Tr(\cal{H})$} \label{sec:tri-decomp}
%

Recall the following characterizations of hom spaces in $\H$.

 \begin{proposition}
\label{LSEndThm}
\cite[Theorem 3.9]{LS13}
There is an isomorphism of algebras
\begin{equation*}
\End_{\H}(\P^n) \cong \dah_n \otimes \k[q,q^{-1}][c_0, c_1, \ldots].
\end{equation*}
In particular, when $n=0$ we have $\End(\mathbf{1}) \cong \k[q,q^{-1}][c_0, c_1, \ldots]$.
\end{proposition}

\begin{corollary} \label{cor:endone}
In terms of star-calculus we also have that
\[
\End(\mathbf{1}) \cong  \k[q,q^{-1}][C_{0}, C_{1}, \ldots].
\]
\end{corollary}

\begin{proof}
Using \eqref{eq:def-star} it is not difficult to define a change of basis between dotted bubbles $c_j$ and their starred counterparts $C_j$.
\end{proof}

Let $J_{m,n} $ be the ideal of $\End_{\H}(\P^m\Q^n)$ generated by diagrams which contain at least one arc connecting a pair of upper points.

\begin{proposition}
\label{LSEndThmPmQn}
\cite[Proposition 3.11]{LS13}
There exists a short exact sequence
\begin{equation*}
0 \rightarrow J_{m,n} \rightarrow \End_{\H}(\P^m\Q^n) \rightarrow \dah_m \otimes (\dah_n)^{op} \otimes \k[q,q^{-1}][c_0, c_1, \ldots] \rightarrow 0.
\end{equation*}
Furthermore, this sequence splits.
\end{proposition}

\begin{lemma}\label{lem:invert}
If $f,g\in \dah_n$ with $fg = 1$, then in fact $f,g\in H_n\subset \dah_n$.
\end{lemma}
\begin{proof}
We consider the non-negative integral filtration on $\dah_n$, with associated graded $\mbox{gr}(\dah_n) \cong H_n\rtimes \k[q,q^{-1}][x_1,\dots,x_n]$, where $H_m$ is the finite Hecke algebra Definition~\ref{def:finite-hecke}.  The degree 0 part of this filtration is precisely $H_n$. Now $\mbox{gr}(f)\mbox{gr}(g) = 1$, which implies that $\mbox{gr}(f)$ and $\mbox{gr}(g)$ are in $H_n$. Thus $f$ and $g$ are in the degree 0 part of the filtration, as desired.
\end{proof}

\begin{lemma} \label{lem:Hindecomposables}
The indecomposable objects of $\H$ are of the form $\P^m\Q^n$ for $m,n \in \mathbb{Z}_{\geq 0}$.
\end{lemma}
\begin{proof}
An indecomposable object must be of the form $\P^m\Q^n$ because a subsequence $\Q\P$ produces a decomposition
$\P\Q \oplus \mathbf{1}$. There is a morphism from $\P^m\Q^n$ to $\P^a\Q^b$ if and only if $m-n=a-b$.
If $a+b \neq m+n$ then the composition $\P^a\Q^b \rightarrow \P^m\Q^n \rightarrow \P^a\Q^b$ produces cups and caps or circles.
By Proposition ~\ref{LSEndThmPmQn} it follows that this composition cannot be the identity.

Now suppose there are maps
$ f \colon \P^m \Q^n \rightarrow \P^m \Q^n $
and
$ g \colon \P^m \Q^n \rightarrow \P^m \Q^n $ such that $ gf$ is the identity.
We claim that in fact $fg$ is the identity, too. Since a composition containing caps or cups cannot be the identity,
Proposition ~\ref{LSEndThmPmQn} implies $f$ is a monoidal composition of $ f_1 $ and $ f_2 $ where
$ f_1 \colon \P^m \rightarrow \P^m $ and $ f_2 \colon \Q^n \rightarrow \Q^n$.
Therefore, $ f_1 $ can be identified with an element in $\dah_m$ and $ f_2$ may be identified with an element in $\dah_n^{op}$.
Similarly, $g$ is a monoidal composition of $g_1$ and $g_2$ where
$ g_1 \colon \P^m \rightarrow \P^m $ and $ g_2 \colon \Q^n \rightarrow \Q^n $, so that
 $g_1$ can be identified with an element in $\dah_m$ and $g_2$ may be identified with an element in
$\dah_n^{op}$.
But now, by Lemma \ref{lem:invert}, $f_1$ and $g_1$ must belong to $H_m$.  Since $H_m$ is semi-simple at generic $q$ (over an algebraically closed field),
$ g_1 f_1 $ is the identity if and only $f_1 g_1$ is the identity.
Thus $f$ is an isomorphism, and $\P^m\Q^n$ is indecomposable.
\end{proof}

\begin{lemma}
\label{HHVSLEMMA}
There is an isomorphism of vector spaces
\begin{equation*}
\Tr(\H) \cong
\bigoplus_{m,n \geq 0}
\Tr\left(\dah_m\right) \otimes \Tr\left((\dah_n)^{op}\right) \otimes \k[q,q^{-1}][c_0, c_1, \ldots].
\end{equation*}
\end{lemma}

\begin{proof}
Recall that the trace is determined by endomorphisms of indecomposable objects~(see Proposition~\ref{prop:indecomposables}).
Hence, to compute $\Tr(\H)$ it suffices to consider endomorphisms of all indecomposable objects ($m$ products of $ \P$ followed by $n$ products of $\Q$) modulo the trace ideal $\mathcal{I}$.
That is
\begin{equation*}
\Tr(\H) \cong \bigoplus_{m,n \geq 0} \End(\P^m\Q^n) / \mathcal{I}.
\end{equation*}
Since any map from $\P^m\Q^n \to \P^{m'}\Q^{n'}$ with $m\neq m'$ or $n\neq n'$ must contain a cap or cup,  it is clear that the ideal $\cal{I}$ is equal to the direct sum of the ideal $J_{m,n}$ from Proposition ~\ref{LSEndThmPmQn} together with the ideal generated by $fg-gf$ where $f,g \colon \P^m \Q^n \rightarrow \P^m\Q^n$.
Using the trace relation together with the relations in $\H$, any map contained in the ideal $J_{m,n}$ can be reduced to a sum of endomorphisms of $\P^{m'} \Q^{n'}$ for some $m'$, $n'$ that are not in the ideal $J_{m,n}$ by induction on the number of cap/cups in the diagram.  Indeed, each cap/cup can be isotoped around the annulus using the trace relation and relations in $\H$ to eventually create a curl or circle, which can then be further reduced to diagrams containing fewer caps and cups.
By applying Proposition ~\ref{LSEndThmPmQn}
\begin{align*}
\Tr(\H) &\cong \bigoplus_{m,n \geq 0}
\Tr\left(\dah_m \otimes (\dah_n)^{op} \otimes \k[q,q^{-1}][c_0, c_1, \ldots] \right) \\%
&=
\bigoplus_{m,n \geq 0}
\Tr\left(\dah_m\right) \otimes \Tr\left((\dah_n)^{op}\right) \otimes \k[q,q^{-1}][c_0, c_1, \ldots].
\end{align*}

\end{proof}

%
\subsection{A generating set for $\Tr(\cal{H})$}
%

The definitions of the following elements in $\Tr(\H)$ are motivated by the definition of the generators $P_\bx$ for the trace of the affine Hecke category in Remark \ref{rmk:allgens} along with some computations involving their commutators.
\begin{definition}\label{def_ws}
For $j,k > 0$ we define the following elements in $\Tr(\qh)$:
\begin{align}
w_{j,0} &:= \frac{1-q^{-1}}{1-q^{-j}}\left[ \sum_{i=0}^{j-1}q^{-i}\;
\hackcenter{\begin{tikzpicture}[scale=0.8, xscale=-1.0]
  \draw[thick,->] (3.0,0) .. controls ++(0,1.25) and ++(0,-1.75) .. (-0.6,2);
  \draw[thick,->] (-0.6,0) .. controls ++(0,1) and ++(0,-1.2) .. (0,2);
  \draw[thick,->] (0,0) .. controls ++(0,1) and ++(0,-1.2) .. (.6,2);
  \draw[thick,->] (.6,0) .. controls ++(0,1) and ++(0,-1.2) .. (1.2,2);
  \draw[thick,->] (1.8,0) .. controls ++(0,1) and ++(0,-1.2) .. (2.4,2);
  \draw[thick,->] (2.4,0) .. controls ++(0,1) and ++(0,-1.2) .. (3,2);
  \node at (1.2,.35) {$\dots$};
  \node at (1.8,1.65) {$\dots$};
   \node at (0,-.2) {$\underbrace{\hspace{0.48in}}$};
   \node at (2.4,-.2) {$\underbrace{\hspace{0.48in}}$};
   \node at (2.4,-.5) {$\scs i+1$};
    \node at (0,-.5) {$\scs j-1-i$};
 \draw  (-.14,1.20) circle (4pt);
\draw  (.35,.99) circle (4pt);
\draw  (.87,.88) circle (4pt);
\end{tikzpicture}} \right] \qquad \text{($j$ strands)} \nn \\
w_{1,k} &:= \left[ \;\; \hackcenter{\begin{tikzpicture}[scale=0.8]
    \draw[thick, ->] (0,0) -- (0,1.5) node[pos=.55, shape=coordinate](DOT){};;
    \node[star,star points=6,star point ratio=0.5, fill] at  (DOT) {} ;
    \node at (.35,.8) {$\scs k$};
\end{tikzpicture}}
\; \right]
= \left[ \left(\frac{1}{q-1} \;\;
\hackcenter{\begin{tikzpicture}[scale=0.8]
    \draw[thick, ->] (0,0) -- (0,1.5) node[pos=.55, shape=coordinate](DOT){};;
    \filldraw  (DOT) circle (2.5pt);
\end{tikzpicture}} \;\; + \;\; \frac{q}{(q-1)^2} \;\;
\hackcenter{\begin{tikzpicture}[scale=0.8]
    \draw[thick, ->] (0,0) -- (0,1.5) ;
\end{tikzpicture}} \;\;\;\right)^k \right] \nn \\
w_{j,1} &:=  q^{\frac{(1-j)}{2}}\left[ \;\;
\hackcenter{
\begin{tikzpicture}[scale=0.8]
  \draw[thick,->] (3,0) .. controls ++(0,1.25) and ++(0,-1.75) .. (0,2);
  \draw[thick,->] (0,0) .. controls ++(0,1) and ++(0,-1.2) .. (.6,2);
  \draw[thick,->] (.6,0) .. controls ++(0,1) and ++(0,-1.2) .. (1.2,2);
  \draw[thick,->] (1.8,0) .. controls ++(0,1) and ++(0,-1.2) .. (2.4,2);
  \draw[thick,->] (2.4,0) .. controls ++(0,1) and ++(0,-1.2) .. (3,2);
  \node at (1.2,.35) {$\dots$};
  \node at (1.8,1.65) {$\dots$};
\node[star,star points=6,star point ratio=0.5, fill] at  (2.93 ,.3) {} ;
\end{tikzpicture}}
\;\;
\right]
  \;=\; q^{\frac{(j-1)}{2}}
\left[ \;\;
\hackcenter{
\begin{tikzpicture}[scale=0.8]
  \draw[thick,->] (3,0) .. controls ++(0,1.25) and ++(0,-1.75) .. (0,2);
  \draw[thick,->] (0,0) .. controls ++(0,1) and ++(0,-1.2) .. (.6,2);
  \draw[thick,->] (.6,0) .. controls ++(0,1) and ++(0,-1.2) .. (1.2,2);
  \draw[thick,->] (1.8,0) .. controls ++(0,1) and ++(0,-1.2) .. (2.4,2);
  \draw[thick,->] (2.4,0) .. controls ++(0,1) and ++(0,-1.2) .. (3,2);
  \node at (1.2,.35) {$\dots$};
  \node at (1.8,1.65) {$\dots$};
\draw  (.4,1.1) circle (4pt);
\draw  (.9,.94) circle (4pt);
\draw  (1.99,.72) circle (4pt);
\draw  (2.51,.57) circle (4pt);
\node[star,star points=6,star point ratio=0.5, fill] at  (.07 ,1.6) {} ;
\end{tikzpicture}}
\;\;
\right]  \qquad \text{($j$ strands)} \nn \\
w_{0,k} &:=
\frac{1}{\{k\}}\left( \tilde{C}_k + \frac{(q-1)^2}{q}\sum_{j=1}^k (k+1-j) \tilde{C}_j C_{k-j}\right)\nn
\end{align}
Let $\Psi$ be in the involution that flips diagrams across the $y$-axis and switches the orientations of the strands. This involution reverses the order of tensor products, so it induces a $\k[q,q^{-1}]$-linear algebra anti-involution $\Psi: \Tr(\H) \to \Tr(\H)$. We define
\[
 w_{-a,b} := \Psi(w_{a,b}).
\]
\end{definition}

\begin{definition}\label{def_yts}
For an element $y_1^{a_1} \cdots y_n^{a_n} t_w \in \ah^+$, recall that
$[y_1^{a_1} \cdots y_n^{a_n} t_w]$ denotes its image in $\Tr(\ah^+)$.
Via the isomorphism in Proposition ~\ref{LSEndThm} we define elements
\begin{align}
[y_1^{a_1} \cdots y_n^{a_n} t_w]_{\uparrow} &\in \Tr(\End_{\H}(\P^n)) \label{def_up}\\
[y_1^{a_1} \cdots y_n^{a_n} t_w]_{\downarrow}  &\in \Tr(\End_{\H}(\Q^n)) \label{def_down}.
\end{align}
The elements in ~\eqref{def_up} (respectively ~\eqref{def_down}) correspond to
upward (respectively downward) pointing diagrams with crossings governed by $t_w$ with $a_i$ dots on the $i$th strand on the top counting from the left. This is illustrated on the right hand side of the following diagram:
\end{definition}
\[
\hackcenter{\begin{tikzpicture}
    \draw[very thick] (-1.2,0) -- (-1.2,1.5);
    \draw[very thick] (-.6,0) -- (-.6,1.5);
    \draw[very thick] (.3,0) -- (.3,1.5);
    \draw[fill=white!20,] (-1.4,.5) rectangle (.5,1);
    \node () at (-.3,.75) {$t_w$};
    \node at (-.6,1.2) {$\bullet$};
    \node at (-.8,1.3) {$\scs a_2$};
        \node at (-1.2,1.2) {$\bullet$};
    \node at (-1.4,1.3) {$\scs a_1$};
    \node at (.3,1.2) {$\bullet$};
    \node at (.1,1.3) {$\scs a_n$};
    \node at (-.1, .2) {$\cdots$};
\end{tikzpicture}}
\quad \mapsto
\quad
\hackcenter{\begin{tikzpicture}
      \path[draw,blue, very thick, fill=blue!10]
        (-2.3,-.6) to (-2.3,.6) .. controls ++(0,1.85) and ++(0,1.85) .. (2.3,.6)
         to (2.3,-.6)  .. controls ++(0,-1.85) and ++(0,-1.85) .. (-2.3,-.6);
      \path[draw, blue, very thick, fill=white]
            (-0.1,0) .. controls ++(0,.35) and ++(0,.35) .. (0.3,0)
            .. controls ++(0,-.35) and ++(0,-.35) .. (-0.1,0);
    \draw[very thick] (-1.85,-.7) -- (-1.85, .7).. controls ++(0,.95) and ++(0,.95) .. (1.65,.7)
        to (1.65,-.7) .. controls ++(0,-.95) and ++(0,-.95) .. (-1.85,-.7);
    \draw[very thick] (-1.25,-.55) -- (-1.25,.55) .. controls ++(0,.65) and ++(0,.65) .. (1.1,.55)
        to (1.1,-.55) .. controls ++(0,-.65) and ++(0,-.65) .. (-1.25, -.55);
    \draw[very thick] (-.45,-.4) -- (-.45,.4) .. controls ++(0,.35) and ++(0,.35) .. (.55,.4)
        to (.55, -.4) .. controls ++(0,-.35) and ++(0,-.35) .. (-.45,-.4);
    \draw[fill=white!20,] (-2,-.25) rectangle (-.3,.25);
    \node () at (-1.2,0) {$t_w$};
    \node at (-2.05,.6) {$\scs a_1$};
    \node at (-1.85,.5) {$\bullet$};
    \node at (-1.45,.6) {$\scs a_2$};
    \node at (-1.25,.5) {$\bullet$};
    \node at (-.65,.55) {$\scs a_n$};
    \node at (-.45,.45) {$\bullet$};
    \node at (-.8, -.5) {$\cdots$};
\end{tikzpicture}}
 \]

\begin{definition} \label{def:si}
In a similar manner to the elements described in ~\ref{def_up} and ~\ref{def_down} there are elements
\begin{align}
[s_1^{a_1} \cdots s_n^{a_n} t_w]_{\uparrow} &\in \Tr(\End_{\H}(\P^n)) \label{def_upstar}\\
[s_1^{a_1} \cdots s_n^{a_n} t_w]_{\downarrow}  &\in \Tr(\End_{\H}(\Q^n)) \label{def_downstar}
\end{align}
where now the dots on the tops of the diagrams are starred dots.
\end{definition}

In the following sequence of lemmas, we describe a set of generators of $\Tr(\H)$ (as an algebra). We then decrease the size of this set in stages until we show in Proposition \ref{genprop1} that $\Tr(\H)$ is generated (as an algebra) by the elements in Definition \ref{def_ws}. The following elements in the Hecke algebra will play a crucial role.
\begin{equation*}
\gamma_n := t_{n-1} \cdots t_1 \hspace{.5in}
\gamma'_n := t_1 \cdots t_{n-1}.
\end{equation*}

\begin{lemma}
\label{genlemma1}
A generating set for the trace is given by
\begin{equation*}
\bigcup_{n \geq 1} \{[y_1^{a_1} \cdots y_n^{a_n} t_w]_{\uparrow}, [y_1^{a_1} \cdots y_n^{a_n} t_w]_{\downarrow}  | w \in S_n,  a_1, \ldots, a_n \in \Z_{\geq 0}  \};
\bigcup \{ \tilde{c}_{k}, c_{k} | k \in \Z_{\geq 0}     \}.
\end{equation*}
\end{lemma}

\begin{proof}
In the category (before idempotent completion) any object is a direct sum of objects of the form $\P^a \Q^b$ (see Lemma~\ref{lem:Hindecomposables}).
Now we argue as in \cite{CLLS15} that in the trace, any cap or cup is part of a bubble or a curl.
Thus any map between objects descends in the trace to a product of endomorphisms of $\P$, $\Q$ along with bubble terms.  Endomorphisms of $\P^m$ and $\Q^n$ are described in Theorem 3.9 of \cite{LS13}.  They are spanned by elements of the form $[y_1^{a_1} \cdots y_m^{a_m} t_{\alpha}]_{\uparrow} $
where $ \alpha$ is a reduced expression in $S_m$ and
$ [y_1^{b_1} \cdots y_n^{b_n} t_{\beta}]_{\downarrow}$
where $\beta$ is a reduced expression in $S_n$ respectively.
\end{proof}

\begin{lemma}
\label{genlemma2}
A generating set for the trace is given by
$$\bigcup_{n \geq 1} \{ [y_1^{a_1} \cdots y_n^{a_n} \gamma_n]_{\uparrow},
[y_1^{a_1} \cdots y_n^{a_n} \gamma_n]_{\downarrow}
 |  a_1, \ldots, a_n \in \Z_{\geq 0}  \}
\bigcup \{ \tilde{c}_{k}, c_{k} | k \in \Z_{\geq 0}  \}.$$
\end{lemma}

\begin{proof}
This is proved by induction on the Bruhat order and the degree of the exponents of the $y_i$. Let $w$ be a reduced expression in $S_n$.  Consider the elements
$[y_1^{a_1} \cdots y_n^{a_n} t_w]_{\uparrow}$ or $[y_1^{a_1} \cdots y_n^{a_n} t_w]_{\downarrow}$.

Assume that we could generate all elements
of the form
$[y_1^{b_1} \cdots y_n^{b_n} t_w']_{\uparrow}$ or $[y_1^{b_1} \cdots y_n^{b_n} t_w']_{\downarrow}$
where $w'$ is smaller than $w$ in the Bruhat order or $w=w' $ and $ b_i \leq a_i$ for all $i$ with at least one $b_i$ strictly smaller.

Let $ g$ be a reduced expression in $S_n$ such that $g w g^{-1}$ is a product of cycles.
Then using relations in the Hecke algebra, dot sliding relations, and the induction hypothesis we get that both
\begin{align*}
[y_1^{a_1} \cdots y_n^{a_n} t_w]_{\uparrow} &= [t_g]_{\uparrow} [y_1^{a_1} \cdots y_n^{a_n} t_w]_{\uparrow} [t_{g}^{-1}]_{\uparrow} \\
[y_1^{a_1} \cdots y_n^{a_n} t_w]_{\downarrow} &= [t_g]_{\downarrow} [y_1^{a_1} \cdots y_n^{a_n} t_w]_{\downarrow} [t_{g}^{-1}]_{\downarrow}
\end{align*}
are generated by elements described in the lemma.
\end{proof}

\begin{lemma}
\label{genlemma3}
A generating set for the trace is given by
$$\bigcup_{n \geq 1} \{ [y_1^{a_1} \gamma_n]_{\uparrow}, [y_1^{a_1} \gamma_n]_{\downarrow}   |  a_1  \in \Z_{\geq 0}  \}
\bigcup \{ \tilde{c}_{k}, c_{k} | k \in \Z_{\geq 0}  \}.$$
\end{lemma}

\begin{proof}
Using ~\eqref{eq:nil-dot} we may move a dot on a strand of
$ [y_1^{a_1} \cdots y_n^{a_n} \gamma_n]_{\uparrow}  $ or
$ [y_1^{a_1} \cdots y_n^{a_n} \gamma_n]_{\downarrow}  $
to the left, at the cost of lower order terms where that dot disappears.  Repeating this process (and inducting on the number of dots) we get that the sets described in the lemma generate the trace.

\end{proof}



\begin{lemma}
\label{genlemma5}
For $n \geq 1$,
\begin{itemize}
\item Elements $[\gamma_n]_{\uparrow}$ are generated by ${w}_{j,0}$ for $j \geq 1$,
\item Elements $[\gamma_n]_{\downarrow}$ are generated by ${w}_{-j,0}$ for $j \geq 1$.
\end{itemize}
\end{lemma}
\begin{proof}
One could express circle crossings in terms of the other crossing and identity terms using \eqref{eq:inv-crossing}.

For the proof of the second item, note that an element $[\gamma_n]_{\downarrow}$ could be written in terms of elements $[\gamma_j']_{\downarrow}$.  The rest of the proof is similar to the proof of the first item.
\end{proof}

\begin{lemma}
\label{genlemma6}
For $n \geq 1$,
\begin{itemize}
\item Elements $[y_1^{a_1} \gamma_n']_{\uparrow}$ are generated by ${w}_{j,1}$ for $j \geq 1$,
\item Elements $[y_1^{a_1} \gamma_n']_{\downarrow}$ are generated by ${w}_{-j,1}$ for $j \geq 1$.
\end{itemize}
\end{lemma}

\begin{proof}
This follows from the definitions and relations in the category.
\end{proof}

\begin{lemma}
\label{genlemma7}
Elements of the form
$[y_1^{a_1} \gamma_n]_{\uparrow}, [y_1^{a_1} \gamma_n]_{\downarrow}$
are generated by ${w}_{j,0}$ and ${w}_{0,1}$.
\end{lemma}

\begin{proof}
Note that ${w}_{0,1}= \frac{1}{\{1\}} \left(\frac{1}{\{1\}^2} + C_0\right)$.
Now we know from a bubble slide equation that
\begin{equation*}
[[\gamma_n]_{\uparrow},C_0]=-\{1\}^2\sum_{i=1}^n
[s_i \gamma_n]_{\uparrow}.
\end{equation*}
Thus we could produce $[y_1^{a_1} \gamma_n]_{\uparrow}$ from repeatedly applying the adjoint action of ${w}_{0,1}$.

Generating the other element in the statement of the lemma may be accomplished in a similar manner.
\end{proof}



\begin{lemma}
\label{genlemma9}
Elements of the form $C_i$ and $\tilde{C}_i$ are generated by ${w}_{0,k}$ for $ k \leq i+1$.
\end{lemma}

\begin{proof}
Since
${w}_{0,1} = \frac{1}{\{1\}} \left(\frac{1}{\{1\}^2} + C_0\right) $, the statement is true for $i=0$.

By definition
\begin{align*}
{w}_{0,k} &= \frac{1}{\{k\}} ( \tilde{C}_k + \{1\}^2 \sum_{j=1}^k (k+1-j) \tilde{C}_j C_{k-j} )\\
&=\frac{1}{\{k\}} (p(C_0, \ldots, C_{k-2}) + k C_{k-1} + \{1\}^2 \sum_{j=2}^k p_j(C_0, \ldots, C_{j-2}) C_{k-j})
\end{align*}
where $p$ and $p_j$ are polynomials in variables indicated above.
The lemma now follows by induction on $i$.
\end{proof}

\begin{lemma}
\label{genlemma10}
The elements $\tilde{c}_{r} $ are generated by ${w}_{0,k}$.
\end{lemma}

\begin{proof}
This follows directly from the definition of $\tilde{C}_r$ and Lemma ~\ref{genlemma9}.
\end{proof}

\begin{proposition}
\label{genprop1}
The trace $\Tr(\H)$ is generated as an algebra by ${w}_{\pm j,0}$ for $j \geq 1$ and ${w}_{0,k}$ for $k \geq 1$.
\end{proposition}

\begin{proof}
This follows from Lemmas ~\ref{genlemma3}, ~\ref{genlemma5}, ~\ref{genlemma6}, ~\ref{genlemma7}, and ~\ref{genlemma10}.
\end{proof}

%
\section{An isomorphism between the elliptic Hall algebra and $\Tr(\cal{H})$}\label{sec:mainthm}
%

%
\subsection{The isomorphism}
%
In this subsection we construct a linear isomorphism $\varphi: \mathbb{E}^{\bullet, \geq} \to \Tr(\H)$ from the upper half of the elliptic Hall algebra described in Section \ref{sec_upperhalf} to the trace of the $q$-Heisenberg category. We define  $\varphi = \varphi^-\otimes \varphi^0\otimes \varphi^+$ using the triangular decomposition $\mathbb{E}^{\bullet,\geq} = \mathbb{E}^{<,\geq}\otimes \mathbb{E}^{0,\geq}\otimes \mathbb{E}^{>,\geq}$.
Each component of $\varphi$ will be an algebra morphism, and we will show $\varphi$ is a $\k$-linear isomorphism by a dimension count. We will show that $\varphi$ itself is an algebra morphism by checking the cross-relations of Proposition \ref{prop_cross}.

As an immediate Corollary of the algebra homomorphism $\dah_n \to \Tr(\End_{\H}(\P^n))$ from Proposition~\ref{LSEndThm} and the isomorphism $\dah_n\cong \ah_n^+$ we have the following.

\begin{corollary}\label{cor_upperiso}
 There is an algebra map $\tilde \varphi: \Tr(\ah^+)\to \Tr(\H)$ from the trace of the positive affine Hecke category to the trace of the $q$-Heisenberg category. Precomposing $\tilde \varphi$ with the isomorphism $\Tr(\ah^+) \to \Sk_q^{>, \geq}(T^2)$ of Theorem \ref{thm_ahatoskein} and the isomorphism $\Sk_q^{>,\geq}(T^2) \to \E^{>, \geq}$ in the last statement of Proposition \ref{prop_tridecomp}, we obtain an algebra morphism
 \[
 \varphi^+: \mathbb{E}^{>,\geq} \to \Tr(\H)
 \]
\end{corollary}

Let $\varphi^-$ be the composition of $\varphi$ with the anti-automorphism $\Psi$ of Definition \ref{def_ws}, precomposed with the anti-automorphism of $\mathbb{E}$ sending $\bw_{a,b} \mapsto \bw_{-a,b}$. (In other words, $\varphi^-: \mathbb{E}^{<,\geq} \to \Tr(\H)$ is an algebra morphism with $\varphi^-(\bw_{a,b}) = w_{a,b}$, for $a < 0$ and $b \geq 0$.) Let $\varphi^0: \mathbb{E}^{0,\geq} \to \Tr(\H)$ be the map sending $\bw_{0,k}$ to $w_{0,k}$.

\begin{proposition}
The map $\varphi:\mathbb{E}^{\bullet, \geq} \to \Tr(\H)$ defined as $\varphi = \varphi^-\otimes \varphi^0\otimes \varphi^+$ is a linear isomorphism.
\end{proposition}
\begin{proof}
This follows from the triangular decomposition of $\Tr(\H)$ from Proposition~\ref{HHVSLEMMA}, the generation Proposition~\ref{genprop1}, and our control of over the size of the trace from Proposition \ref{prop_trdim}.
\end{proof}

\begin{theorem}\label{thm_mainthm}
The map $\varphi: \mathbb{E}^{\bullet, \geq} \to \Tr(\H)$ is an isomorphism of algebras.
\end{theorem}
\begin{proof}
We have already shown that $\varphi$ is a linear isomorphism, and that each component of $\varphi$ is an algebra morphism. By Proposition \ref{prop_tridecomp}, it suffices to show that the cross-relations \eqref{eq_relhard}--\eqref{eq_rel6} hold after applying $\varphi$. We will translate these relations to relations between diagrammatic generators in Proposition \ref{prop_cross}. These cross-relations will be proven diagrammatically in Section \ref{sec_crossrelations}.
\end{proof}

%
\subsection{Diagrammatic generators}\label{sec_diagdefs}
%
In this section we give diagrammatic definitions of the generators $\varphi(\bw_\bx)$ needed to check the cross-relations \eqref{eq_relhard}--\eqref{eq_rel6}. We first recall the definition of the skein elements $P_\bx \in \Sk_q(T^2)$ used in \cite{MS14}.

Let $\sigma_i = q^{-1/2}t_i$ in the finite Hecke algebras (see Definition~\ref{def:finite-hecke}), so that
$(\sigma_i-q^{1/2})(\sigma_i+q^{-1/2}) = 0$. In the skein model of the finite Hecke algebra, $\sigma_i$ is the positive crossing of strand $i$ and $i+1$ (see Definition \ref{def:braidgens}). We define elements $A_{i,j} \in H_{i+j+1}$ in the finite Hecke algebra as follows:
\[
A_{i,j} = \sigma_1\cdots\sigma_i \sigma_{i+1}^{-1}\cdots \sigma_{i+j}^{-1}
\]
The finite Hecke algebra is the quotient of the braid group by the HOMFLYPT skein relation, and we define $P_{k,0} \in \Sk_q(T^2)$ by
\[
P_{k,0} := \frac{\{1\}}{\{k\}} \sum_{i=0}^{k-1}\cl(A_{i, k-i-1})
\]
where $\cl(-)$ is the closure map of Section \ref{sec_comparison}. In particular, $P_{k,0}$ crosses the meridian $k$ times and the longitude $0$ times.
Then the general elements $P_\bx$ are defined using the $\SL_2(\Z)$ action on $\Sk_q(T^2)$ -- by definition, $\gamma\cdot P_\bx := P_{\gamma \bx}$.

We would like to find elements in $\ah_n$ whose images are $P_{1,k}$ and $P_{j,1}$. Since the element $P_{1,k}$ is the simple closed curve which crosses the (horizontal) meridian once and the longitude $k$ times, it is the closure of $x_1^k \in \ah_1$. Since $P_{j,1}$ is the simple closed curve crossing the meridian $j$ times and the longitude once, it is the closure of an element in $\ah_j$, which we can take to be $\cl(x_j \sigma_{j-1}\cdots \sigma_1) = P_{j,1}$. \
\[
P_{5,1} \;\; = \;\;
\hackcenter{
\begin{tikzpicture} [scale=0.8]
\fill[gray!20!white] (-1.5,0) rectangle (2,3);
\draw[ thick, red, directed=0.7,directed=0.78] (-1.5,0) -- (-1.5,3);
\draw[ thick, red, directed=0.7,directed=0.78] (2,0) -- (2,3);
\draw[ thick, blue ] (2,0) -- (-1.5,0);
\draw[ thick, blue] (2,3) -- (-1.5,3);
\draw[ultra thick,<-] (-.6,3) .. controls ++(0,-1.0) and ++(.8,0) .. (-1.5,1.5);
\draw[ultra thick,->] (1.2,0) .. controls ++(0,1.0) and ++(-.5,0) .. (2,1.5);
\draw[ultra thick,->] (-0,0) .. controls ++(0,1.3) and ++(0,-1.3) .. (.6,3);
\draw[ultra thick,->] (-.6,0) .. controls ++(0,1.3) and ++(0,-1.3) .. (0,3);
\draw[ultra thick,->] (.6,0) .. controls ++(0,1.3) and ++(0,-1.3) .. (1.2,3);
\end{tikzpicture} }
\;\; = \;\;
\hackcenter{
\begin{tikzpicture} [scale=0.8]
\fill[gray!20!white] (-1,0) rectangle (2,3);
\draw[ thick, red, directed=0.7,directed=0.78] (-1,0) -- (-1,3);
\draw[ thick, red, directed=0.7,directed=0.78] (2,0) -- (2,3);
\draw[ thick, blue ] (2,0) -- (-1,0);
\draw[ thick, dotted ] (-1,1.5) -- (2,1.5);
\draw[ thick, blue] (2,3) -- (-1,3);
\draw[ultra thick,<-] (-.6,3) .. controls ++(0,-1.4) and ++(0,1) .. (1.3,1.5);
\draw[ultra thick] (1.3,1.5) .. controls ++(0,-.8) and ++(.4,0) .. (-1,1);
\fill[gray!20!white] (-.8,.5) rectangle (-.4,1.4);
\fill[gray!20!white] (-.2,.5) rectangle (.2,1.4);
\fill[gray!20!white] (.4,.5) rectangle (.8,1.4);
\fill[gray!20!white] (-.3,1.8) rectangle (.06,2.4);
\fill[gray!20!white] (.25,1.8) rectangle (.6,2.4);
\fill[gray!20!white] (.83,1.8) rectangle (1.1,2.4);
\draw[ultra thick,->] (1.2,0) .. controls ++(0,1.0) and ++(-.5,0) .. (2,1);
\draw[ultra thick,->] (-0,0) .. controls ++(0,2.2) and ++(0,-1.3) .. (.6,3);
\draw[ultra thick,->] (-.6,0) .. controls ++(0,2.2) and ++(0,-1.3) .. (0,3);
\draw[ultra thick,->] (.6,0) .. controls ++(0,2.2) and ++(0,-1.3) .. (1.2,3);
\end{tikzpicture} }
\;\; = \;\;
[x_4 \sigma_3 \sigma_2 \sigma_1].
\]
We now recall from \eqref{eq_ydef} that $y_i := (q-1)x_i - q/(q-1)$. Solving for $x_i$, we obtain
\[
x_i = \frac{1}{q-1} y_i + \frac{q}{(q-1)^2}.
\]
Recall that under the map $\dah_n \to \End(\P^{n})$,  the element $y_i$ is sent to a dotted strand. The point of the definition of the star-dot (see Definition \ref{eq:def-star}) is that it is the image in the quantum Heisenberg category of the element $x_i$ in the affine Hecke algebra.

To finish computing the images of the generators that we need, we recall that $t_i$ maps to a crossing of strands $i$ and $i+1$, and $t_i^{-1}$ maps to a circle crossing. Since $\sigma_i = q^{-1/2} t_i$, the diagrams in Definition \ref{def_ws} have been rescaled by a power of $q$ depending on the net number of crossings. We now summarize the discussion so far.

\begin{proposition}\label{prop:uppertrianglemap}
The algebra morphism $\varphi^+: \mathbb{E}^{>,\geq}(T^2) \to \Tr(\H)$ satisfies $\varphi(\bw_{a,b}) = w_{a,b}$, for values of $(a,b)$ for which the right hand side has been defined (see Definition \ref{def_ws}). Furthermore, there exist elements $w_{a,b} \in \Tr(\H)$ for $a > 0$ and $b \geq 0$ that satisfy
\begin{equation}\label{eq_comminprop}
[w_{a,b},w_{c,d}] = -(q^{k/2} - q^{-k/2}) w_{a+c,b+d}
\end{equation}
where $k=ad-bc$.
\end{proposition}
\begin{proof}
The first statement is a summary of the discussion in this subsection. The second is a restatement of the second part of Corollary \ref{cor_upperiso}.
\end{proof}

\begin{remark}
For our purposes, we only formally need the definitions of the $w_{a,b}$ that we have given in Definition \ref{def_ws}. However, for the sake of completeness, we have given an algebraic description of elements in $\Tr(\ah)$ that map to the $w_{a,b}$ for $a > 0$ and $b \geq 0$ in Remark \ref{rmk:allgens}.
\end{remark}

For convenience while doing computations in the next section, we will remove the normalizations of the $w_{a,b}$ and define
\begin{align}
\tilde w_{\pm 1,k} &= w_{\pm 1,k}\notag \\
\tilde w_{j,0} &= \frac{1-q^{-j}}{1-q^{-1}} w_{j,0}\label{eq_normalizations}\\
\tilde w_{0,k} &= \left(q^{k/2}-q^{-k/2}\right) w_{0,k} = \{k\}w_{0,k}\notag\\
\wt_{j,1} &= q^{\frac{1-j}{2}} w_{j,1}  \quad \quad \text{ for } j>0  \notag .
\end{align}

Finally, to complete the proof of Theorem \ref{thm_mainthm} we need to prove that the cross-relations in Proposition \ref{prop_tridecomp} hold in $\Tr(\H)$. We now translate these cross-relations into our renormalized $\wt_{a,b}$.

\begin{proposition}\label{prop_cross}
In terms of the normalized generators $\wt_\bx$, the cross-relations needed to finish the proof of Theorem \ref{thm_mainthm} are the following:
\begin{description}
  \item[{\bf CR1}]  $[\wt_{\pm1, 0}, \wt_{\mp1, k}] =
\mp \wt_{0,k}$;
\medskip


\item[{\bf CR2}]
$[\tilde{w}_{\pm 1,r}, \tilde{w}_{0,k}]= \mp \{k\}^2 \tilde{w}_{1,k+r}$;
\smallskip

\item[{\bf CR3}]
$[\wt_{1,k},\wt_{-1,1}] = -\wt_{0,k+1}$;
\smallskip

\item[{\bf CR4}]
$[\wt_{m,0}, \wt_{-n,0}] = -n \frac{(1-q^{-n})(1-q^n)}{(1-q^{-1})^2} \delta_{m,n}$;
\smallskip

\item[{\bf CR5}]
$[\wt_{n,0},\wt_{-1,1}] = -q^{-n} \frac{(q^n-1)^2}{(q-1)}  \wt_{n-1,1}$.
\end{description}
\end{proposition}
\begin{proof}
This follows immediately from the normalizations in \eqref{eq_normalizations} and from the cross-relations in Proposition \ref{prop_tridecomp}.
\end{proof}

%
\subsection{Cross Relations }\label{sec_crossrelations}
%
In this section we prove the cross-relations between the diagrammatic generators $\wt_{a,b}$ listed in Proposition \ref{prop_cross} that are needed to finish the proof of Theorem \ref{thm_mainthm}.
Equation~{\bf CR1} is proven in Proposition~\ref{prop_lrvert},
{\bf CR2} is Corollary~\ref{cor_prob3tilde},
{\bf CR3} is Proposition~\ref{prop:prob4},
{\bf CR4} is Proposition~\ref{prop:heis-prob5}, and
{\bf CR5} is Corollary~\ref{cor:prob6}.

%
\subsubsection{Cross Relations involving bubbles}\label{sec_crossrelations-bub}
%

\begin{lemma} \label{lem:alt-w0k}
The element $\wt_{0,k}$ can alternatively be written as
\[
\wt_{0,k} =
\; \left[ \;\;
  \hackcenter{\begin{tikzpicture}[scale=0.8]
 \node at (1,.45)  {$\scs k $};
 \node[star,star points=6,star point ratio=0.5, fill] at  (.7,.125) {} ;
\draw  (0,0) arc (180:360:0.35cm) [thick];
\draw[->](.7,0) arc (0:180:0.35cm) [thick];
\end{tikzpicture}}
\;\;\right]
\;\; +\;\;
\frac{(q-1)}{q}
\sum_{i=1}^{k}
 \left[ \;\;
  \hackcenter{\begin{tikzpicture}[scale=0.8]
    \draw[thick] (0,2) .. controls ++(0,.45) and ++(0,.45) .. (.8,2)
        node[pos=.5, shape=coordinate](DOT2){};
    \node[star,star points=6,star point ratio=0.5, fill] at  (DOT2) {} ;
    \draw[thick] (.8,1) .. controls ++(0,-.45) and ++(0,-.45) .. (0,1);
    \draw[thick,->] (0,1) .. controls ++(0,.5) and ++(0,-.5) .. (.8,2)
           node[pos=.0, shape=coordinate](DOT){};
    \node[star,star points=6,star point ratio=0.5, fill] at  (DOT) {};
    \node at (-.45,.7) {$\scs k-i$};;
    \draw[thick, ->] (0,2).. controls ++(0,-.5) and ++(0,.5) .. (.8,1);
    \node at (.8,2.45) {$\scs i$};;
\end{tikzpicture}} \;\; \right] .
\]
\end{lemma}

\begin{proof}
The first part of \eqref{eq:figeight} implies that
\[
\frac{(q-1)}{q}
\sum_{i=1}^{k}
 \left[ \;\;
  \hackcenter{\begin{tikzpicture}[scale=0.8]
    \draw[thick] (0,2) .. controls ++(0,.45) and ++(0,.45) .. (.8,2)
        node[pos=.5, shape=coordinate](DOT2){};
    \node[star,star points=6,star point ratio=0.5, fill] at  (DOT2) {} ;
    \draw[thick] (.8,1) .. controls ++(0,-.45) and ++(0,-.45) .. (0,1);
    \draw[thick,->] (0,1) .. controls ++(0,.5) and ++(0,-.5) .. (.8,2)
           node[pos=.0, shape=coordinate](DOT){};
    \node[star,star points=6,star point ratio=0.5, fill] at  (DOT) {};
    \node at (-.45,.7) {$\scs k-i$};;
    \draw[thick, ->] (0,2).. controls ++(0,-.5) and ++(0,.5) .. (.8,1);
    \node at (.8,2.45) {$\scs i$};;
\end{tikzpicture}} \;\; \right]
\;\; = \;\;
\frac{(q-1)^2}{q}
\sum_{i=1}^{k} \sum_{\ell=1}^{i}
 \hackcenter{ \begin{tikzpicture}
 \draw  (.7,.4) node {$\scs \ell$};
 \node[star,star points=6,star point ratio=0.5, fill] at  (.7,.125) {} ;
\draw  (0,0) arc (180:360:0.35cm) [thick];
\draw[->](.7,0) arc (0:180:0.35cm) [thick];
\end{tikzpicture} }
\;
 \hackcenter{ \begin{tikzpicture}
 \draw  (-.3,.4) node {$\scs k-\ell$};
 \node[star,star points=6,star point ratio=0.5, fill] at  (0,.125) {} ;
\draw  (0,0) arc (180:360:0.35cm) [thick];
\draw[<-](.7,0) arc (0:180:0.35cm) [thick];
\end{tikzpicture} }
\]
so the claim is immediate from the definition of $\wt_{0,k}$.
\end{proof}

\begin{proposition}
\begin{equation*}
\tilde{C}_k = \frac{1}{\{1\}^2} \tilde{C}_{k-1} + \sum_{j=1}^{k-1} \tilde{C}_j C_{k-1-j}.
\end{equation*}
\end{proposition}

\begin{proof}
The claim follows immediately from \eqref{eq:bub-stardot} and the definition \eqref{eq:def-star}.
\end{proof}


\begin{proposition}\label{prop_lrvert}
$[\wt_{\pm1, 0}, \wt_{\mp1, k}] =
\mp \wt_{0,k}$.
\end{proposition}

\begin{proof}
This is a straightforward computation using inductive star slide and the trace relations.
\begin{align*}
\wt_{1,0}\wt_{-1,k}
\;\; &=\;\;
 \left[ \;\;\;  \hackcenter{\begin{tikzpicture}[scale=0.8]
    \draw[thick, ->] (0,0) -- (0,1.5) ;
    \node at (1.15,.8) {$\scs k$};
    \draw[thick, ->] (.8,1.5) -- (.8,0) node[pos=.55, shape=coordinate](DOT){};;
    \node[star,star points=6,star point ratio=0.5, fill] at  (DOT) {} ;
\end{tikzpicture}}
\; \right]
\;\; \refequal{\eqref{heis:up down}} \;\; q^{-1}
 \left[ \;\;
  \hackcenter{\begin{tikzpicture}[scale=0.8]
    \draw[thick] (0,0) .. controls ++(0,.75) and ++(0,-.45) .. (.8,1) ;
    \node at (1.1,.4) {$\scs k$};
    \draw[thick,<-] (.8,0) .. controls ++(0,.75) and ++(0,-.45) .. (0,1)
        node[pos=.2, shape=coordinate](DOT2){};
    \node[star,star points=6,star point ratio=0.5, fill] at  (DOT2) {} ;
    \draw[thick] (0,1 ) .. controls ++(0,.5) and ++(0,-.5) .. (.8,2);
    \draw[thick, ->] (.8,1) .. controls ++(0,.5) and ++(0,-.5) .. (0,2);
\end{tikzpicture}}
\;\; \right]
\;\;
=\;\; q^{-1}
 \left[ \;\;
  \hackcenter{\begin{tikzpicture}[scale=0.8]
    \draw[thick,<-] (0,0) .. controls ++(0,.5) and ++(0,-.5) .. (.8,1);
    \draw[thick] (.8,0) .. controls ++(0,.5) and ++(0,-.5) .. (0,1);
    \draw[thick,->] (0,1) .. controls ++(0,.5) and ++(0,-.5) .. (.8,2);
    \node at (1.15,1) {$\scs k$};;
    \draw[thick] (.8,1) .. controls ++(0,.5) and ++(0,-.5) .. (0,2)
        node[pos=.0, shape=coordinate](DOT2){};
    \node[star,star points=6,star point ratio=0.5, fill] at  (DOT2) {} ;
\end{tikzpicture}}
\;\; \right]
\\
\;\; &\refequal{\eqref{eq:ind-star}}\;\;
q^{-1}
 \left[ \;\;
  \hackcenter{\begin{tikzpicture}[scale=0.8]
    \draw[thick,<-] (0,0) .. controls ++(0,.75) and ++(0,-.45) .. (.8,1)
        node[pos=.2, shape=coordinate](DOT2){};
    \node[star,star points=6,star point ratio=0.5, fill] at  (DOT2) {} ;
    \draw[thick] (.8,0) .. controls ++(0,.75) and ++(0,-.45) .. (0,1);
    \draw[thick,->] (0,1) .. controls ++(0,.5) and ++(0,-.5) .. (.8,2);
    \node at (-.3,.3) {$\scs k$};;
    \draw[thick] (.8,1) .. controls ++(0,.5) and ++(0,-.5) .. (0,2);
\end{tikzpicture}}
\;\; \right]
\;\; -\;\;
\frac{(q-1)}{q}
\sum_{i=1}^k
 \left[ \;\;
  \hackcenter{\begin{tikzpicture}[scale=0.8]
    \draw[thick,<-] (0,0) .. controls ++(0,.45) and ++(0,.45) .. (.8,0)
        node[pos=.5, shape=coordinate](DOT2){};
    \node[star,star points=6,star point ratio=0.5, fill] at  (DOT2) {} ;
    \draw[thick] (.8,1) .. controls ++(0,-.35) and ++(0,-.35) .. (0,1);
    \draw[thick,->] (0,1) .. controls ++(0,.5) and ++(0,-.5) .. (.8,2)
           node[pos=.0, shape=coordinate](DOT){};
    \node[star,star points=6,star point ratio=0.5, fill] at  (DOT) {};
    \node at (-.45,.8) {$\scs k-i$};;
    \draw[thick] (.8,1) .. controls ++(0,.5) and ++(0,-.5) .. (0,2);
    \node at (.9,.45) {$\scs i$};;
\end{tikzpicture}} \;\; \right]
\\
\;\; &\refequal{\eqref{heis:downup}}\;\;
 \left[ \;\; \hackcenter{\begin{tikzpicture}[scale=0.8]
    \draw[thick, <-] (0,0) -- (0,1.5) node[pos=.55, shape=coordinate](DOT){};;
    \node[star,star points=6,star point ratio=0.5, fill] at  (DOT) {} ;
    \node at (-.35,.8) {$\scs k$};
    \draw[thick, <-] (.8,1.5) -- (.8,0) ;
\end{tikzpicture}}
\; \right]
\;\; - \;\;
 \left[ \;\;
  \hackcenter{\begin{tikzpicture}[scale=0.8]
    \draw[thick,<-] (0,0) .. controls ++(0,.55) and ++(0,.55) .. (.8,0)
        node[pos=.5, shape=coordinate](DOT2){};
    \node[star,star points=6,star point ratio=0.5, fill] at  (DOT2) {} ;
    \node at (.75,.6) {$\scs k$};;
    \draw[thick,<-] (.8,1.5) .. controls ++(0,-.55) and ++(0,-.55) .. (0,1.5);
\end{tikzpicture}} \;\; \right]
\;\; -\;\;
\frac{(q-1)}{q}
\sum_{i=1}^{k}
 \left[ \;\;
  \hackcenter{\begin{tikzpicture}[scale=0.8]
    \draw[thick,<-] (0,0) .. controls ++(0,.45) and ++(0,.45) .. (.8,0)
        node[pos=.5, shape=coordinate](DOT2){};
    \node[star,star points=6,star point ratio=0.5, fill] at  (DOT2) {} ;
    \draw[thick] (.8,1) .. controls ++(0,-.35) and ++(0,-.35) .. (0,1);
    \draw[thick,->] (0,1) .. controls ++(0,.5) and ++(0,-.5) .. (.8,2)
           node[pos=.0, shape=coordinate](DOT){};
    \node[star,star points=6,star point ratio=0.5, fill] at  (DOT) {};
    \node at (-.45,.7) {$\scs k-i$};;
    \draw[thick] (.8,1) .. controls ++(0,.5) and ++(0,-.5) .. (0,2);
    \node at (.8,.45) {$\scs i$};;
\end{tikzpicture}} \;\; \right]
\end{align*}
Thus, we have shown that
\[
[\wt_{1,0},\wt_{-1,k}]
\;\; = \;\;
-
\; \left[ \;\;
  \hackcenter{\begin{tikzpicture}[scale=0.8]
 \node at (1,.45)  {$\scs k $};
 \node[star,star points=6,star point ratio=0.5, fill] at  (.7,.125) {} ;
\draw  (0,0) arc (180:360:0.35cm) [thick];
\draw[->](.7,0) arc (0:180:0.35cm) [thick];
\end{tikzpicture}}
\;\;\right]
\;\; -\;\;
\frac{(q-1)}{q}
\sum_{i=1}^{k }
 \left[ \;\;
  \hackcenter{\begin{tikzpicture}[scale=0.8]
    \draw[thick] (0,2) .. controls ++(0,.45) and ++(0,.45) .. (.8,2)
        node[pos=.5, shape=coordinate](DOT2){};
    \node[star,star points=6,star point ratio=0.5, fill] at  (DOT2) {} ;
    \draw[thick] (.8,1) .. controls ++(0,-.45) and ++(0,-.45) .. (0,1);
    \draw[thick,->] (0,1) .. controls ++(0,.5) and ++(0,-.5) .. (.8,2)
           node[pos=.0, shape=coordinate](DOT){};
    \node[star,star points=6,star point ratio=0.5, fill] at  (DOT) {};
    \node at (-.55,.7) {$\scs k -i$};;
    \draw[thick, ->] (0,2).. controls ++(0,-.5) and ++(0,.5) .. (.8,1);
    \node at (.8,2.45) {$\scs i$};;
\end{tikzpicture}} \;\; \right]
\]
so that Lemma~\ref{lem:alt-w0k} completes the proof.

The second case can be proven by applying the anti-involution $\Psi$ to the case above.  Alternatively, it is an easy computation.
\begin{align*}
\wt_{-1,0}\wt_{1,k}
\;\; &=\;\;
 \left[ \;\;\;  \hackcenter{\begin{tikzpicture}[scale=0.8]
    \draw[thick, <-] (0,0) -- (0,1.5) ;
    \node at (1.15,.8) {$\scs k$};
    \draw[thick, <-] (.8,1.5) -- (.8,0) node[pos=.55, shape=coordinate](DOT){};;
    \node[star,star points=6,star point ratio=0.5, fill] at  (DOT) {} ;
\end{tikzpicture}}
\; \right]
\;\; \refequal{\eqref{heis:downup}} \;\; q^{-1}
 \left[ \;\;
  \hackcenter{\begin{tikzpicture}[scale=0.8]
    \draw[thick,<-] (0,0) .. controls ++(0,.75) and ++(0,-.45) .. (.8,1) ;
    \node at (1.1,.4) {$\scs k$};
    \draw[thick] (.8,0) .. controls ++(0,.75) and ++(0,-.45) .. (0,1)
        node[pos=.2, shape=coordinate](DOT2){};
    \node[star,star points=6,star point ratio=0.5, fill] at  (DOT2) {} ;
    \draw[thick, ->] (0,1 ) .. controls ++(0,.5) and ++(0,-.5) .. (.8,2);
    \draw[thick] (.8,1) .. controls ++(0,.5) and ++(0,-.5) .. (0,2);
\end{tikzpicture}}
\;\; \right]
\;\;
+\;\;
 \left[ \;\;
  \hackcenter{\begin{tikzpicture}[scale=0.8]
    \draw[thick,<-] (0,0) .. controls ++(0,.55) and ++(0,.55) .. (.8,0)
        node[pos=.5, shape=coordinate](DOT2){};
    \node[star,star points=6,star point ratio=0.5, fill] at  (DOT2) {} ;
    \node at (.75,.6) {$\scs k$};;
    \draw[thick,<-] (.8,1.5) .. controls ++(0,-.55) and ++(0,-.55) .. (0,1.5);
\end{tikzpicture}} \;\; \right]
\\
&\refequal{\text{trace}}
 \;\; q^{-1}
 \left[ \;\;
  \hackcenter{\begin{tikzpicture}[scale=0.8]
    \draw[thick] (0,0) .. controls ++(0,.5) and ++(0,-.5) .. (.8,1);
    \draw[thick,<-] (.8,0) .. controls ++(0,.5) and ++(0,-.5) .. (0,1);
    \draw[thick] (0,1) .. controls ++(0,.5) and ++(0,-.5) .. (.8,2);
    \node at (1.15,1) {$\scs k$};;
    \draw[thick,->] (.8,1) .. controls ++(0,.5) and ++(0,-.5) .. (0,2)
        node[pos=.0, shape=coordinate](DOT2){};
    \node[star,star points=6,star point ratio=0.5, fill] at  (DOT2) {} ;
\end{tikzpicture}}
\;\; \right]
\;\; + \;\;
\; \left[ \;\;
  \hackcenter{\begin{tikzpicture}[scale=0.8]
 \node at (1,.45)  {$\scs k$};
 \node[star,star points=6,star point ratio=0.5, fill] at  (.7,.125) {} ;
\draw  (0,0) arc (180:360:0.35cm) [thick];
\draw[->](.7,0) arc (0:180:0.35cm) [thick];
\end{tikzpicture}}
\;\;\right]
\\
&\refequal{\eqref{eq:ind-star}}
q^{-1}
 \left[ \;\;
  \hackcenter{\begin{tikzpicture}[scale=0.8]
    \draw[thick] (0,0) .. controls ++(0,.75) and ++(0,-.45) .. (.8,1)
        node[pos=.2, shape=coordinate](DOT2){};
    \node[star,star points=6,star point ratio=0.5, fill] at  (DOT2) {} ;
    \draw[thick,<-] (.8,0) .. controls ++(0,.75) and ++(0,-.45) .. (0,1);
    \draw[thick] (0,1) .. controls ++(0,.5) and ++(0,-.5) .. (.8,2);
    \node at (-.3,.3) {$\scs k$};;
    \draw[thick,->] (.8,1) .. controls ++(0,.5) and ++(0,-.5) .. (0,2);
\end{tikzpicture}}
\;\; \right]
\;\; +\;\;
\frac{(q-1)}{q}
\sum_{i=1}^k
 \left[ \;\;
  \hackcenter{\begin{tikzpicture}[scale=0.8]
    \draw[thick] (0,0) .. controls ++(0,.45) and ++(0,.45) .. (.8,0)
        node[pos=.5, shape=coordinate](DOT2){};
    \node[star,star points=6,star point ratio=0.5, fill] at  (DOT2) {} ;
    \draw[thick,<-] (.8,1) .. controls ++(0,-.35) and ++(0,-.35) .. (0,1);
    \draw[thick] (0,1) .. controls ++(0,.5) and ++(0,-.5) .. (.8,2)
           node[pos=.0, shape=coordinate](DOT){};
    \node[star,star points=6,star point ratio=0.5, fill] at  (DOT) {};
    \node at (-.35,.8) {$\scs i$};;
    \draw[thick,->] (.8,1) .. controls ++(0,.5) and ++(0,-.5) .. (0,2);
    \node at (.9,.45) {$\scs k-i$};;
\end{tikzpicture}} \;\; \right]
\;\; + \;\;
\; \left[ \;\;
  \hackcenter{\begin{tikzpicture}[scale=0.8]
 \node at (1,.45)  {$\scs k$};
 \node[star,star points=6,star point ratio=0.5, fill] at  (.7,.125) {} ;
\draw  (0,0) arc (180:360:0.35cm) [thick];
\draw[->](.7,0) arc (0:180:0.35cm) [thick];
\end{tikzpicture}}
\;\;\right]
\\
&\refequal{\eqref{heis:up down}}
 \left[ \;\; \hackcenter{\begin{tikzpicture}[scale=0.8]
    \draw[thick, ->] (0,0) -- (0,1.5) node[pos=.55, shape=coordinate](DOT){};;
    \node[star,star points=6,star point ratio=0.5, fill] at  (DOT) {} ;
    \node at (-.35,.8) {$\scs k$};
    \draw[thick, ->] (.8,1.5) -- (.8,0) ;
\end{tikzpicture}}  \;\; \right]
\;\; + \;\;
\; \left[ \;\;
  \hackcenter{\begin{tikzpicture}[scale=0.8]
 \node at (1,.45)  {$\scs k$};
 \node[star,star points=6,star point ratio=0.5, fill] at  (.7,.125) {} ;
\draw  (0,0) arc (180:360:0.35cm) [thick];
\draw[->](.7,0) arc (0:180:0.35cm) [thick];
\end{tikzpicture}}
\;\;\right]
\;\; + \;\;
\frac{(q-1)}{q}
\sum_{i=1}^{k }
 \left[ \;\;
  \hackcenter{\begin{tikzpicture}[scale=0.8]
    \draw[thick, ->] (0,2) .. controls ++(0,.45) and ++(0,.45) .. (.8,2)
        node[pos=.5, shape=coordinate](DOT2){};
    \node[star,star points=6,star point ratio=0.5, fill] at  (DOT2) {} ;
    \draw[thick,<-] (.8,1) .. controls ++(0,-.45) and ++(0,-.45) .. (0,1);
    \draw[thick] (0,1) .. controls ++(0,.5) and ++(0,-.5) .. (.8,2)
           node[pos=.0, shape=coordinate](DOT){};
    \node[star,star points=6,star point ratio=0.5, fill] at  (DOT) {};
    \node at (-.45,.7) {$\scs i$};;
    \draw[thick] (0,2).. controls ++(0,-.5) and ++(0,.5) .. (.8,1);
    \node at (.9,2.45) {$\scs k-i$};;
\end{tikzpicture}} \;\; \right]
\end{align*}

\end{proof}

%

We need the following combinatorial identities.

\begin{lemma}
\label{binomsquareconj}
For $\ell=0,1,2$ we have
$$ (\delta+1)^\ell \binom{\gamma}{\gamma}+\delta^\ell \binom{\gamma+1}{\gamma}+\cdots+1^\ell \binom{\gamma+\delta}{\gamma} =
\begin{cases}
\binom{\gamma+1+\delta}{\delta} & \text{ if } \ell=0 \\
\binom{\gamma+2+\delta}{\delta} & \text{ if } \ell=1 \\
\binom{\gamma+3+\delta}{\delta} + \binom{\gamma+2+\delta}{\delta-1} & \text{ if } \ell=2
\end{cases}$$
\end{lemma}
\begin{proof}
For a power series $f(x)$ we will denote by $[x^a] f(x)$ the coefficient of $x^a$ in $f(x)$. We will also use the standard fact that
$$[x^a] \frac{1}{(1-x)^b} = \binom{b+a-1}{b-1}.$$

Consider the case $\ell=2$. Then
$$\binom{\gamma}{\gamma+i} = [x^i] \frac{1}{(1-x)^{\gamma+1}}$$
while
$$(\delta+1-i)^2 = [x^{\delta+1-i}] \left(x \frac{d}{dx} \right)^2 \frac{1}{1-x} = [x^{\delta+1-i}] \frac{x+x^2}{(1-x)^3}.$$
Thus we find that the LHS of the equality equals
\begin{align*}
\sum_{i=0}^{i=\delta} [x^i] \frac{1}{(1-x)^{\gamma+1}} [x^{\delta+1-i}] \frac{x+x^2}{(1-x)^3}
&= [x^{\delta+1}] \frac{x+x^2}{(1-x)^{\gamma+4}} \\
&= [x^\delta] \frac{1}{(1-x)^{\gamma+4}} + [x^{\delta-1}] \frac{1}{(1-x)^{\gamma+4}} \\
&= \binom{\gamma+3+\delta}{\delta} + \binom{\gamma+2+\delta}{\delta-1}.
\end{align*}

This proves the case $\ell=2$. The cases $\ell=0,1$ are proved similarly by replacing $(x \frac{d}{dx})^2 \frac{1}{1-x}$ above with $(x \frac{d}{dx})^\ell \frac{1}{1-x}$.
\end{proof}


\begin{lemma}\label{lemma_crazyidentitylemma1}
\begin{enumerate}
\item
For $0<p<k$ we have
\begin{equation*}
\sum_{j=0}^p (-1)^{p-j} \frac{1}{2k-j} \binom{2k-j}{p} \binom{p}{j}=0.
\end{equation*}
\item For $p=k$ we have
\begin{equation*}
\sum_{j=1}^k (-1)^{j+1} \frac{2k}{k+j} \binom{k+j}{2j}\binom{2j}{j}=\sum_{j=1}^k (-1)^{j+1} \frac{2k}{k+j}\binom{k+j}{j} \binom{k}{j}=2.
\end{equation*}

\end{enumerate}
\end{lemma}

\begin{proof}

As before we use that $[x^a] \frac{1}{(1-x)^b} = \binom{b+a-1}{b-1}$.

We now prove the first part.
Replacing $j$ by $p-j$ one can rewrite the identity as
$$\sum_{j=0}^p (-1)^j \binom{2k-p+j-1}{p-1} \binom{p}{j} = 0.$$
Now, the LHS is equal to
$$(-1)^p \sum_{j=0}^p [x^{2k-2p+j}] \frac{1}{(1-x)^p} [x^{p-j}](1-x)^p.$$
This equals $(-1)^p [x^{2k-p}] (1)$ which is zero.

For the second part we will prove that
$$\sum_{j=0}^k (-1)^{j+1} \frac{2k}{k+j} \binom{k+j}{k} \binom{k}{j} = 0$$
where the sum on the LHS starts from $j=0$. This is clearly the same as the middle term equaling $2$ and it is easy to see that it is also the same as the LHS equaling $2$.

One can rewrite this as
$$\sum_{j=0}^k (-1)^{j+1} \binom{k+j-1}{k-1} \binom{k}{j} = 0.$$
Arguing as before the LHS is equal to
$$(-1)^{k+1} \sum_{j=0}^k [x^j] \frac{1}{(1-x)^k} [x^{k-j}] (1-x)^k.$$
This is equal to $(-1)^{k+1} [x^k] (1)$ which is again zero.

\end{proof}

\begin{lemma}\label{lemma_crazyidentity}
\begin{equation*}
\sum_{j=1}^k \{1\}^{2j} \left( -k \binom{k+j-1}{2j-1} + \sum_{l=1}^{k-j+1} (k-l)(l)\binom{k+j-l-2}{2j-3} \right) = -q^k + 2 - q^{-k}
\end{equation*}
\end{lemma}

\begin{proof}
The equality for the coefficients of $q^{-(k-1)}, \ldots, q^{-1},1$  follow from Lemmas ~\ref{binomsquareconj} and ~\ref{lemma_crazyidentitylemma1} after simplifying the left hand side. The equality for the coefficients of $1,q,\dots,q^{k-1}$ then follows since the left hand side is an even function of $q$. Finally, the coefficients of $q^{-k}$ and $q^k$ are easily seen to be $-1$.
\end{proof}

\begin{proposition}
\label{prob3bubblemove1}
The commutator $[\wt_{1,j}, C_k]$ equals
\begin{equation*}
\sum_{t=1}^k \sum_{l=1}^t \{1\}^{2l} \binom{l+t-1}{2l-1} C_{k-t} \wt_{1,t+j}
- \sum_{l=1}^{k+1} \{1\}^{2l} \binom{k+l}{2l-1}  \wt_{1,k+1+j}
+ \sum_{l=1}^{k} \{1\}^{2(l-1)} \binom{k+l-1}{2l-1}  \wt_{1,k+j}
\end{equation*}
\end{proposition}

\begin{proof}
This is proved by induction on $k$ using ~\eqref{eq:c-slide} and the second part of Lemma ~\ref{binomsquareconj}.
\end{proof}

\begin{proposition}
\label{prob3bubblemove2}
\begin{equation*}
[\tilde{C}_k, \wt_{1,j}]=\sum_{t=1}^{k-1} \sum_{l=1}^t  \{1\}^{2l} \binom{t+l-1}{2l-1}
\wt_{1,j+t} \tilde{C}_{k-t}.
\end{equation*}
\end{proposition}

\begin{proof}
This is proved by induction on $k$ using formula \eqref{eq:cc-slide} and the second part of Lemma ~\ref{binomsquareconj}.
\end{proof}

\begin{proposition}
\label{modifiedproblem3}
\begin{equation*}
[\tilde{w}_{1,r}, \tilde{w}_{0,k}]=
\tilde w_{1,k+r} \sum_{j=1}^k \{1\}^{2j} \left( -k \binom{k+j-1}{2j-1} + \sum_{l=1}^{k-j+1} (k-l)(l)\binom{k+j-l-2}{2j-3} \right)
\end{equation*}
\end{proposition}

\begin{proof}
By definition,
\begin{align}
[\wt_{1,r}, \wt_{0,k}] &= [\wt_{1,r}, \tilde{C}_k + \{1\}^2 \sum_{j=1}^k (k+1-j) \tilde{C}_j C_{k-j}] \nonumber \\
&= [\wt_{1,r}, \tilde{C}_k] + \{1\}^2 \sum_{j=1}^k (k+1-j)[\wt_{1,r}, \tilde{C}_j]C_{k-j}
+
\{1\}^2 \sum_{j=1}^k (k+1-j) \tilde{C}_j [\wt_{1,r}, C_{k-j}] \label{eq1modifiedproblem3}.
\end{align}
Now using Proposition ~\ref{prob3bubblemove1} and Proposition ~\ref{prob3bubblemove2} to move terms of the form $C_a$ and $\tilde{C}_b$ to the left in ~\eqref{eq1modifiedproblem3} we get

\begin{align*}
[\wt_{1,r}, \wt_{0,k}] =
& -\{1\}^2 \sum_{l=1}^{k-1} (k-l) \tilde{C}_l \wt_{1,k-l+r}
- \{1\}^4 \sum_{j=2}^k \sum_{l=1}^{j-1} (k+1-j)(j-l) \tilde{C}_l C_{k-j} \wt_{1,j-l+r} \\
&- \sum_{j=1}^k \sum_{l=1}^{j-1} \sum_{t=1}^{k-j} \sum_{p=1}^t (k+1-j)(j-l)
\binom{p+t-1}{2p-1} \{1\}^{2p+4} \tilde{C}_l C_{k-j-t} \wt_{1,t+j-l+r} \\
&+ \sum_{j=1}^k \sum_{l=1}^{j-1} \sum_{p=1}^{k-j+1} (k+1-j)(j-l)
\binom{k-j+p}{2p-1} \{1\}^{2p+4} \tilde{C}_l \wt_{1,k-l+1+r} \\
&- \sum_{j=1}^k \sum_{l=1}^{j-1} \sum_{p=1}^{k-j} (k+1-j)(j-l)
\binom{k-j+p-1}{2p-1} \{1\}^{2p+2} \tilde{C}_l \wt_{1,k-l+r} \\
&+ \sum_{j=1}^k \sum_{t=1}^{k-j} \sum_{l=1}^t (k+1-j) \binom{l+t-1}{2l-1}
\{1\}^{2l+2} \tilde{C}_j C_{k-j-t} \wt_{1,t+r} \\
&- \sum_{j=1}^k \sum_{l=1}^{k-j+1} (k+1-j) \binom{k-j+l}{2l-1}
\{1\}^{2l+2} \tilde{C}_j \wt_{1,k-j+1+r} \\
&+\sum_{j=1}^k \sum_{l=1}^{k-j} (k+1-j)\binom{k-j+l-1}{2l-1} \{ 1\}^{2l} \tilde{C}_j \wt_{1,k-j+r}.
\end{align*}
Combining like terms and using Lemma \ref{binomsquareconj} gives the proposition.
\end{proof}

\begin{corollary}\label{cor_prob3tilde} The following relations
\begin{enumerate}
\item
$[\tilde{w}_{1,r}, \tilde{w}_{0,k}]=
- \{k\}^2 \tilde{w}_{1,k+r}
$
\item
$[\tilde{w}_{-1,r}, \tilde{w}_{0,k}]=
\{k\}^2 \tilde{w}_{-1,k+r}
$
\end{enumerate}
hold in $\Tr(\H)$.
\end{corollary}

\begin{proof}
The first part follows from Proposition ~\ref{modifiedproblem3} and Lemma ~\ref{lemma_crazyidentity}.
The second item is obtained from the first part by using the involution $\Psi$.
\end{proof}

%
\subsubsection{Cross Relations for like orientations}\label{sec_crossrelations-mix}
%

Introduce the non-symmetric quantum integers
\begin{equation} \label{eq:nonsymmetric-int}
[n]_q := \frac{1-q^n}{1-q},
\end{equation}
so that if $n > 0$ then $[n]_q = 1+q+q^2 + \dots + q^{n-1}$ and
\begin{align}
 [-n]_q
 = \frac{1-q^{-n}}{1-q}
 = -\frac{1-q^n}{q^n(1-q)}
 = -\left(\frac{1}{q}+ \frac{1}{q^2} + \dots + \frac{1}{q^n} \right).
\end{align}
For $n>0$ the quantum factorials are then defined as
\begin{equation}
 [n]_q! := [n]_q [n-1]_q \dots [1]_q .
\end{equation}

We use the following idempotents in the Hecke algebra corresponding to the row and column partitions, respectively:

\begin{equation*}
e(n)=\frac{1}{[n]_q !}\sum_{w \in S_n} t_w
\hspace{.5in}
e'(n)=\frac{1}{[n]_{q^{-1}}!}\sum_{w \in S_n} (-q)^{-l(w)} t_w.
\end{equation*}

The following is an immediate consequence of the Chern character map and Theorem 4.7 of Licata-Savage~\cite{LS13} which implies the analogous results at the level of K-theory. (See Definitions \ref{def_yts} and \ref{def:si} for notation.)

\begin{proposition}\label{prop:relations} \hfill
\label{LSrelations}
\begin{enumerate}
\item $[e(n)]_{\downarrow} [e(m)]_{\downarrow} = [e(m)]_{\downarrow} [e(n)]_{\downarrow} $,
\item $[e'(n)]_{\uparrow} [e'(m)]_{\uparrow}=[e'(m)]_{\uparrow}[e'(n)]_{\uparrow}$,
\item $[e(n)]_{\downarrow}[e'(m)]_{\uparrow}=[e'(m)]_{\uparrow}[e(n)]_{\downarrow}+[e'(m-1)]_{\uparrow}[e(n-1)]_{\downarrow} $,
\item $[e(n)]_{\downarrow} [e(m)]_{\uparrow} = \sum_{k \geq 0}
[e(m-k)]_{\uparrow} [e(n-k)]_{\downarrow}$.
\end{enumerate}
\end{proposition}

We define the following formal power series:
\begin{equation*}
A(u):=\sum_{n \geq 1}  \frac{{w}_{n,0}}{n} u^n
\hspace{.5in}
B(v):=\sum_{n \geq 1}  \frac{{w}_{-n,0}}{n} v^n
\end{equation*}
We also define
\begin{equation*}
H^+(u):=1+\sum_{n \geq 1} [e(n)]_{\uparrow} u^n
\hspace{.5in}
H^-(v):=1+\sum_{n \geq 1} [e(n)]_{\downarrow} v^n.
\end{equation*}

\begin{proposition}\label{prop_mortonseries}
We have the following equality of generating series
\begin{equation*}
H^+(u)= e^{A(u)}
\hspace{.5in}
H^-(v)= e^{B(v)}.
\end{equation*}
\end{proposition}

\begin{proof}
See ~\cite[Section 4]{MM08} Morton-Manchon which relies on results found in ~\cite{Mor02}.
\end{proof}

\begin{lemma}
Consider two generating series $X(u)=X_1u+X_2u^2+\dots$ and $Y(v)=Y_1v+Y_2v^2+\dots$ where the coefficients do not necessarily commute. Then $X(u)$ and $Y(v)$ commute if and only if $\exp(X(u))$ and $\exp(Y(v))$ commute.
\end{lemma}
\begin{proof}
If $[X(u),Y(v)]=0$ then clearly $[\exp(X(u)),\exp(Y(v))]=0$. Conversely, assuming
$$[\exp(X(u)),\exp(Y(v))]=0$$
we need to show $[X_i,Y_j]=0$. We do this by induction on $i+j$. Note that
$$[u^i] \exp(X(u)) = X_i + \mbox{ terms of lower order }$$
where terms of lower order refers to products of various $X_{i'}$ with $i' < i$. Thus
$$[u^iv^j] \exp(X(u))\exp(Y(v)) = X_iY_j + \mbox{ terms of lower order }$$
and likewise
$$[u^iv^j] \exp(Y(v))\exp(X(u)) = Y_jX_i + \mbox{ terms of lower order }.$$
By induction on $i+j$ it follows that $[X_i,Y_j]=0$.
\end{proof}

\begin{corollary}
Consider two generating series $A(u)$ and $B(v)$ as above where the coefficients do not necessarily commute. Suppose that
$$\exp(B(v)) \exp(A(u)) = C(u,v) \exp(A(u)) \exp(B(v))$$
where $C(u,v)$ commutes with both $\exp(A(u))$ and $\exp(B(v))$. Then $C(u,v) = \exp(-[A(u),B(v)])$.
\end{corollary}
\begin{proof}
Since $C(u,v)$ commutes with $\exp(B(v))$ we have
$$\exp(B(v)) \exp(A(u)) \exp(-B(v)) = \exp(\ln C(u,v)+A(u)).$$
The left hand side equals
$$\exp(\exp(B(v)) A(u) \exp(-B(v)))$$
from which we get
$$\exp(B(v)) A(u) \exp(-B(v)) = \ln C(u,v)+A(u)$$
or equivalently
$$[\exp(B(v)),A(u)] = \ln C(u,v) \exp(B(v)).$$
This means that $[\exp(B(v)),A(u)]$ commutes with $B(v)$ and hence $[A(u),B(v)]$ commutes with $\exp(B(v))$. This means $\exp([A(u),B(v)])$ commutes with $\exp(B(v))$ and by the previous Corollary $[A(u),B(v)]$ commutes with $B(v)$. A similar arguments shows $[A(u),B(v)]$ also commutes with $A(u)$.

Finally, this condition implies by a standard argument (see for instance Lemma 9.43 of \cite{N4}) that
$$\exp(B(v)) \exp(A(u)) = \exp(-[A(u),B(v)]) \exp(A(u)) \exp(B(v)).$$
It follows that $C(u,v) = \exp(-[A(u),B(v)])$.
\end{proof}

\begin{proposition} \label{prop:heis-prob5}
We have the following Heisenberg relation
\begin{equation}\label{eq:heisenberg}
[{w}_{n,0}, {w}_{-m,0}]=-n \delta_{n,m}.
\end{equation}
\end{proposition}
\begin{proof}
It is easy to check that the last relation in Proposition \ref{prop:relations} implies that
$$\exp(B(v)) \exp(A(u)) = \frac{1}{1-uv} \exp(A(u)) \exp(B(v)).$$
Thus, taking $C(u,v) = \frac{1}{1-uv}$ it follows by the previous corollary that $\frac{1}{1-uv}=\exp(-[A(u),B(v)])$ and
$$[A(u),B(v)] = \ln(1-uv).$$
Since $\ln(1-uv)=\sum_{n \geq 1} \frac{(uv)^n}{n}$, this combined with Proposition ~\ref{prop_mortonseries} implies (\ref{eq:heisenberg}).
\end{proof}

%
\subsubsection{Cross Relations for mixed orientations}\label{sec_crossrelations-mix}
%

\begin{proposition} \label{prop:prob4}
For all $k\geq 0$ the equation
\[
[\wt_{\pm 1,k},\wt_{\mp 1,1}] = \mp \wt_{0,k+1}
\]
holds in $\Tr(\cal{H})$.
\end{proposition}

\begin{proof}
This is a straightforward computation using inductive star slide and the trace relations.  We prove the first case, and the second follows by applying the anti-involution $\Psi$.
\begin{align*}
\wt_{1,k}\wt_{-1,1}
\;\; &:=\;\;
 \left[ \;\; \hackcenter{\begin{tikzpicture}[scale=0.8]
    \draw[thick, ->] (0,0) -- (0,1.5) node[pos=.55, shape=coordinate](DOT){};;
    \node[star,star points=6,star point ratio=0.5, fill] at  (DOT) {} ;
    \node at (-.35,.8) {$\scs k$};
    \draw[thick, ->] (.8,1.5) -- (.8,0) node[pos=.55, shape=coordinate](DOT){};;
    \node[star,star points=6,star point ratio=0.5, fill] at  (DOT) {} ;
\end{tikzpicture}}
\; \right]
\;\;\refequal{\eqref{heis:up down}} \;\; q^{-1}
 \left[ \;\;
  \hackcenter{\begin{tikzpicture}[scale=0.8]
    \draw[thick] (0,0) .. controls ++(0,.75) and ++(0,-.45) .. (.8,1)
        node[pos=.2, shape=coordinate](DOT){};
    \node[star,star points=6,star point ratio=0.5, fill] at  (DOT) {};
    \node at (-.3,.4) {$\scs k$};
    \draw[thick,<-] (.8,0) .. controls ++(0,.75) and ++(0,-.45) .. (0,1)
        node[pos=.2, shape=coordinate](DOT2){};
    \node[star,star points=6,star point ratio=0.5, fill] at  (DOT2) {} ;
    \draw[thick] (0,1 ) .. controls ++(0,.5) and ++(0,-.5) .. (.8,2);
    \draw[thick, ->] (.8,1) .. controls ++(0,.5) and ++(0,-.5) .. (0,2);
\end{tikzpicture}}
\;\; \right]
\;\;
\refequal{\text{trace}}\;\; q^{-1}
 \left[ \;\;
  \hackcenter{\begin{tikzpicture}[scale=0.8]
    \draw[thick,<-] (0,0) .. controls ++(0,.5) and ++(0,-.5) .. (.8,1);
    \draw[thick] (.8,0) .. controls ++(0,.5) and ++(0,-.5) .. (0,1);
    \draw[thick,->] (0,1) .. controls ++(0,.5) and ++(0,-.5) .. (.8,2)
           node[pos=.0, shape=coordinate](DOT){};
    \node[star,star points=6,star point ratio=0.5, fill] at  (DOT) {};
    \node at (-.35,1) {$\scs k$};;
    \draw[thick] (.8,1) .. controls ++(0,.5) and ++(0,-.5) .. (0,2)
        node[pos=.0, shape=coordinate](DOT2){};
    \node[star,star points=6,star point ratio=0.5, fill] at  (DOT2) {} ;
\end{tikzpicture}}
\;\; \right]
\\
&\refequal{\eqref{eq:nil-star}}\;\; q^{-1}
 \left[ \;\;
  \hackcenter{\begin{tikzpicture}[scale=0.8]
    \draw[thick,<-] (0,0) .. controls ++(0,.75) and ++(0,-.45) .. (.8,1)
        node[pos=.2, shape=coordinate](DOT2){};
    \node[star,star points=6,star point ratio=0.5, fill] at  (DOT2) {} ;
    \draw[thick] (.8,0) .. controls ++(0,.75) and ++(0,-.45) .. (0,1);
    \draw[thick,->] (0,1) .. controls ++(0,.5) and ++(0,-.5) .. (.8,2)
           node[pos=.0, shape=coordinate](DOT){};
    \node[star,star points=6,star point ratio=0.5, fill] at  (DOT) {};
    \node at (-.35,1) {$\scs k$};;
    \draw[thick] (.8,1) .. controls ++(0,.5) and ++(0,-.5) .. (0,2);
\end{tikzpicture}}
\;\; \right]
\;\; - \;\;
\frac{(q-1)}{q}
 \left[ \;\;
  \hackcenter{\begin{tikzpicture}[scale=0.8]
    \draw[thick,<-] (0,0) .. controls ++(0,.45) and ++(0,.45) .. (.8,0)
        node[pos=.5, shape=coordinate](DOT2){};
    \node[star,star points=6,star point ratio=0.5, fill] at  (DOT2) {} ;
    \draw[thick] (.8,1) .. controls ++(0,-.35) and ++(0,-.35) .. (0,1);
    \draw[thick,->] (0,1) .. controls ++(0,.5) and ++(0,-.5) .. (.8,2)
           node[pos=.0, shape=coordinate](DOT){};
    \node[star,star points=6,star point ratio=0.5, fill] at  (DOT) {};
    \node at (-.35,1) {$\scs k$};;
    \draw[thick] (.8,1) .. controls ++(0,.5) and ++(0,-.5) .. (0,2);
\end{tikzpicture}} \;\; \right]
\\
\;\; &\refequal{\eqref{eq:ind-star}}\;\;
q^{-1}
 \left[ \;\;
  \hackcenter{\begin{tikzpicture}[scale=0.8]
    \draw[thick,<-] (0,0) .. controls ++(0,.75) and ++(0,-.45) .. (.8,1)
        node[pos=.2, shape=coordinate](DOT2){};
    \node[star,star points=6,star point ratio=0.5, fill] at  (DOT2) {} ;
    \draw[thick] (.8,0) .. controls ++(0,.75) and ++(0,-.45) .. (0,1)
        node[pos=.2, shape=coordinate](DOT){};
    \node[star,star points=6,star point ratio=0.5, fill] at  (DOT) {};
    \draw[thick,->] (0,1) .. controls ++(0,.5) and ++(0,-.5) .. (.8,2);
    \node at (1.1,.3) {$\scs k$};;
    \draw[thick] (.8,1) .. controls ++(0,.5) and ++(0,-.5) .. (0,2);
\end{tikzpicture}}
\;\; \right]
\;\; -\;\;
\frac{(q-1)}{q}
\sum_{i=1}^k
 \left[ \;\;
  \hackcenter{\begin{tikzpicture}[scale=0.8]
    \draw[thick,<-] (0,0) .. controls ++(0,.45) and ++(0,.45) .. (.8,0)
        node[pos=.5, shape=coordinate](DOT2){};
    \node[star,star points=6,star point ratio=0.5, fill] at  (DOT2) {} ;
    \draw[thick] (.8,1) .. controls ++(0,-.35) and ++(0,-.35) .. (0,1);
    \draw[thick,->] (0,1) .. controls ++(0,.5) and ++(0,-.5) .. (.8,2)
           node[pos=.0, shape=coordinate](DOT){};
    \node[star,star points=6,star point ratio=0.5, fill] at  (DOT) {};
    \node at (-.45,.8) {$\scs k-i$};;
    \draw[thick] (.8,1) .. controls ++(0,.5) and ++(0,-.5) .. (0,2);
    \node at (.9,.45) {$\scs i+1$};;
\end{tikzpicture}} \;\; \right]
\;\; - \;\;
\frac{(q-1)}{q}
 \left[ \;\;
  \hackcenter{\begin{tikzpicture}[scale=0.8]
    \draw[thick,<-] (0,0) .. controls ++(0,.45) and ++(0,.45) .. (.8,0)
        node[pos=.5, shape=coordinate](DOT2){};
    \node[star,star points=6,star point ratio=0.5, fill] at  (DOT2) {} ;
    \draw[thick] (.8,1) .. controls ++(0,-.35) and ++(0,-.35) .. (0,1);
    \draw[thick,->] (0,1) .. controls ++(0,.5) and ++(0,-.5) .. (.8,2)
           node[pos=.0, shape=coordinate](DOT){};
    \node[star,star points=6,star point ratio=0.5, fill] at  (DOT) {};
    \node at (-.35,1) {$\scs k$};;
    \draw[thick] (.8,1) .. controls ++(0,.5) and ++(0,-.5) .. (0,2);
\end{tikzpicture}} \;\; \right]
\\
\;\; &\refequal{\eqref{heis:downup}}\;\;
 \left[ \;\; \hackcenter{\begin{tikzpicture}[scale=0.8]
    \draw[thick, <-] (0,0) -- (0,1.5) node[pos=.55, shape=coordinate](DOT){};;
    \node[star,star points=6,star point ratio=0.5, fill] at  (DOT) {} ;
    \node at (1.15,.8) {$\scs k$};
    \draw[thick, <-] (.8,1.5) -- (.8,0) node[pos=.55, shape=coordinate](DOT){};;
    \node[star,star points=6,star point ratio=0.5, fill] at  (DOT) {} ;
\end{tikzpicture}}
\; \right]
\;\; - \;\;
 \left[ \;\;
  \hackcenter{\begin{tikzpicture}[scale=0.8]
    \draw[thick,<-] (0,0) .. controls ++(0,.55) and ++(0,.55) .. (.8,0)
        node[pos=.5, shape=coordinate](DOT2){};
    \node[star,star points=6,star point ratio=0.5, fill] at  (DOT2) {} ;
    \node at (.85,.6) {$\scs k+1$};;
    \draw[thick,<-] (.8,1.5) .. controls ++(0,-.55) and ++(0,-.55) .. (0,1.5);
\end{tikzpicture}} \;\; \right]
\;\; -\;\;
\frac{(q-1)}{q}
\sum_{i=1}^{k+1}
 \left[ \;\;
  \hackcenter{\begin{tikzpicture}[scale=0.8]
    \draw[thick,<-] (0,0) .. controls ++(0,.45) and ++(0,.45) .. (.8,0)
        node[pos=.5, shape=coordinate](DOT2){};
    \node[star,star points=6,star point ratio=0.5, fill] at  (DOT2) {} ;
    \draw[thick] (.8,1) .. controls ++(0,-.35) and ++(0,-.35) .. (0,1);
    \draw[thick,->] (0,1) .. controls ++(0,.5) and ++(0,-.5) .. (.8,2)
           node[pos=.0, shape=coordinate](DOT){};
    \node[star,star points=6,star point ratio=0.5, fill] at  (DOT) {};
    \node at (-.65,.7) {$\scs k+1-i$};;
    \draw[thick] (.8,1) .. controls ++(0,.5) and ++(0,-.5) .. (0,2);
    \node at (.8,.45) {$\scs i$};;
\end{tikzpicture}} \;\; \right]
\end{align*}
Thus, we have shown that
\[
[\wt_{1,k},\wt_{-1,1}]
\;\; = \;\;
-
\; \left[ \;\;
  \hackcenter{\begin{tikzpicture}[scale=0.8]
 \node at (1,.45)  {$\scs k+1$};
 \node[star,star points=6,star point ratio=0.5, fill] at  (.7,.125) {} ;
\draw  (0,0) arc (180:360:0.35cm) [thick];
\draw[->](.7,0) arc (0:180:0.35cm) [thick];
\end{tikzpicture}}
\;\;\right]
\;\; -\;\;
\frac{(q-1)}{q}
\sum_{i=1}^{k+1}
 \left[ \;\;
  \hackcenter{\begin{tikzpicture}[scale=0.8]
    \draw[thick] (0,2) .. controls ++(0,.45) and ++(0,.45) .. (.8,2)
        node[pos=.5, shape=coordinate](DOT2){};
    \node[star,star points=6,star point ratio=0.5, fill] at  (DOT2) {} ;
    \draw[thick] (.8,1) .. controls ++(0,-.45) and ++(0,-.45) .. (0,1);
    \draw[thick,->] (0,1) .. controls ++(0,.5) and ++(0,-.5) .. (.8,2)
           node[pos=.0, shape=coordinate](DOT){};
    \node[star,star points=6,star point ratio=0.5, fill] at  (DOT) {};
    \node at (-.65,.7) {$\scs k+1-i$};;
    \draw[thick, ->] (0,2).. controls ++(0,-.5) and ++(0,.5) .. (.8,1);
    \node at (.8,2.45) {$\scs i$};;
\end{tikzpicture}} \;\; \right]
\]
so that equation~\eqref{eq:figeight} completes the proof.
\end{proof}


Recall from Definition~\ref{def:si} that $s_i$ denotes a star on the $i$th strand (counting from left to right).  Let $\tilde{t}_{m,0}$ denote the element in $\dah_n$ whose trace is $\wt_{m,0}$.  We write $s_i\circ\wt_{m,0} = [s_i\tilde{t}_{m,0}]$.

\begin{lemma} \label{lem:wj0star}
For $m>0$ the following
\[
 s_i\circ\wt_{m,0}
 \;\; = \;\;
 m \wt_{m,1} + \frac{(q-1)}{q}\sum_{\ell=i}^{m-1} \wt_{\ell,1}\wt_{m-\ell,0}
 -\frac{(q-1)}{q} \sum_{\ell=m-i+1}^{m-1} \wt_{m-\ell,0}\wt_{\ell,1}
\]
holds in $\Tr(\cal{H})$.
\end{lemma}

\begin{proof}
To prove this identity, use the star-slide relations \eqref{eq:nil-circle-star} and \eqref{eq:nil-cool} together with the definition of the inverse crossing \eqref{eq:inv-crossing} to rewrite normal crossings in terms of circle crossings and resolution terms.
\end{proof}

The following Lemmas will be useful for proving the remaining cross relation.

\begin{lemma} \label{lem:left}
\begin{align}
\hackcenter{
\begin{tikzpicture}[scale=0.8]
    \draw  [thick](0,0) .. controls (0,.5) and (.7,.5) .. (.9,0);
    \draw  [thick](0,0) .. controls (0,-.5) and (.7,-.5) .. (.9,0);
    \draw  [thick](1,-1) .. controls (1,-.5) .. (.9,0);
    \draw  [thick,->](.9,0) .. controls (1,.5) .. (1,1) ;
      \draw  (.88,0) circle (4pt);
    \node[star,star points=6,star point ratio=0.5, fill] at  (0,0) {} ;;
    \node at  (-.3,0) {$\scs 2$} ;;
\end{tikzpicture}}
\;\; &= \;\;
  \frac{1}{(q-1)}\;\;
\hackcenter{
\begin{tikzpicture}[scale=0.8]
    \draw[thick,->]  [thick](0,-1) -- (0,1);
   \node[star,star points=6,star point ratio=0.5, fill] at  (0,0) {} ;;
\end{tikzpicture}}
\\
\hackcenter{
\begin{tikzpicture}[scale=0.8]
    \draw  [thick](0,0) .. controls (0,.5) and (.7,.5) .. (.9,0);
    \draw  [thick](0,0) .. controls (0,-.5) and (.7,-.5) .. (.9,0);
    \draw  [thick](1,-1) .. controls (1,-.5) .. (.9,0);
    \draw  [thick,->](.9,0) .. controls (1,.5) .. (1,1) ;
    \node[star,star points=6,star point ratio=0.5, fill] at  (0,0) {} ;;
    \node at  (-.3,0) {$\scs 2$} ;;
\end{tikzpicture}}
\;\; &= \;\;
  \frac{q}{(q-1)}\;\;
\hackcenter{
\begin{tikzpicture}[scale=0.8]
    \draw[thick,->]  [thick](0,-1) -- (0,1);
   \node[star,star points=6,star point ratio=0.5, fill] at  (0,0) {} ;;
\end{tikzpicture}} \;\; + \;\;  (q-1)\;
\hackcenter{
\begin{tikzpicture}[scale=0.8]
    \draw[thick,->]  [thick](1.3,-1) -- (1.3,1);
    \node[star,star points=6,star point ratio=0.5, fill] at  (.7,.125) {} ;
    \node  at  (.85,.4) {$\scs 2$} ;
\draw  (0,0) arc (180:360:0.35cm) [thick];
\draw[->](.7,0) arc (0:180:0.35cm) [thick];
\end{tikzpicture}}
\end{align}
\end{lemma}

\begin{proof}
These relations follow immediately from \eqref{eq:star-curl} using the star sliding relations \eqref{eq:nil-star} and \eqref{eq:nil-circle-star}.
\end{proof}

\begin{lemma} \label{lem:middle}
\begin{align}
\hackcenter{\begin{tikzpicture}[scale=0.8]
    \draw[thick,->] (0,0) .. controls ++(0,1) and ++(0,-1) .. (1.6,2);
    \draw[thick,->] (1.6,0) .. controls ++(0,1.8) and ++(0,.8) .. (.1,1)
        node[pos=0.5, shape=coordinate](DOT){};
    \node[star,star points=6,star point ratio=0.5, fill] at  (DOT) {};
    \draw[thick,->] (.1,1) .. controls ++(.3,-1.0) and ++(.7,-1.6) .. (.6,1.3)
        node[pos=0.47, shape=coordinate](DOT){}
           .. controls ++(-.1,.3) and ++ (0,-.3) .. (.3,2) ;
    \node[star,star points=6,star point ratio=0.5, fill] at  (DOT) {} ; ;;
    \draw  (.74,.94) circle (4pt);
    \draw  (1.19,1.25) circle (4pt);
\end{tikzpicture}}
& \;\; = \;\; \frac{q}{(q-1)}\;\;
\hackcenter{\begin{tikzpicture}[scale=0.8]
    \draw[thick, ->] (0,0) .. controls (0,.75) and (.75,.75) .. (.75,1.5);
    \draw[thick, ->] (.75,0) .. controls (.75,.75) and (0,.75) .. (0,1.5)
        node[pos=.75, shape=coordinate](DOT){};
    \node[star,star points=6,star point ratio=0.5, fill] at  (DOT) {} ;
    \draw  (.375,.75) circle (4pt);
\end{tikzpicture}}
\;\;\; - \;\;\;
\hackcenter{\begin{tikzpicture}[scale=0.8]
    \draw[thick, ->] (0,0) -- (0,1.5);;
    \draw[thick, ->] (.75,0) -- (.75,1.5)node[pos=.55, shape=coordinate](DOT){};
    \node[star,star points=6,star point ratio=0.5, fill] at  (DOT) {} ;
\end{tikzpicture}} \;\; = \;\;
\frac{q}{(q-1)} \;\;
\hackcenter{\begin{tikzpicture}[scale=0.8]
    \draw[thick, ->] (0,0) .. controls (0,.75) and (.75,.75) .. (.75,1.5);
    \draw[thick, ->] (.75,0) .. controls (.75,.75) and (0,.75) .. (0,1.5)
        node[pos=.25, shape=coordinate](DOT){};
    \node[star,star points=6,star point ratio=0.5, fill] at  (DOT) {} ;
    \draw  (.375,.75) circle (4pt);
\end{tikzpicture}}
\\
\hackcenter{\begin{tikzpicture}[scale=0.8]
    \draw[thick,->] (0,0) .. controls ++(0,1) and ++(0,-1) .. (1.6,2);
    \draw[thick,->] (1.6,0) .. controls ++(0,1.8) and ++(0,.8) .. (.1,1)
        node[pos=0.5, shape=coordinate](DOT){};
    \node[star,star points=6,star point ratio=0.5, fill] at  (DOT) {};
    \draw[thick,->] (.1,1) .. controls ++(.3,-1.0) and ++(.7,-1.6) .. (.6,1.3)
        node[pos=0.47, shape=coordinate](DOT){}
           .. controls ++(-.1,.3) and ++ (0,-.3) .. (.3,2) ;
    \node[star,star points=6,star point ratio=0.5, fill] at  (DOT) {} ; ;;
    \draw  (1.19,1.25) circle (4pt);
\end{tikzpicture}}
& \;\; = \;\; \frac{q^2}{(q-1)}\;\;
\hackcenter{\begin{tikzpicture}[scale=0.8]
    \draw[thick, ->] (0,0) .. controls (0,.75) and (.75,.75) .. (.75,1.5);
    \draw[thick, ->] (.75,0) .. controls (.75,.75) and (0,.75) .. (0,1.5)
        node[pos=.75, shape=coordinate](DOT){};
    \node[star,star points=6,star point ratio=0.5, fill] at  (DOT) {} ;
    \draw  (.375,.75) circle (4pt);
\end{tikzpicture}}
\\
\hackcenter{\begin{tikzpicture}[scale=0.8]
    \draw[thick,->] (0,0) .. controls ++(0,1) and ++(0,-1) .. (1.6,2);
    \draw[thick,->] (1.6,0) .. controls ++(0,1.8) and ++(0,.8) .. (.1,1)
        node[pos=0.5, shape=coordinate](DOT){};
    \node[star,star points=6,star point ratio=0.5, fill] at  (DOT) {};
    \draw[thick,->] (.1,1) .. controls ++(.3,-1.0) and ++(.7,-1.6) .. (.6,1.3)
        node[pos=0.47, shape=coordinate](DOT){}
           .. controls ++(-.1,.3) and ++ (0,-.3) .. (.3,2) ;
    \node[star,star points=6,star point ratio=0.5, fill] at  (DOT) {} ; ;;
\end{tikzpicture}}
& \;\; = \;\; \frac{q^2}{(q-1)}\;\;
\hackcenter{\begin{tikzpicture}[scale=0.8]
    \draw[thick, ->] (0,0) .. controls (0,.75) and (.75,.75) .. (.75,1.5);
    \draw[thick, ->] (.75,0) .. controls (.75,.75) and (0,.75) .. (0,1.5)
        node[pos=.75, shape=coordinate](DOT){};
    \node[star,star points=6,star point ratio=0.5, fill] at  (DOT) {} ;
\end{tikzpicture}}
\end{align}
\end{lemma}

\begin{proof}
This is straight-forward using star calculus relations.
\end{proof}

\begin{lemma} \label{lem:right}
\begin{align}
\hackcenter{\begin{tikzpicture}[scale=0.8]
    \draw[thick] (0,0) .. controls ++(0,1) and ++(-.2,-1) .. (1,1.5);
    \draw[thick,->] (1,1.5) .. controls ++(.3,1) and ++(.3,1) .. (.3,1.5)
     node[pos=0.15, shape=coordinate](DOT){};;
    \node[star,star points=6,star point ratio=0.5, fill] at  (DOT) {} ; ;
    \draw[thick] (.3,1.5) .. controls ++(-.2,-1) and ++(.2,-1) .. (1,.8) node[pos=0.65, shape=coordinate](DOT){};;
    \node[star,star points=6,star point ratio=0.5, fill] at  (DOT) {} ; ;
    \draw[thick,->] (1,.8) .. controls ++(-.2,1) and ++(-.1,-1) .. (.4,2.8)  ;
    \draw  (.88,1.15) circle (4pt);
\end{tikzpicture}}
&\;\; =\;\; \frac{q}{q-1} \;\; \;
\hackcenter{\begin{tikzpicture}[scale=0.8]
    \draw[thick,->] (0,0) -- (0,2.5) node[pos=0.5, shape=coordinate](DOT){};
    \node[star,star points=6,star point ratio=0.5, fill] at  (DOT) {} ; ;
\end{tikzpicture}}
\;\; - \;\; (q-1)\;\;\;
\hackcenter{\begin{tikzpicture}[scale=0.8]
    \draw[thick,->]   (-.5,-1.25) -- (-.5,1.25) ;
    \node[star,star points=6,star point ratio=0.5, fill] at  (.7,.125) {} ;
    \node  at  (.9,.4) {$\scs 2$} ;
\draw  (0,0) arc (180:360:0.35cm) [thick];
\draw[->](.7,0) arc (0:180:0.35cm) [thick];
\end{tikzpicture}}  \\
\hackcenter{\begin{tikzpicture}[scale=0.8]
    \draw[thick] (0,0) .. controls ++(0,1) and ++(-.2,-1) .. (1,1.5);
    \draw[thick,->] (1,1.5) .. controls ++(.3,1) and ++(.3,1) .. (.3,1.5)
     node[pos=0.15, shape=coordinate](DOT){};;
    \node[star,star points=6,star point ratio=0.5, fill] at  (DOT) {} ; ;
    \draw[thick] (.3,1.5) .. controls ++(-.2,-1) and ++(.2,-1) .. (1,.8) node[pos=0.65, shape=coordinate](DOT){};;
    \node[star,star points=6,star point ratio=0.5, fill] at  (DOT) {} ; ;
    \draw[thick,->] (1,.8) .. controls ++(-.2,1) and ++(-.1,-1) .. (.4,2.8)  ;
\end{tikzpicture}}
&\;\; =\;\; \frac{q^2}{q-1} \;\; \;
\hackcenter{\begin{tikzpicture}[scale=0.8]
    \draw[thick,->] (0,0) -- (0,2.5) node[pos=0.5, shape=coordinate](DOT){};
    \node[star,star points=6,star point ratio=0.5, fill] at  (DOT) {} ; ;
\end{tikzpicture}}
\end{align}
\end{lemma}

\begin{proof}
This is easy using star-calculus.
\end{proof}

\begin{proposition}\label{prop_prob6commutator}
For all $n>0$ we have
\begin{align}
 [\wt_{n,0},\wt_{-1,1}] &=
 -\left( n^2\frac{(q-1)}{q} +(n-1)^2\frac{(q-1)^3}{q^2}\right) \wt_{n-1,1} \nn \\
 & \qquad
-    \frac{(q-1)^2}{q^2}
 \sum_{\ell=1}^{n-2}(\ell+1)[\wt_{\ell,1},\wt_{n-1-\ell,0}]
-\frac{(q-1)^4}{q^3}
\left(
 \sum_{\ell=1}^{n-2} (\ell)[\wt_{\ell,1},\wt_{n-1-\ell,0}]
\right) \nn
\end{align}

\end{proposition}

\begin{proof}
Recall that the generator $\wt_{n,0}$ is comprised of both crossings and circle crossings with various factors of $q$.   As a notational convenience we introduce a new symbol
\[
\hackcenter{\begin{tikzpicture}[scale=0.8]
    \draw[thick, ->] (0,0) .. controls (0,.75) and (.75,.75) .. (.75,1.5);
    \draw[thick, ->] (.75,0) .. controls (.75,.75) and (0,.75) .. (0,1.5);
    \draw[fill=blue]  (.375,.75) circle (4pt);
\end{tikzpicture}} \in
\left\{\;\;\;
\hackcenter{\begin{tikzpicture}[scale=0.8]
    \draw[thick, ->] (0,0) .. controls (0,.75) and (.75,.75) .. (.75,1.5);
    \draw[thick, ->] (.75,0) .. controls (.75,.75) and (0,.75) .. (0,1.5);
    \draw  (.375,.75) circle (4pt);
\end{tikzpicture}} \quad, \quad
\hackcenter{\begin{tikzpicture}[scale=0.8]
    \draw[thick, ->] (0,0) .. controls (0,.75) and (.75,.75) .. (.75,1.5);
    \draw[thick, ->] (.75,0) .. controls (.75,.75) and (0,.75) .. (0,1.5);
\end{tikzpicture}} \;\;\;
\right\}
\]
to stand for either a crossing or a circle crossing.   Then we have the following:
  \[
\left[ \;\;
\hackcenter{
\begin{tikzpicture}[scale=0.8]
 \begin{scope}[yscale=-1,shift={+(0,-2)}]
  \draw[thick,<-] (3,0) .. controls (3,1.25) and (0,.25) .. (0,2);
  \draw[thick,<-] (0,0) .. controls (0,1) and (.6,.8) .. (.6,2);
  \draw[thick,<-] (.6,0) .. controls (.6,1) and (1.2,.8) .. (1.2,2);
  \draw[thick,<-] (1.8,0) .. controls (1.8,1) and (2.4,.8) .. (2.4,2);
  \draw[thick,<-] (2.4,0) .. controls (2.4,1) and (3,.8) .. (3,2);
  \node at (1.2,.35) {$\dots$};
  \node at (1.8,1.65) {$\dots$};
  \draw[fill=blue]  (.4,1.08) circle (4pt);
    \draw[fill=blue]  (.9,.94) circle (4pt);
    \draw[fill=blue]  (1.99,.72) circle (4pt);
    \draw[fill=blue]  (2.51,.57) circle (4pt);
 \end{scope}
  \draw[thick,<-] (3.6,0) -- (3.6,2) node[pos=.5, shape=coordinate](DOT){};
  \node[star,star points=6,star point ratio=0.5, fill] at  (DOT) {} ;;
\end{tikzpicture}}
\; \;\;\right]
\refequal{\eqref{heis:up down}}\;  q^{-n} \;
\left[ \;
\hackcenter{
\includegraphics{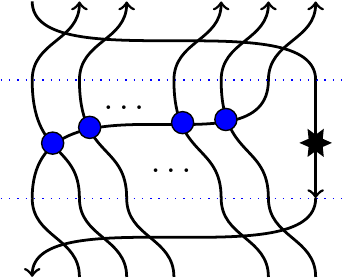}} \right]
\;\; \refequal{\eqref{heis:up-triple},\eqref{eq:circ-triple}} \;\;
q^{-n}\;
\left[ \;
\hackcenter{
    \includegraphics{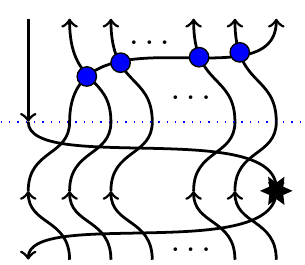}
} \right]
\]

\[
\refequal{\eqref{eq:doub-star}}
\left[ \;\;
\hackcenter{
\begin{tikzpicture}[scale=0.8]
 \begin{scope}[yscale=-1,shift={+(0,-2)}]
  \draw[thick,<-] (3,0) .. controls (3,1.25) and (0,.25) .. (0,2);
  \draw[thick,<-] (0,0) .. controls (0,1) and (.6,.8) .. (.6,2);
  \draw[thick,<-] (.6,0) .. controls (.6,1) and (1.2,.8) .. (1.2,2);
  \draw[thick,<-] (1.8,0) .. controls (1.8,1) and (2.4,.8) .. (2.4,2);
  \draw[thick,<-] (2.4,0) .. controls (2.4,1) and (3,.8) .. (3,2);
  \node at (1.2,.35) {$\dots$};
  \node at (1.8,1.65) {$\dots$};
  \draw[fill=blue]  (.4,1.08) circle (4pt);
    \draw[fill=blue]  (.9,.94) circle (4pt);
    \draw[fill=blue]  (1.99,.72) circle (4pt);
    \draw[fill=blue]  (2.51,.57) circle (4pt);
 \end{scope}
  \draw[thick,<-] (-.6,0) -- (-.6,2) node[pos=.5, shape=coordinate](DOT){};
  \node[star,star points=6,star point ratio=0.5, fill] at  (DOT) {} ;;
\end{tikzpicture}}
\; \;\;\right]
- (q-1)^2\sum_{j=0}^{n-1} q^{(-n+j)}\left[ \;\;\;
\hackcenter{
\includegraphics{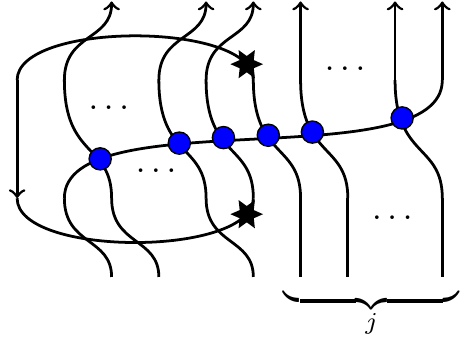}}
\;\;\; \right]
\]
Using \eqref{heis:up-triple} and \eqref{eq:circ-triple} the second term above can be written as
\[
- (q-1)^2 q^{-1} \left[ \;
\hackcenter{
\includegraphics{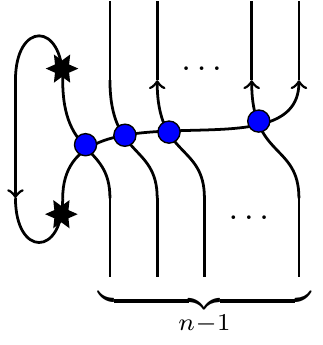}
} \;\;\right]
\;\;
- \;\;(q-1)^2\sum_{j=0}^{n-2} q^{(-n+j)}\left[ \;
\hackcenter{
\includegraphics{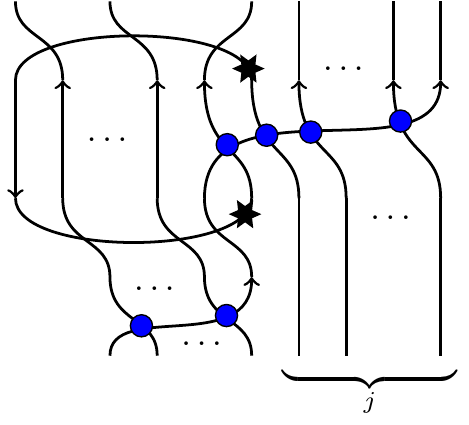}
} \;\; \right]
\]
\[
\refequal{\eqref{heis:up down}}
- (q-1)^2 q^{-1} \left[ \;\;
\hackcenter{
\includegraphics{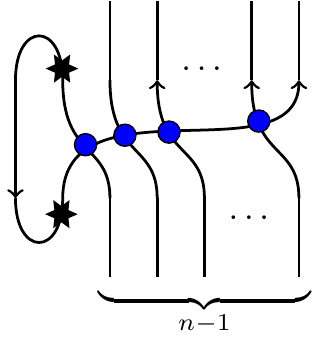}
} \;\; \right]
\;\;
- \;\;(q-1)^2q^{-2}\sum_{j=0}^{n-2}  \left[ \;\;
\hackcenter{
\includegraphics{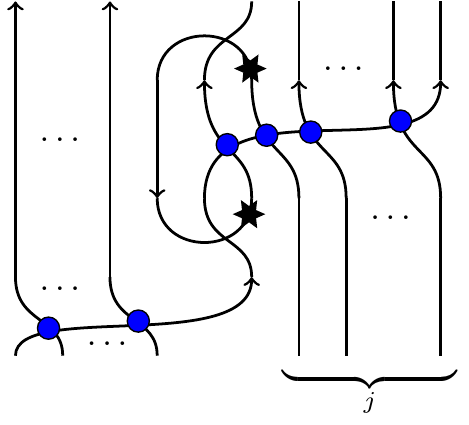}
} \;\;\right]
\]
For any choice of crossing the first term can be computed using Lemma~\ref{lem:left}.  The $j=0$ term of the sum can be simplified using Lemma~\ref{lem:right}.  The remaining terms in the $j$-summation can be simplified using Lemma~\ref{lem:middle}.

Using these facts, we can compute the commutator
\begin{align}
[\wt_{n,0},\wt_{-1,1}] &=
-\frac{(q-1)}{q} \big(\wt_{n-1,1}
+(s_1 \circ \wt_{n-1,0}) + (s_{n-1}\circ \wt_{n-1,0})
\big) \nn
\\
& \qquad
-\frac{(q-1)}{q}\left(1 + \frac{(q-1)^2}{q} \right) \sum_{i=1}^{n-1} (s_{i}\circ \wt_{n-1,0})
\end{align}
Lemma~\ref{lem:wj0star} implies that
\begin{align}
(s_1\circ \wt_{m,0}) + (s_m \circ \wt_{m,0})
& \;\; = \;\;
2m \wt_{m,1} + \frac{(q-1)}{q}
\left(\sum_{\ell=1}^{m-1} \wt_{\ell,1}\wt_{m-\ell,0}
 -  \sum_{\ell=1}^{m-1} \wt_{m-\ell,0}\wt_{\ell,1} \right) \nn \\
 & \;\; = \;\;
 2m \wt_{m,1} + \frac{(q-1)}{q}
 \sum_{\ell=1}^{m-1}[\wt_{\ell,1},\wt_{m-\ell,0}]
\end{align}
and also that
\begin{align}
 \sum_{i=1}^{m} s_i\circ\wt_{m,0}
 & \;\; = \;\;  m^2 \wt_{m,1}
 + \frac{(q-1)}{q}\sum_{i=1}^{m-1}\sum_{\ell=i}^{m-1} \wt_{\ell,1}\wt_{m-\ell,0}
 -\frac{(q-1)}{q} \sum_{i=2}^m\sum_{\ell=m-i+1}^{m-1} \wt_{m-\ell,0}\wt_{\ell,1} \nn \\
 &\;\; = \;\; m^2 \wt_{m,1}
 + \frac{(q-1)}{q}\left(\sum_{i=1}^{m-1}\sum_{\ell=i}^{m-1} \wt_{\ell,1}\wt_{m-\ell,0}
 -  \wt_{m-\ell,0}\wt_{\ell,1} \right) \\ \nn
  &\;\; = \;\; m^2 \wt_{m,1}
 + \frac{(q-1)}{q}
 \sum_{\ell=1}^{m-1} (\ell)[\wt_{\ell,1},\wt_{m-\ell,0}]
\end{align}
Hence,
\begin{align}
&[\wt_{n,0},\wt_{-1,1}] = \nn\\
&\;
-\frac{(q-1)}{q} \wt_{n-1,1}
-\frac{(q-1)}{q} \left( 2(n-1) \wt_{n-1,1} + \frac{(q-1)}{q}
 \sum_{\ell=1}^{n-2}[\wt_{\ell,1},\wt_{n-1-\ell,0}]
\right) \nn
\\
& \qquad
-\frac{(q-1)}{q}\left(1 + \frac{(q-1)^2}{q} \right)
\left(
(n-1)^2 \wt_{n-1,1}
 + \frac{(q-1)}{q}
 \sum_{\ell=1}^{n-2} (\ell)[\wt_{\ell,1},\wt_{n-1-\ell,0}]
\right)
\nn \\
 &=
 -\left( n^2\frac{(q-1)}{q} +(n-1)^2\frac{(q-1)^3}{q^2}\right) \wt_{n-1,1} \nn \\
 & \qquad
-    \frac{(q-1)^2}{q^2}
 \sum_{\ell=1}^{n-2}(\ell+1)[\wt_{\ell,1},\wt_{n-1-\ell,0}]
-\frac{(q-1)^4}{q^3}
\left(
 \sum_{\ell=1}^{n-2} (\ell)[\wt_{\ell,1},\wt_{n-1-\ell,0}]
\right). \nn
\end{align}
\end{proof}
\TL{My name is Tony, I have checked this proof.  }

\begin{corollary} \label{cor:prob6}
 We have the following relation:
 \begin{align}
 [\wt_{n,0},\wt_{-1,1}] &= -q^{-n} \frac{(q^n-1)^2}{(q-1)}  \wt_{n-1,1}.
\end{align}
\end{corollary}
\begin{proof}
By rewriting the relations in Proposition \ref{prop:uppertrianglemap} in terms of the normalizations in \eqref{eq_normalizations}, we see that for $n,\ell > 0$ we have
\begin{align}
[\wt_{\ell,1},\wt_{n-1-\ell,0}]
= q^{1-(n-1-\ell)} \frac{(q^{n-1-\ell}-1)^2}{(q-1)} \wt_{n-1,1}.
\end{align}

Then for $n > 0$ the right hand side of Proposition \ref{prop_prob6commutator} becomes the following:
\begin{align}
 [\wt_{n,0},\wt_{-1,1}] &=
 -\left( n^2\frac{(q-1)}{q} +(n-1)^2\frac{(q-1)^3}{q^2}\right) \wt_{n-1,1} \nn \\
 & \qquad
-    \frac{(q-1)^2}{q^2}
 \sum_{\ell=1}^{n-2}(\ell+1)
 q^{ -(n-2-\ell)} \frac{(q^{n-1-\ell}-1)^2}{(q-1)}  \wt_{n-1,1} \nn \\
 & \qquad
-\frac{(q-1)^4}{q^3}
\left(
 \sum_{\ell=1}^{n-2} (\ell)
  q^{ -(n-2-\ell)} \frac{(q^{n-1-\ell}-1)^2}{(q-1)}  \wt_{n-1,1}.
\right) \nn
\end{align}
This formula reduces to the following:
\begin{align}
 [\wt_{n,0},\wt_{-1,1}] &= -q^{-n} \frac{(q^n-1)^2}{(q-1)}  \wt_{n-1,1}.
\end{align}
\end{proof}

\section{An action on symmetric functions}\label{sec:symaction}

The action of $\H$ on $\oplus_n H_n\fmod$ (see \cite{LS13}) induces an action of $\Tr(\H)$ on $\oplus_n\Tr(H_n\fmod)$. We identify the latter with the ring $\sym$ of symmetric functions. By Corollary~\ref{cor:generators}, $\Tr(\H)$ is generated by the elements $w_{k,0}$ for $k \in \Z$ and by $w_{0,1}$.  We compute the action of $w_{k,0}$ in Subsection \ref{sec:wk0act} and the action of $w_{0,1}$ in Subsection \ref{sec:w01act}. This characterizes the action of $\Tr(\H)$ on $\sym$.

In \cite{SV13}, Schiffmann and Vasserot constructed an action of $\hE_{\sigma,\bar \sigma}$ on $\sym$ using Hilbert schemes. In Subsection \ref{sec_SVact} we recall this action explicitly. We also show that the $\sigma = \bar\sigma^{-1}$ specialization of the Schiffmann-Vasserot action agrees with the action of $\Tr(\H)$ from the previous paragraph after twisting one action by a graded automorphism of $\Tr(\H)$.

\subsection{The Fock space representation}

The standard embeddings $H_{n-1} \subset H_n$ define induction and restriction bimodules where
\begin{itemize}
  \item induction, denoted ${}_n\bP_{n-1}$, is given by $H_n$ considered as an $(H_n,H_{n-1})$-bimodule and
  \item restriction, denoted ${}_{n-1}\bQ_{n}$, is given by $H_n$ considered as an $(H_{n-1},H_{n})$-bimodule.
\end{itemize}

Let $\mathbb{H}_k$ denote the category whose objects are direct summands of a finite sum of tensor products of induction and restriction bimodules that start with the category of $H_k$-modules.   For instance,
\begin{equation}
 (_{k+2}\bP_{k+1}) \otimes_{H_{k+1}}
 (_{k+1}\bQ_{k+2}) \otimes_{H_{k+2}}
 (_{k+2}\bP_{k+1}) \otimes_{H_{k+1}}
  ({}_{k+1}\bP_k) \nn
\end{equation}
is a bimodule in $\mathbb{H}_k$, that we often abbreviate as $\bP\bQ\bP\bP_{k}$ with the intermediate subscripts omitted. The morphisms in $\mathbb{H}_k$ are bimodule homomorphisms.  Note that we have defined $\mathbb{H}_k$ to be Karoubi complete, so that summands of finite direct sums of the above type of bimodules are also objects in this category. Note that $\mathbb{H}_k$ is not monoidal since one cannot compose these bimodules if the domain and range do not match up.

Following our earlier notation for $\H$, for a sequence $\epsilon=\epsilon_1\epsilon_2\dots\epsilon_{\ell}$, with each $\epsilon_j\in\{+,-\}$ we set ${}_{n}(\bP_{+})_{n-1}:={}_{n}\bP_{n-1}$ and ${}_{n-1}(\bP_{-})_{n}:={}_{n-1}\bQ_{n}$ and write $(\bP_{\epsilon})_k = (\bP_{\epsilon_1}\bP_{\epsilon_2}\dots\bP_{\epsilon_{\ell}})_k$.  By convention we set $\bP_{\epsilon}$ to zero if any of the intermediate subscripts are negative.  Define
\[
|\epsilon|=\#\{ i \mid \epsilon_i =+\} - \#\{ i \mid \epsilon_i =-\}.
 \]

For each $k \geq 0$, Licata and Savage define a (non-monoidal) functor  $\cal{F}_k\maps Kar(\H) \to \mathbb{H}_k$ determined by sending $\P_{\epsilon}$ to $(\bP_{\epsilon})_k$.  For example,
\[
\cal{F}_k(\P\P\Q\P) = \bP\bP\bQ\bP_{k}.
\]
The morphisms in $\H$ are sent to specific bimodule homomorphisms. The family of functors $\cal{F}_k$ fit into a 2-categorical framework giving rise to a categorification of the Bosonic Fock space representation of the Heisenberg algebra~\cite[Definition 5.9]{LS13}.

\subsection{Trace action on symmetric functions}\label{sec:wk0act}

The family of functors $\cal{F}_n\maps \cal{H} \to \bH_n$ induce a  $\Bbbk$-linear additive functor
\begin{equation} \label{eq:AA1}
 \cal{H} \otimes \bH_k \longrightarrow \bH_k
\end{equation}
defined on objects by sending $\P_{\epsilon} \otimes (\bP_{\epsilon'})_k$ with $|\epsilon'|=m$ to
\[
(\bP_{\epsilon} \bP_{\epsilon'})_k=(\bP_{\epsilon})_{k+m} (\bP_{\epsilon'})_k = \cal{F}_{k+m}(\P_{\epsilon})(\bP_{\epsilon'})_k
 \]
 in $\bH_k$. (This has a 2-categorical interpretation in \cite{LS13}.) This functor equips $\Tr(\bH_k)$ with the structure of a module over the algebra $\Tr(\H)$.

The category $\bH_0$ is a particularly well studied object.  The objects $\bP_{\epsilon}$ with $|\epsilon|=n$ of this category are just $(H_n,H_0)$-bimodules built from tensor products of induction and restriction bimodules.  But in our conventions $H_0=\Bbbk$, so that such a bimodule is simply an $H_n$-module.  Since all finite dimensional $H_n$-modules are direct summands of sums of modules induced from the trivial representation we have that the category $\bH_0$ decomposes as a direct sum
\[
\bH_0 = \bigoplus_{n \geq 0} H_n\fmod.
\]

The additive functor \eqref{eq:AA1} defines an action of $\Tr(\H)$ on
\[
\Tr(\bH_0) = \bigoplus_{n\geq 0} \Tr(H_n\fmod).
\]
The finite Hecke algebra $H_n$ is a symmetric algebra, see for example~\cite{Geck}.  This implies that the Chern character map
\[
h_n\maps K_0(H_n\fmod) \longrightarrow \Tr(H_n\fmod)
\]
is an isomorphism.  Taking the sum of these maps over all $n$ gives an isomorphism
\[
h \maps K_0(\bH_0) \longrightarrow \Tr(\bH_0).
\]
The natural inclusions $H_{n-1} \subset H_n$ gives rise to an induction product on $K_0(\bH_0)$ by sending $[N] \in K_0(H_n)$ and $[M]\in K_0(H_m)$ to
\[
 [N]\cdot [M] := [\Ind_{n,m}^{n+m} (N \otimes M)] \in K_0(H_{n+m}).
\]

Let $z_{\lambda}$ denote the Young symmetrizer corresponding to the irreducible $H_n$-module indexed by the partition $\lambda \vdash n$ defined below in equation \eqref{eq:zlambda}.  Then $z_{\lambda}$ defines an idempotent $\P^n\to \P^n$ in $\H$.  We write $\P^{(\lambda)}:=(\P^n,z_{\lambda})$ for the corresponding object in $Kar(\H)$.  Similarly, we set $\Q^{(\lambda)}:=(\Q^n,z_{\lambda})$.  The identification $\Tr(Kar(\H)) \cong \Tr(\H)$, the isomorphism $h^{-1}\maps \Tr(\bH_0) \to K_0(\bH_0)$, and the map \eqref{eq:AA1} give rise to an action
\[
  \Tr(Kar(\H)) \otimes K_0(\bH_0)  \to K_0(\bH_0).
\]
Under this action, the class $[M]_{\cong}$ of an $H_n$-module $M$ is acted upon by $[\Id_{\P^{(\lambda)}\Q^{(\mu)}}]\in \Tr(\H)$ with $|\lambda|-|\mu|=m$ by sending it to the class of the module
\[
(\bP^{(\lambda)}\bQ^{(\mu)})_n  \otimes_{H_n} M \in \H_{n+m}\fmod,
\]
where $(\bP^{(\lambda)})_n$, respectively $(\bQ^{(\mu)})_n$, denotes the image of $\P^{(\lambda)}$ (respectively $\Q^{(\mu)}$) under the induced functors $Kar(\cal{F}_n) \maps Kar(\H) \to \bH_n$.

The action of $\Tr(Kar(\H))$ on $K_0(\bH_0)$ is closely related to the induction product on $\bH_0$.
In particular, $\cal{F}_0(\Id_{\P^{(\lambda)}})$ is the $H_n$-module homomorphism given by multiplication by the central idempotent $z_{\lambda}$.  Under the inverse of the Chern characer map the trace of this $H_n$-module homomorphism is sent to the class of the irreducible $H_n$-module $[L_{\lambda}]$ indexed by $\lambda$.   The module $(\bP^{(\lambda)})_n \otimes_{H_n} M$ is just $\Ind_{n,m}^{n+m}(L_{\lambda} \otimes M)$.

Recall the following:

\begin{theorem}[{\cite[Thm.~5.7]{WW15}}]
There is an algebra isomorphism
\[
ch_q \maps \bigoplus_{n\geq 0} K_0(H_n)\otimes_{\Z}\C(q) \to \sym
\]
that sends the class of the irreducible $L_{\lambda}$ to the Schur function $s_{\lambda}$.
\end{theorem}

Putting these results together, we have the following result.

\begin{theorem} \label{thm:Fock-action}
The Fock space 2-representation of $\H$ induces an action
\begin{equation}
\Tr(Kar(\H)) \otimes \sym \longrightarrow \sym
\end{equation}
of $\Tr(Kar(\H))$ on $\sym$ where $[\P^{(\lambda)}]$ acts by multiplication by $s_{\lambda}$ and $[\Q^{(\mu)}] \cdot 1 = 0$.
\end{theorem}

\begin{remark}
Note that $\Tr(\H)$ acts on the traces of $\bH_k$ for all $k \ge 0$. We postpone describing these representations for $k>0$ to later work.
\end{remark}

\subsection{Computing the action of $w_{0,1}$ on $\sym$}\label{sec:w01act}
In this section we compute the action of $w_{0,1}$ on $\sym$. Its action on an $H_n$-module is given by multiplying by a central element $w_n \in H_n$, which implies $w_{0,1}$ acts by a scalar on an irreducible $H_n$-module $L_\lambda$. This implies $w_{0,1}\cdot s_\lambda = c_\lambda s_\lambda$ for some constant $c_\lambda$. To compute this constant, we write the central element $w_n$ in terms of Jucys-Murphy elements and use the work of Dipper and James to compute the action of the Jucys-Murphy elements on the irreps $L_\lambda$.


\subsubsection{Irreducible representations for finite Hecke algebra}
Let $\lambda=(\lambda_1,\dots, \lambda_m)$ be a  composition of $n$. We write $\lambda'$ for the conjugate partition.   Let $S_n$ denote the symmetric group on $n$ letters and $S_{\lambda} = S_{\lambda_1} \times \dots \times S_{\lambda_m}$ be the corresponding Young subgroup.   We write $\mf{t}^{\lambda}$ for the standard $\lambda$-tableau with entries numbered sequentially along rows. Likewise, write $\mf{t}_{\lambda}$ for the standard $\lambda$-tableau with entries numbered sequentially along columns.  The symmetric group $S_n$ acts on  the right of  a $\lambda$-tableau by permuting the entries.  For each partition $\lambda$ we define an element $w_{\lambda}\in S_n$ as the unique permutation satisfying
\begin{equation}
\mf{t}_{\lambda} =\mf{t}^{\lambda}w_{\lambda}.
\end{equation}
(Note that this is equivalent to the unique $(S_{\lambda},S_{\lambda'})$ distinguished double coset representative with the trivial intersection property $w_{\lambda}^{-1}S_{\lambda}w_{\lambda} \cap W_{\lambda'} = \{1\}$.)
Define
\[
x_{\lambda} := \sum_{w \in W_{\lambda}}t_{w}, \qquad
y_{\lambda} := \sum_{w \in W_{\lambda}}(-q)^{-\ell(w)}t_w,
\]
The irreducible $H_n$-module   $L_{\lambda}$ is defined as the right ideal generated by
\begin{equation} \label{eq:zlambda}
  z_{\lambda} = x_{\lambda}t_{w_{\lambda}} y_{\lambda'}.
\end{equation}

Let $\ast$ denote the $\k(q)$-linear anti-involution of $H_n$ induced by the map $t_w \mapsto t_{w^{-1}}$.  Using this anti-involution we can convert from right modules to left modules. Setting
\begin{equation}
  z_{\lambda}^{\ast} = (x_{\lambda}t_{w_{\lambda}} y_{\lambda'})^\ast = y_{\lambda'}t_{w_{\lambda}^{-1}} x_{\lambda},
\end{equation}
we can consider the irreducible left $H_n$-module $L_{\lambda}$ as the left ideal generated by $z_{\lambda}^{\ast}$.

\subsubsection{Jucys-Murphy action on irreducibles}
The Jucys-Murphy elements $L_i=L_i^{LS}$ are $q^{-1}L_i = L_i^{DJ}$ of the operators $L_i^{DJ}$ appearing in the Dipper-James paper \cite{DJ87}.
We are following \cite{LS13} where
\begin{align}
 L_{k+1} = L_{k+1}^{LS} &= \sum_{i=1}^k q^{i-k}t_i \dots t_k \dots t_i \label{eq_JMdef1}\\
& = t_k + q^{-1}t_{k-1}t_kt_{k-1} + q^{-2}t_{k-2}t_{k-1}t_kt_{k-1}t_{k-2} + \dots
+ q^{1-k}t_1 \dots t_k \dots t_1\label{eq_JMdef2}.
\end{align}
Note that in our convention we set $L_1=L_1^{LS}=0$, but $L_1^{DJ}=1$.

We enumerate the nodes of a Young diagram $\lambda$ by $\{(i,j) \mid 1 \leq i, 1\leq j \leq \lambda_i\}$.  Following \cite{DJ87} define the residue of the $(i,j)$ node to be
$[(j-i)]_q$ using the non-symmetric quantum integer defined in \eqref{eq:nonsymmetric-int}.
For example, if $\lambda=(4,3,2)$ then the residues are given below.
%
%
\[
\begin{array}{cccc}
  0 & 1 & 1+q & 1+q+q^{2} \\
  -q^{-1} & 0 & 1 &  \\
  -q^{-1}-q^{-2}& -q^{-1} &  &
\end{array}
\;\; = \;\;
 \begin{array}{cccc}
  0    & \text{$[1]_q$}  & \text{$[2]_q$} & \text{$[3]_q$} \\
  \text{$[-1]_q$} &  0   & \text{$[1]_q$} &     \\
  \text{$[-2]_q$} & \text{$[-1]_q$} &     &
\end{array}
\]

For $w\in S_n$ and $1\leq m \leq n$, let the residue $r_{\lambda w}(m)$ of $m$ in $t^{\lambda}w$ be the residue of the node occupied by $m$ in $t^{\lambda}w$.

\begin{theorem}[Theorem 3.14 \cite{DJ87}] \label{thm:DJ}
Let $\lambda$ be a partition of $n$ and $1\leq m \leq n$.  Then $z_{\lambda}L_m^{DJ} = r_{\lambda w_{\lambda}}(m)z_{\lambda}$. In other words, $z_{\lambda}L_m^{LS} = q\cdot r_{\lambda w_{\lambda}}(m)z_{\lambda}$.
\end{theorem}

Note here that $r_{\lambda w_{\lambda}}(m)$ is the residue of $m$ in the tableau $t^{\lambda}w_{\lambda}$, but this is just the column labeled standard $\lambda$-tableau $t_{\lambda}$. So the action of the $m$th Jucys-Murphy element is given by reading off the residue of the node labeled by $m$ in the column standard $\lambda$-tableau $t_{\lambda}$.

\begin{corollary} \label{cor:DJ}
Let $\lambda$ be a partition of $n$ and $1\leq m \leq n$.  Then
 $L_m z_{\lambda}^{\ast} = q\cdot r_{\lambda w_{\lambda}}(m)z_{\lambda}^{\ast}$.
\end{corollary}

\subsubsection{The action of $w_{0,1}$ on  $H_{n+1}$}
One can compute directly from the definition of $w_{0,1}$ in Definition~\ref{def_ws} that
\begin{equation}
\begin{split}
  w_{0,1} &= \frac{1}{\{1\}}\left( \tilde{C}_1 +\frac{(q-1)^2}{q}\tilde{C}_1C_0\right)=\frac{1}{\{1\}}\left(\frac{q}{(q-1)^2} + C_0\right) =
q^{-\frac{1}{2}}\frac{q}{(q-1)}\left( \frac{q}{(q-1)^2} + C_0\right) \\
&= q^{-\frac{1}{2}}\left( \frac{q^2}{(q-1)^3} + \frac{q}{(q-1)} C_0\right).
\end{split}
\end{equation}
Recall $C_0$ is the clockwise bubble with no stars.  From \cite{LS13} we compute that $C_0$ acts on $H_{n+1}$ by multiplication by the element
\begin{align}
x
&= 1 + \sum_{i=1}^{n} q^{i-(n+1)} t_i \dots t_{n-1}t_n t_n t_{n-1} \dots t_i\label{eq_x_try2} \\
&= 1+n + \frac{(q-1)}{q}\sum_{i=1}^{n+1} L_i \label{eq_xfinal_try2}
\end{align}
where we repeatedly applied the quadratic relation in the finite Hecke algebra and used the definition of the Jucys-Murphy elements.

Therefore we have shown that $w_{0,1}$ acts on $H_{n+1}$ by

\begin{equation}\label{eq_wact}
q^{-\frac{1}{2}}\left( \frac{q^2}{(q-1)^3} + \frac{q}{(q-1)}x\right) = \frac{1}{\{1\}}\left( \frac{1}{\{1\}^2} + x\right) =
\frac 1 {\{1\}}\left( \frac 1 {\{1\}^2} + n + 1 + \frac{ q-1} q \sum_{i=1}^{n+1} L_i\right).
\end{equation}


%
%

\subsubsection{The action of $w_{0,1}$ on the irreducible representation $L_{\lambda}$ of $H_{n+1}$}

Recall that we consider $L_{\lambda}$ as the left ideal generated by $z_{\lambda}^{\ast}$.  Our goal is to compute $w_{0,1}z_{\lambda}^{\ast}$.    Recall that in Theorem~\ref{thm:DJ} and Corollary~\ref{cor:DJ} the residue sequence of $t_{\lambda}$ determines the eigenvalues of the Jucys-Murphy elements.  Because $w_{0,1}z_{\lambda}^{\ast}$ involves the action of the sum $\sum_{i=1}^{n+1}L_i$ of all Jucys-Murphy operators,  the numbering of the tableau does not matter, only the residues of all boxes appearing in $\lambda$.
Define
\begin{equation}
 {\rm res}(\lambda) = \sum_{X \in {\rm Nodes}(\lambda)}  {\rm res}(X).
\end{equation}

\begin{lemma}\label{lemma_xact}
For any partition $\lambda$ of $n+1$ we have  (adding the extra $q$ to convert from $L_i^{DJ}$ to $L_i^{LS}$)
\[
\left(\sum_{i=1}^{n+1}L_i\right)z_{\lambda}^{\ast} = q\cdot {\rm res}(\lambda)z_{\lambda}^{\ast}.
\]
\end{lemma}

\begin{proof}
This is immediate from Corollary~\ref{cor:DJ}.
\end{proof}

We can now compute the action of $x$ on $L_\lambda$:
\begin{align}
(x)z_\lambda^{\ast} &=\left( n + 1 + \frac{(q-1)}{q} \sum_{i=1}^{n+1}L_{n+1}\right)z_\lambda^{\ast}\notag \\
&= \left( n + 1 + \frac{q-1}{q} \sum_{y \in \rm{Nodes}(\lambda)} q\frac{q^{c(y)}-1}{q-1}\right)z_\lambda^{\ast}\notag\\
&= \left( \sum_{y \in \rm{Nodes}(\lambda)} q^{c(y)}\right) z_\lambda^{\ast}\label{eq_xactfinal}
\end{align}
where $c(y) := j-i$ is the \emph{content} of the box $y$, if the coordinates of $y$ are $(i,j)$.

\begin{theorem}\label{thm_tract}
We have  for any partition $\lambda$ of $n+1$ that
\begin{align}
w_{0,1}z_\lambda^{\ast} &= \frac{1}{\{1\}}\left( \frac{1}{\{1\}^2} + \sum_{y \in \lambda} q^{c(y)}\right) z_\lambda^{\ast}
\end{align}
(again where $c(y) = j-i$ if $y$ is the box with coordinates $(i,j)$).
\end{theorem}
\begin{proof}
This is equation \eqref{eq_xactfinal} combined with equation \eqref{eq_wact}.
\end{proof}

\subsection{Hilbert Schemes}\label{sec_SVact}
We would like to compare our action of $\Tr(\H)$ on $\sym$ with the action of the elliptic Hall algebra $\hE_{q,t}$ on $\sym$ constructed by Schiffmann and Vasserot using Hilbert schemes. Here we briefly recall their action from \cite{SV13} and translate this to an action of $\Tr(\H)$ using our isomorphism.

For clarity we first record the identifications we use between the parameters in \cite{SV13} and the standard $q_{Mac}$ and $t_{Mac}$ parameters for Macdonald polynomials (see \cite[Eq.~(1.6)]{SV13}).
\begin{equation}
\sigma = q^{-1}_{SV} = q^{-1}_{Mac},\quad \quad \bar \sigma = t^{-1}_{SV} = t_{Mac}
\end{equation}
These parameters are specialized to obtain the parameters in the present paper and \cite{MS14} as follows:
\begin{equation}
\sigma = q_{ours} = s^{2}_{MS},\quad \quad \bar \sigma = q^{-1}_{ours} = s^{-2}_{MS}
\end{equation}
Given a partition $\lambda$, let $P_\lambda(q_{Mac},t_{Mac}) \in \sym$ be the Macdonald polynomial of shape $\lambda$, with Macdonald's parameters $q$ and $t$. These form a basis of $\sym$ and satisfy the following properties:
\begin{equation}\label{eq_macprops}
P_\lambda(q_{Mac},t_{Mac}) = P_\lambda(q^{-1}_{Mac},t^{-1}_{Mac}),\quad \quad  P_\lambda(q_{Mac},q_{Mac}) = s_\lambda
\end{equation}

Recall the algebra $\hE_{1,\tau}$ from Section \ref{sec_specialization} (in the notation of \cite{SV13}, this is $\E_c$). This algebra is generated by elements $\bu_\bx$ for $\bx \in \Z^2$, and has parameters $\sigma$, $\bar \sigma$ (the central parameters have been specialized to $1$ and $\tau = (\sigma \bar \sigma)^{-1/2}$). Schiffmann and Vasserot defined an action of $\hE_{1,\tau}$ on $\sym$. To describe their action we first define some auxiliary operators on $\sym$. Let $p_\ell$ be the operator which multiplies by the power sum $p_\ell$, and let $\frac \partial {\partial p_\ell}$ be the corresponding differential operator. We also let $\Delta_{\pm \ell}^\infty$ be
the operator on $\sym$ defined as follows (see the line above \cite[Prop. 4.10]{SV13}):
\begin{equation}
\Delta_{\pm \ell}^\infty \cdot P_\lambda(q_{SV},t^{-1}_{SV}) :=\pm B_\lambda^{\pm \ell}(q_{SV},t_{SV}) P_\lambda(q_{SV},t_{SV}^{-1})
\end{equation}
where the polynomial $B_\lambda^{\pm \ell}$ is the following (see the first line of the proof of \cite[Prop. 4.7]{SV13}):
\[
B_\lambda^{\pm \ell}(q_{SV},t_{SV}) = \sum_{s \in \lambda} q_{SV}^{\ell x(s)}t_{SV}^{\ell y(s)} = \sum_{s \in \lambda} \sigma^{-\ell x(s)}\bar\sigma^{-\ell y(s)}.
\]
Here, if $s \in \lambda$ is a box with row-column coordinates $(i,j)$, then $x(s) = i-1$ and $y(s) = j-1$ (see \cite[Sec.~ 0.4]{SV13}).

\begin{theorem}[{\cite[Prop. 4.10]{SV13}}]
There is a unique action $\varphi$ of $\hE_{1,\tau}$ on $\sym$ such that
\begin{align*}
\varphi(\bu_{\ell,0}) &= \frac{\bar \sigma^{-\ell / 2}}{1 - \sigma^{-\ell}} p_\ell\\
\varphi(\bu_{-\ell,0}) &= -l \frac{\sigma^{-\ell/2}}{1 - \bar \sigma^{-\ell}} \frac{\partial}{\partial p_\ell}\\
\varphi(\bu_{0,\ell}) &= \Delta^\infty_{\ell} - \frac 1 {(1-\sigma^{-\ell})(1-\bar \sigma^{-\ell})}\\
\varphi(\bu_{0,-\ell}) &= \Delta^\infty_{-\ell} + \frac 1 {(1-\sigma^{\ell})(1-\bar\sigma^{\ell})}
\end{align*}
\end{theorem}


Let us now specialize these formulas to $\sigma = q_{ours}$ and $\bar \sigma = q^{-1}_{ours}$ and write them in terms of our renormalized generators $\bw_\bx := \{d(\bx)\} \bu_\bx$:
\begin{align*}
\varphi(\bw_{\ell,0}) &= q^{\ell} p_\ell\\
\varphi(\bw_{-l,0}) &= \ell q^{-\ell} \frac{\partial}{\partial p_\ell}\\
\varphi(\bw_{0,\ell}) &= \{\ell\} \Delta^\infty_{\ell} + \frac 1 {\{\ell\}}\\
\varphi(\bw_{0,-\ell}) &= \{\ell\} \Delta^\infty_{-\ell} - \frac 1 {\{\ell\}}
\end{align*}

We recall that in this specialization, the subalgebra generated by $\bw_{a,b}$ with $b \geq 0$ is called $\bE^{\bullet, \geq}$ and that it is isomorphic to $\Tr(\H)$, using the map $\bw_{a,b} \mapsto w_{a,b}$.

\begin{proposition}\label{lemma_rescale}
The map $\bw_{a,b} \mapsto q^{-a} \{1\}^{-b} \bw_{a,b}$ extends to an algebra isomorphism of $\bE^{\bullet, \geq}$. When the above representation of $\bE^{\bullet, \geq}$ is twisted by this automorphism, it is completely determined by the following formulas:
\begin{align}
\bw_{\ell,0}\cdot s_\lambda &= p_\ell \cdot s_\lambda\\
\bw_{-\ell,0} \cdot s_\lambda &= \ell \frac \partial {\partial p_\ell} s_\lambda\\
\bw_{0,\ell} \cdot s_\lambda &= \left( \frac 1 {\{1\}^\ell\{\ell\}} + \frac{\{\ell\}}{\{1\}^\ell}\sum_{y \in \lambda} q^{\ell c(y)}\right) s_\lambda
\end{align}
\end{proposition}
\begin{proof}
This first part is clear based on the presentation of $\Tr(\H)$ that we give, since it shows that $\Tr(\H)$ is $\mathbb N \times \Z$-graded. The claimed formulas for the action just follow from rescaling, from the definition of the operator $\Delta_l^\infty$, and from the two properties of Macdonald polynomials listed in \eqref{eq_macprops}.
\end{proof}

Let $\varphi': \Tr(\H) \to \bE^{\bullet, \geq} \to \End(\sym)$ be the representation of Proposition \ref{lemma_rescale} composed with the isomorphism $\Tr(\H) \to \bE^{\bullet,\geq}$ of Theorem \ref{thm_mainthm}, which sends $ w_\bx \mapsto\bw_\bx$.  Let $\varphi'': \Tr(\H) \to \End(\sym)$ be the representation discussed in Theorem \ref{thm:Fock-action}.

\begin{theorem}
The representations $\varphi'$ and $\varphi''$ are isomorphic.
\end{theorem}
\begin{proof}
First, the operator $\varphi'(w_{0,1})$ is the same as the operator $\varphi''(w_{0,1})$ by Proposition \ref{lemma_rescale} and Theorem \ref{thm_tract}.
By Lemma \ref{lemma_pkact} (below) the operators $\varphi'(w_{k,0})$ and $\varphi''(w_{k,0})$ are equal for $k > 0$, and by Lemmas \ref{lemma_wnegzero} and \ref{lemma_diff} the same is true for $k < 0$. By Corollary \ref{cor:generators}, $\Tr(\H)$ is generated as an algebra by $w_{0,1}$ and the $w_{k,0}$ for $k \in \Z$. This shows that the two representations $\varphi'$ and $\varphi''$ are equal.
\end{proof}

\begin{lemma}\label{lemma_pkact}
In the action of $\Tr(\H)$ on $\sym$ from Theorem~\ref{thm:Fock-action}, we have $w_{k,0} = p_k$ as operators, where $p_k$ acts by multiplication.
\end{lemma}
\begin{proof}
We recall from Theorem~\ref{thm:Fock-action} that we have identified $\Tr(\bH_0) \stackrel \sim \to \sym$, and that under this identification, the central primitive idempotent $z_\lambda$ corresponding to the irrep $S_\lambda$ is sent to the Schur function $s_\lambda$.

We recall a useful result of Lukac in \cite{Luk05}. There is a natural closure map $\cl_n: \Tr(Br_n) \to \Sk^+_q(ann)$ from the trace of the (group algebra of the) braid group to the (positive) skein module of the annulus, where a braid is sent to its closure in the annulus. (Here the `positive' skein module of the annulus is the subspace of $\Sk_q(ann)$ consisting of strands which only wrap counterclockwise around the annulus.) This map factors through to the quotient $\Tr(H_n) \to \Sk^+_q(ann)$. Lukac constructed a map $\alpha: \sym \stackrel \sim \to \Sk^+_q(ann)$ and showed that $\cl_n(z_\lambda) \in \Sk^+_q(ann)$ is equal to the image $\alpha(s_\lambda)$ of the Schur function \cite[Thm.~8.1]{Luk05}. (He also showed that $\alpha$ is an \emph{algebra} isomorphism \cite[Thm.~ 8.2]{Luk05}.) This shows that our map $\Tr(H_n) \to \sym$ agrees with the closure map $\Tr(H_n) \to \Sk_q^+(ann)$ composed with $\alpha^{-1}: \Sk_q^+(ann) \to \sym$. Finally, \cite[Thm.~ 3]{Mor02} (attributed there to Aiston) showed that $\alpha(p_k) = \cl_k(w_{k,0})$.
\end{proof}

\begin{lemma}\label{lemma_wnegzero}
In the action of $\Tr(\H)$ on $\sym$ from Theorem~\ref{thm:Fock-action},   we have $w_{-\ell,0}\cdot 1 = 0$.
\end{lemma}
\begin{proof}
This is immediate from the claim about the action of $\Q^{(\mu)}$ in Theorem~\ref{thm:Fock-action}.
\end{proof}

\begin{lemma}\label{lemma_diff}
Suppose $Y:\sym \to \sym$ is an operator that satisfies
\[
Y\cdot 1 = 0,\quad \quad [Y,p_k] = \delta_{\ell,k} \ell
\]
Then $Y = \ell \frac \partial {\partial p_\ell}$ as operators on $\sym$.
\end{lemma}
\begin{proof}
The action of any operator which satisfies the listed commutator relations is completely determined by its action on $1$ (since
$\sym = \Bbbk[p_1,p_2,\cdots]$ as an algebra).   Then the operators $Y$ and $\ell \frac \partial {\partial p_\ell}$ satisfy the same commutation relations and act the same way on $1$, so they are equal.
\end{proof}

%

%

\newcommand{\etalchar}[1]{$^{#1}$}
\providecommand{\bysame}{\leavevmode\hbox to3em{\hrulefill}\thinspace}
\providecommand{\MR}{\relax\ifhmode\unskip\space\fi MR }
\providecommand{\MRhref}[2]{%
  \href{http://www.ams.org/mathscinet-getitem?mr=#1}{#2}
}
\providecommand{\href}[2]{#2}

%

\end{document}